\documentclass[11pt]{amsart}
\usepackage{graphpap}
\usepackage{amssymb, color}
\usepackage{amsmath}
\usepackage{amsfonts}
\usepackage{hyperref}
\newtheorem{thm}{Theorem}
\newtheorem{lem}{Lemma}
\newtheorem{pro}{Proposition}

\theoremstyle{definition} %style des newtheorem qui suivent 
\newtheorem{rem}{Remark}

\catcode`@=11
\def\section{\@startsection{section}{1}%
  \z@{1.5\linespacing\@plus\linespacing}{.5\linespacing}%
  {\normalfont\bfseries\large\centering}}
\catcode`@=12

\def\pa{\partial}

\def\R{\mathbb R}
\def\C{\mathbb C}

\def\var{\varepsilon}

\def\be{\begin{equation}}
\def\ee{\end{equation}}

\begin{document}

\title[Blow up on a curve for NLS on Riemannian surfaces]{ Blow up on a curve for a nonlinear Schr\"odinger equation on Riemannian surfaces}
\author[N. Godet]{Nicolas Godet}
\address{University of Cergy-Pontoise, Department of Mathematics, CNRS, UMR 8088, F-95000 Cergy-Pontoise}
\email{nicolas.godet@u-cergy.fr}

\begin{abstract}
We consider the focusing quintic nonlinear Schr\"odinger equation posed on a rotationally symmetric surface, typically the sphere $S^2$ or the two dimensional hyperbolic space $H^2$. We prove the existence and the stability of solutions blowing up on a suitable curve with the log log speed. The Euclidean case is handled in \cite{Rap2006} and our result shows that the log log rate persists in other geometries with the assumption of a radial symmetry of the manifold. 
\end{abstract}

\maketitle

\section{Introduction}

We are interested in the focusing nonlinear Schr\"odinger equation 
\begin{equation} \label{nls}
i \partial_t u + \Delta u = - |u|^{4} u, \quad t \geq 0,\  x \in M,
\end{equation}
posed on a complete Riemannian surface $M$. This equation is invariant under some transformations; if $u$ is a solution of (\ref{nls}) then \\

phase transformation: $(t,x) \mapsto e^{i \theta} u(t,x)$ is also a solution for every $\theta \in \R$, \\

translation in time:  $(t,x) \mapsto  u(t+t_0,x)$ is also a solution for every $t_0 \geq 0$, \\

isometry: $ (t,x) \mapsto u(t, R(x))$ is also a solution for every isometry $R$ of $M$. 

\medskip

The equation (\ref{nls}) has some conserved quantities. The more useful are the mass and the energy and are consequences of phase and translation in time invariance: for every $t$,
\[ 
M(u(t)):=\int |u(t)|^2  = M (u(0)),
\]
\[
E(u(t)):=\frac{1}{2} \int | \nabla u(t) |^2  - \frac{1}{6} \int |u(t)|^6 = E(u(0)). 
\]
In this work, we will not use the invariance by isometries which will be useless because of the radial symmetry that we will impose to our solutions. 

\medskip

Here, we investigate the problem of existence of blow up solutions for the equation (\ref{nls}) and in particular the location of singular points and the associated blow up rate in $H^1$.

\medskip

Let us start by recalling some known facts in the Euclidean setting. In the case $M= \R^d$, 
for NLS with a $L^2$-critical or $L^2$-supercritical nonlinearity:
\begin{equation} \label{nls3}
i \partial_t u + \Delta u =- |u|^{p-1} u, \quad t \geq 0, \quad  x \in \R^d, \quad p \in \left [1+ \frac{4}{d} , 1+ \frac{4}{\mathrm{max}(d-2,0) } \right ) ,
\end{equation}
we may prove existence of blow up solutions using the sign of the Virial identity: if $u_0$ is an initial data with negative energy and finite variance, then the corresponding solution blows up in finite time. However, this argument does not describe exactly what really happens at the blow up time; in particular the location of singular points and estimate on the blow up rate are unknown. If we are interested in the blow up rate, several regimes are known at this time. The first one in the $L^2$-critical case ($p=1+4/d$) is characterized by the behavior
\[
 \| \nabla u (t) \|_{L^2(\R^n)} \sim \frac{1}{(T-t)}, \qquad \textrm{as} \quad t \to T,
\]
and is known to be unstable by perturbation of the initial data. It is the easiest to get since due to an additional symmetry in the case $p=1+4/d$, the pseudo-conformal transformation, we may explicitly give a family of solution having this behavior. Unlike the first one, the second regime is stable, and has the blow up rate:
\begin{equation} \label{loglograte}
  \| \nabla u (t) \|_{L^2(\R^n)} \sim \left ( \frac{ \log |\log (T-t)|}{ T-t} \right )^{\frac{1}{2}}, \qquad \textrm{as} \quad t \to T.
\end{equation}
It appears in the $L^2$-critical setting $p=1+4/d$ for any $d$ \cite{MerRap2003}, \cite{MerRap2006} and also in the $L^2$-supercritical case $d=2, p=5$ \cite{Rap2006}. Existence of solutions with this log log behavior for $p=1+4/d$ was first predicted by numerics and heuristic arguments \cite{Fra1985}, \cite{LanPapSulSul1988}, \cite{LemPapSulSul1988b} and proved rigorously by Perelman \cite{Per2001} with blow up in one point and stability in a subspace of $H^1$. Then Merle and Rapha\"el proved that this regime is the only one for initial data with negative energy and mass close to the mass of the ground state, thus obtaining a stability in $H^1$. They successively proved a non sharp upper bound on the blow up rate \cite{MerRap2005b} then the log log upper bound \cite{MerRap2003} and a convergence result on the rest in the decomposition of the solution \cite{MerRap2004} and then the log log lower bound \cite{MerRap2006}. As said before, these result allow to construct solutions with log log speed for an $L^2$-supercritical equation, and with an infinite set of singularities. Indeed, in \cite{Rap2006}, it is shown that equation
\[
i \partial u + \Delta u =- |u|^4 u, \quad t \geq 0, \quad x \in \R^2.
\]
is essentially driven for radial solutions by the one dimensionnal equation ((\ref{nls3}) with $p=5$ and $d=1$) for which we may apply the $L^2$-critical theory and thus gives existence of solutions whose mass concentrates on a circle. In \cite{HolRou2011}, using similar techniques, the authors show the blow up on a $2$-codimensional submanifold for the cubic NLS in $\R^3$ ($p=3$, $d=3$). More recently \cite{MerRapSze2010}, \cite{MerRapSze2012}, two other regimes appeared for slightly $L^2$-supercritical nonlinearities ($1+4/d <p< (1+4/d)(1+ \var), \ d \leq 5$). The first one gives existence of solutions with self-similar behavior proving that the lower bound on the blow up speed induced by the scaling transformation is attained:
\[
\| \nabla u(t) \|_{L^2(\R^d)} \sim \frac{C}{(T-t)^{\alpha}}, \quad \alpha =\frac{1}{2} - \frac{d}{4} + \frac{1}{p-1}.
\]
The second regime concerns the upper bound on the rate and is given by the equivalent
\[
\| \nabla u(t) \|_{L^2(\R^d)} \sim \frac{C}{(T-t)^{\beta}} , \quad \beta=\frac{(p-1)(d-1)}{(p-1)(d-1) + 5-p}.
\]

\medskip

If the equation is posed on a general manifold, with or without boundary, the situation is much less clear. For instance, the classical Virial argument does not work in all generality. However, it may be adapted in some particular manifolds like a bounded domain \cite{Ban2004}, a star shaped domain \cite{Kav1987}, the two dimensional sphere \cite{MaZha2007} or the hyperbolic space \cite{Ban2007}. In a flat geometry, typically a domain of $\R^n$ or a torus, we may localize the Euclidean constructions and prove a blow up result in the $1/t$ and log log regime \cite{BurGerTzv2003}, \cite{PlaRap2007} for the $L^2$-critical equation. In a non flat geometry, the only result concerning the blow up speed we know occurs in the setting of radial and non compact manifolds with $1/t$ blow up speed (\cite{BanCarDuy2011}); the compactness seeming to be an obstruction in this regime in particular due to bad dispersives properties.

\medskip

Here, we ask the question to know if solutions blowing up in a curve with the log log rate (\ref{loglograte}) remains in non flat geometries. The localization of the Euclidean construction in the case of a domain \cite{PlaRap2007} seems to confirm the idea of geometric stability of the log log rate. To prove our result, that is existence of a solution in the log log regime with accumulation of mass on a curve of the manifold, we will follow the approach used to treat the Euclidean case \cite{Rap2006} by imposing a radial symmetry on the curve and the solution. Thus, our geometrical setting will be that of rotationally symmetric surfaces where we will prove the blow up on a one codimensional submanifold.

\medskip

From the work \cite{BurGerTzv2004}, Strichartz estimates with loss of derivatives hold on general compact manifolds without assumption of symmetry from what we may prove a local well posedness result in $H^1$ for (\ref{nls}). Remark that in the noncompact case but with the radial symmetry condition of the manifold, methods developped in \cite{BurGerTzv2004} still apply and this handle the local theory in the energy space for (\ref{nls}). Thus, for every $u_0 \in H^1(M)$, there exists a maximal existence time $T(u_0) >0$ for the solution of (\ref{nls}) with initial data $u_0$. Moreover, we have the criteria: if an initial data $u_0$ is such that $T(u_0) < \infty$ then 
\[
  \| \nabla u (t) \|_{L^2(M)} \to \infty \qquad \textrm{as} \quad t \to T(u_0).
\] 

\medskip

For the sake of clarity, we now give a simplified but non complete statement of our result, we refer to Theorem \ref{theoreme1} for the detailed version. 

\begin{thm} \label{thm22}
 Let $(M,g)$ be a rotationally symmetric surface satisfying a suitable growth condition in the non-compact case (see (\ref{growth}) below). Then there exists an open set of initial data $\mathcal P \subset \{ u_0 \in H^2 (M), \ u_0 \ \textrm{radial} \ \} $ such that if $u_0 \in \mathcal P$, then the corresponding solution blows up on a curve of $M$ with the log log blow up rate: there exits $T(u_0) < \infty$ and $C(u_0) >0$ such that for all $t \in [0, T(u_0) )$,
\[
 \frac{1}{C(u_0) } \left ( \frac{ \log | \log (T-t) |}{ T-t}  \right ) ^{1/2}  \leq \| \nabla u (t) \|_{L^2(M)} \leq C(u_0) \left ( \frac{ \log | \log (T-t) |}{ T-t}  \right ) ^{1/2}.
\]
\end{thm} 

\medskip

Let us give some possible extensions of the result.

\begin{rem}
 A natural question is whether we can find other examples of two dimensional manifolds with blow up on a curve. In some particular cases, it is easy to construct and requires no other result than the $L^2$-critical log log theorem in 1D. Adapting the proof in \cite{PlaRap2007} where the case of a boundary is handled, we may prove a similar result (i.e. blow up in one point with log log rate) in the case $M=S^1$. This allows to prove the same result as Theorem \ref{thm22} for $M=\R \times S^1$ and $M=S^1 \times S^1$. Indeed, let $v$ be a solution of the one dimensional quintic equation in $S^1$ blowing up in a point with log log rate and define $u(t, x,y)=v(t,y)$. Then $v$ is a solution of the two dimensional quintic equation in $M$ and blows up in a circle in the log log regime.
\end{rem}

\begin{rem}
Note that if $N=M^k$ with $M$ rotationally symmetric endowed with the product metric, we may construct directly a solution on $N$ with the log log rate. Just take a solution $u$ on $M$ of (\ref{nls}) given by Theorem \ref{thm22} and set $v(t, x_1, \dots,x_k)=u(t, x_1)$.
\end{rem}

\begin{rem}
To avoid technicalities, we restrict ourselves to a two dimensional manifold but our result is still true for any dimensions $d \geq 3$ with the same method  but in this case the stability result is weaker since it occurs in $H^d_{\mathrm{rad}}$ instead of $H^2_{\mathrm{rad}}$. 
\end{rem}

\begin{rem}
Note that in this paper, manifolds we have considered are without boundary. In the Euclidean case, it has been proved in \cite{PlaRap2007} that the presence of a boundary does not change the result of blow up in one point (with log log speed) that we know in the $\R^n$ case. We may adapt this construction to prove a similar result to Theorem \ref{thm22} for instance for a disk in $\R^2$ with Dirichlet boundary conditions and blow up on a circle.
\end{rem}

\section{Preliminaries}
\subsection{Geometrical setting}
Let us now give the precise assumption on $M$ and what we call a rotationnaly symmetric manifold. We refer to \cite{GalHulLaf2004} and \cite{Pet2006} for proofs of the below results. In a first time, assume that $(M,g)$ is either $(\R^2,g)$ or $(S^2,g)$ where $g$ is a metric on $M$ such that in polar coordonates:
\begin{equation} \label{metric3}
g= dr^2+ h^2(r) g _{ S^1}, \quad r \in (0,\rho),
\end{equation}
and 
\[
 \rho= \left \{ 
\begin{array}{ccc}
 \pi & \textrm{if} & M=S^2, \\
+ \infty & \textrm{if} & M=\R^2.
\end{array}
\right .
\]
where $h: [0,\rho] \to \R^+$ is a smooth function depending only on $r$ and satisfying $h(0)=h(\rho)=0$ and $h(r) >0$ if $r \in (0,\rho)$. The smoothness of the metric at points $r=0$ and eventually $r=\rho$ imposes the following conditions on $h$:
\[
h^{(2k)} (0)=h^{(2k)}(\rho)=0, \qquad k \geq 0,
\] 
and
\[
 h'(0)=1, \qquad h'(\rho)=-1.
\]
Note that if $\rho = \infty$, all these conditions only concern the points $r=0$ and in particular $h$ does not need to vanish at infinity. Recall that polar coordonates on the sphere are defined for a point $u =(x,y,z) \in S^2 \setminus \{ (0,0,-1), (0,0,1) \} $ by  
\[
 r=d((0,0,-1),u), \qquad \theta= \mathrm{arg}( (x,y)) \in S^1.
\]
We will call rotationally symmetric surface all surface isometric to one of these two prototypes. We will make the assumption 
\begin{equation} \label{growth}
h' \leq C h
\end{equation}
for some $C>0$. This condition on the metric essentially says that the volume of a ball of radius $r$ cannot grow more than exponentially in $r$. If $M$ is compact, this assertion is automatically satisfied. In the general case, it is not restrictive and includes most of the usual cases and in particular the hyperbolic space. 

\medskip

Independently of the radial structure of $M$, we may define the Laplace-Beltrami operator as follow. For a function $f:M \to \C $, we define $\nabla f$ as the unique vector field satisfying:
\[
 \forall (x,h) \in TM, \    df(x)(h)= ( \nabla f (x) , h).
\]
In local charts, $\nabla f$ writes:
\[
(\nabla f (x))_j= \sum_i g^{ij} \partial_j f
\]
where $g^{ij}$ are the coefficients of $g^{-1}$. The divergence of a vector field $X$ of $M$ is defined for a volume form on $M$ (that we assume oriented):
\[
 L_X \omega = ( \mathrm{div} X) \omega ,
\]
where $L_X \omega$ is the Lie derivative of $\omega$ along $X$. In local coordonates
\[
\mathrm{div} X= ( \mathrm{det} g )^{- \frac{1}{2}} \sum_j \partial_j \left ( (\mathrm{det} g)^{\frac{1}{2}}   X_j \right ).
\]
The canonical measure on $M$ writes in local charts 
\begin{equation} \label{measure}
 d x: = \sqrt{ | \mathrm{det} g |} dx_1 dx_2 = h(r) dr d \theta.
 \end{equation}
The Laplace operator is defined on smooth functions by
\[
 \Delta = \mathrm{div} \circ \nabla =(\mathrm{det} g)^{- \frac{1}{2}} \sum_{i,j} \partial_i \left (g^{ij} (\mathrm{det} g)^{\frac{1}{2}} \partial_j f  \right )
\]
that with the coordonates $(r, \theta)$ becomes 
\[
\Delta = \partial_r^2 f + \frac{h'}{h} \partial_r f + \frac{1}{h^2} \partial_{\theta}^2 f.
\]
Then one extends $\Delta$ by duality to the space of distributions on $M$. 

Let us give some examples of such manifolds. The simplest is the case of the Euclidean space $\R^2$. Another example is given by compact revolution surfaces. It consists in taking a suitable pattern and applying it a rotation. More precisely, let $c =(r,z):[0,1] \to \R^{+ \ast} \times \R$ a curve parametrized by arc length and such that $c([0,1])$ is a submanifold of $\R^2$. Then the set
\[
\{ ( r(t) \mathrm{cos} \theta, r(t) \mathrm{sin} \theta, z(t) ), \ t \in [0,1], \ \theta \in \R \}
\]
is a rotationally symmetric manifold. This includes the case of the $2$ sphere. An interresting example is that of the hyperbolic space $H^2$:
\[
H^2= \{ (x,y,z) \in \R^3, \ x^2+y^2= z^2-1, \ z>0 \}
\]
with the metric $g=dx^2+dy^2-dz^2$. One may check that this example fits in our framework with $h(r)=\mathrm{sinh}( r)$ and $\rho= \infty$. 

\medskip

\subsection{A radial Sobolev embedding}

The radial symmetry will be very useful to enable us to work as if we were in a one dimensional space. In particular, Sobolev embeddings are better than the usual  two dimensional ones provided that we stay away from poles $\rho=0$ and eventually $\rho= \infty$ and that is what claim the following proposition.
\begin{pro} \label{sobolevradial}
Let $(M,g)$ be a rotationally symmetric manifold satisfying the growth condition (\ref{growth})and $\Omega$ an open radial subset of $M$ that does not contain pole i.e. of the form $\{ x \in M, r(x) \in (\eta, a- \eta) \}$ for some $\eta >0$. Let $s \in (0,1/2)$ and $p=\frac{2}{1-2s}$. Then we have the Sobolev type inequality; for all $f \in H^s_{\mathrm{r}}( M)$ \footnote{see section \ref{precise statement} for the definition of $H^s_{\mathrm{r}}(M)$} for the definition satisfying $\mathrm{ Supp} f \subset \Omega$,
\[
 \| f \|_{L^{p }(M)} \leq C \|f\|_{H^s(M)}.
\]
\end{pro}

\begin{proof}
 The expression (\ref{measure}) of the measure $dx$ on $M$, the control $h^{\frac{1}{p} } \leq h^{\frac{1}{2}} $ and the one dimensional Sobolev embedding $H^s \hookrightarrow L^p$ allows us to write for $f \in H^s(M)$ with $\mathrm{Supp} f \subset \Omega$,
\[ 
 \|f\|_{L^p(dx) }^p = C \| f h^{\frac{1}{p}}  \|_{L^p(dr)}^p \leq  C \| f h^{\frac{1}{2}}  \|_{H^s(dr)}^p.
\]
Next, we show that for all $s \in [0,1]$, there exists a constant $C>0$ such that for all $f$,
\begin{equation} \label{sobolev}
  \| h^{\frac{1}{2}} f \| _{H^s(dr)} \leq C  \|f\|_{H^s(dx)}.
\end{equation}
We only need to show (\ref{sobolev}) for $s=0$ and $s=1$. Indeed, by complex interpolation, the map 
\[
\begin{array}{rcl}
T_s: H^s(dx) & \to & H^s(dr) \\
 f & \to & h^{\frac{1}{2}} f 
\end{array}
\]
is well defined and continuous for all $s \in [0,1]$ if and only if it is continuous for $s=0$ and $s=1$. For $s=0$, (\ref{sobolev}) is straightforward using the expression of the measure on $M$ in radial coordinates. Let us show (\ref{sobolev}) for $s=1$. 
We write with the assumption $h' \leq C h$,
\begin{eqnarray*}
 \| \partial_r ( h^{\frac{1}{2}} f) \|_{L^2(dr)} &=& \| \frac{1}{2} h^{-\frac{1}{2}} \partial_r h f + h^{\frac{1}{2}} \partial_r f \| _{L^2(dr)} \\
                                                             & \leq & C \| h^{\frac{1}{2}} f \|_{L^2( dr)}+C\| h^{\frac{1}{2}} \partial_r f \|_{L^2( dr)} \\
                                                            & \leq & C\| f\|_{L^2(dx)} + C\|\nabla f\|_{L^2(dx)}= C \|f\|_{H^1(dx)}.
\end{eqnarray*}
\end{proof}

\subsection{Self-similar profiles and spectral property}

Before stating our result, we mention some notations and known results. In the whole paper, we will often use the one dimensional differential operator 
\[
\Lambda = \frac{1}{2}+ y \partial_y ,
\]
which is essentially the generator of the scaling transformation for the quintic one dimensional NLS. For $\eta >0, b \neq 0$ and small, we note $R_b=\sqrt{1- \eta} \frac{2}{|b|}$ and $R_b^-=(1- \eta) \frac{2}{|b|}$. The derivation of the exact log log law relies on the introduction of self-similar profiles which show up when looking for self similar solutions. Let us denote by $Q$ the nonlinear ground state of the one dimensional focusing $L^2$-critical Schr\"odinger equation i.e. the unique positive, radial, smooth and exponentially decaying solution to 
\[
- \partial_{y}^2 Q +Q = Q^5, \qquad y \in \R.
\]
In our setting, we know an explicit expression of $Q$:
\[
Q(y)= \frac{3^{1/4}}{\sqrt{\mathrm{cosh}(2y)}}.
\]

Following \cite{MerRap2003}, let us now introduce modified ground states, approximations of $Q$ which allow to deduce the exact log log behavior.
 
\begin{pro}[\textbf{Self similar profiles}, \cite{MerRap2003}] 
There exist $C > 0, \eta^{\ast} > 0$ such that the following holds. For all $0 < \eta < \eta^{\ast}$, there exists $b^{\ast}(\eta)>0$ such that
for all $|b| < b^{\ast}(\eta)$, there exists a unique radial solution $Q_b$ to
\[
\left \{
\begin{array}{l}
  \partial_y^2 Q_b  - Q_b +ib \Lambda Q_b + Q_b |Q_b|^4=0, \\
  P_b =Q_b e^{i \frac{b |y|^2}{4}} >0 \ \textrm{if} \  y \in [-R_b, R_b] , \\
  Q_b ( R_b)=0.
\end{array}
\right .
\]
\end{pro}
Moreover (see \cite{JohPan1993}), $Q_b$ is in $\dot{H}^1$ but not in $L^2$ since we have the behaviour at infinity 
\[
| Q_b(y) | \approx \frac{C(b)}{ |y|^{\frac{1}{2}}} \ \textrm{ as } |y| \to \infty .
\]
Thus, we need to introduce a cut-off to remove this divergence. For $b \neq 0$, let $\phi_b$ be a radial cut-off satisfying
\[
\left \{
\begin{array}{l}
  \phi_b(y)=0 \textrm{ if } |y| \geq R_b, \\
  \phi_b(y)=1 \textrm{ if } |y| \leq R_b^-, \\
  0 \leq \phi_b \leq 1, \\
  \| \partial _y \phi_b \|_{L^{\infty}} + \| \Delta \phi_b \|_{L^{\infty}} \to 0 \textrm{ as } b \to 0,
\end{array}
\right .
\]
and $\tilde{Q}_b$ the new profile defined by $\tilde{Q}_b= Q_b \phi_b$. We then have the following property.

\begin{pro}[\cite{MerRap2003}]
If we note $\tilde{Q}_b= Q_b \phi_b$, then $\tilde{Q}_b$ is a solution to 
\[ 
\partial_y^2 \tilde{Q}_b - \tilde{Q}_b+ i b \Lambda \tilde{Q}_b + \tilde{Q}_b| \tilde{Q}_b|^4 = - \Psi_b,
\]
for some remainder $\Psi_b$:
\[
 \Psi_b = -2 \partial_y \phi_b  \partial_y Q_b - Q_b \phi_b''-iby Q_b \partial_y \phi_b - ( \phi_b ^5- \phi_b) Q_b | Q_b|^5. 
\]
Moreover, we have: \\

\noindent closeness of $\tilde{Q}_b$ to $Q$: the function $\tilde{Q}_b(y) $ is derivable with respect to $b$ for all $y \in \mathbb R $  and there exists $C>0$ such that
\begin{equation} \label{closeness}
 \| e^{(1-C \eta)\frac{ \pi}{4}  |y|} ( \tilde{Q}_b - Q ) \|_{C^3} + \|e^{(1-C \eta)\frac{ \pi}{4}  |y|}  \left ( \frac{ \pa}{ \pa b} \tilde{Q}_b  + i \frac{ |y|^2}{4} Q \right )   \| _{C^3} \to 0 \textrm{ as } b \to 0, 
\end{equation}
energy estimate:
\begin{equation} \label{energyestimate}
|E(\tilde{Q}_b) | \leq e^{(1+ C \eta) \frac{ \pi}{|b|} } ,
\end{equation}
zero momentum:
\begin{equation}
 \mathrm{Im} \int \partial_y \tilde{Q}_b  \overline{\tilde{Q}_b} =0, 
\end{equation}
supercritical mass:
\begin{equation} \label{supercriticalmass}
  d_0= \frac{d}{d (b^2)} \left ( \int |\tilde{Q}_b|^2 \right )  _{|b^2=0} \in (0, \infty).
\end{equation}
\end{pro}

The following lemma introduce the radiation $\zeta_b$ which enable to describe the refined behavior of the rest $\var$ in the decomposition (\ref{decomposition}).

\begin{lem} \label{lemzeta}
 There exists $C>0, \eta^{\ast}>0$ such that for $\eta \in (0, \eta^{\ast})$, there exists $b^{\ast}(\eta)>0$ such that for $b \in (0,b^{\ast}(\eta))$, there exists a unique radial solution $\zeta_b \in \dot{H}^1(\R)$ to
\[
\partial_y^2 \zeta_b - \zeta_b +ib \Lambda \zeta_b = \Psi_b .
\]
Moreover, if we note
\begin{equation} \label{gammabinfinity}
\Gamma_b = \lim_{|y| \to \infty} |y| |\zeta_b (y)|^2,
\end{equation}
then $\Gamma_b$ is finite, strictly positive and exponentially decreasing in $b$: there exists $D >0$ such that 
\begin{equation} \label{gammabestimate}
\Gamma_b \sim \frac{D}{b} e^{- \frac{ \pi}{b}} \quad \textrm{as} \quad b \to 0;
\end{equation}
and we also have an estimate of the $\dot{H}^1$ norm of $\zeta_b$:
\begin{equation} \label{zetabnorm}
\int | \partial _y \zeta_b|^2 \leq \Gamma_b^{1-C \eta}.
\end{equation}
\end{lem}
 
\begin{proof} 
We set $Z(r)=\zeta_b(r) e^{i\frac{b r^2}{4}}$ for $r \geq 0$ so that the equation in term of $Z$ is now
\begin{equation} \label{eqonz}
 Z''(r) -Z(r)+ \frac{b ^2 r^2}{4} Z(r)= \hat{\Psi}_b(r),  \qquad r \geq 0,
\end{equation}
where $\hat{\Psi}_b(r)= \Psi_b(r) e^{i \frac{b r^2}{4}}$ with the conditions $Z'(0)=0$ and $Z (r) e^{- i \frac{b r^2}{4}}$ is in $\dot{H}^1$.  Let us now focus on the homogeneous version of (\ref{eqonz}):
\begin{equation} \label{eqonzh}
 Z''(r) -Z(r)+ \frac{b ^2 r^2}{4} Z(r)= 0.
\end{equation}
By a succession of changes of variables, we will show that this equation will be reduced to a semiclassical equation in $b$ with one turning point $2/b$. 
Thus, the problem enters in the framework of semiclassical methods but with a turning point so that the usual WKB ansatz will be no longer valid (near $2/b$). 

\medskip

 After the change of variable $x=2-br, y(x)=Z(r)$, the equation is transformed into 
\[
  b ^2 y''(x) - x h^2(x) y(x) =0, \qquad x \in (- \infty, 2],
\]
where $h$ is defined by 
\[
 h(x) =\sqrt{ 1 - \frac{x}{4}}.
\]
We next apply the change of dependant and independant variables 
\[
 s =  \left \{ 
\begin{array}{rcl}
 \left ( \frac{3}{2}  \int _0^x \sqrt{| \xi |} h(\xi)  d \xi \right )  ^{\frac{2}{3}}  & \text{if} & x \in [0,2] \\
 - \left ( -\frac{3}{2}  \int _0^x \sqrt{| \xi |} h(\xi)  d \xi \right ) ^{\frac{2}{3}}  & \text{if} & x \in ( - \infty, 0 ] \\
\end{array}
\right . ,
\]
and $w(s)=y(x) \sqrt{ s'(x)}$ to the latter equation and this gives
\[
 b^2 \frac{d ^2 w}{ds^2}- s w(s) = b^2 g(s) w(s),  \qquad s \in ( -\infty, s_0 ]
\]
where
\[
 s_0 = s(2), \qquad g(s) = \frac{s^{(3)}(x) }{ 2 (s'(x))^3} -  \frac{3 (s''(x) )^2 }{ 4(s'(x))^4} .
\]
Notice that we have the estimates on $s$ for $k \in \{ 0,1,2,3\}$:
\[
| s^{(k)}(x)| \sim C |x| ^{\frac{4}{3}  - k}, \qquad \textrm{ as } x \to - \infty,
\]
so that $g$ decays like
\[
g(s) \sim \frac{C}{s^2}, \quad \textrm{ as } s \to - \infty.
\]
Finally the scaling $s = b^{\frac{2}{3}} t$ maps the latter equation onto a nonhomogeneous Airy type equation:
\begin{equation} \label{eqony}
 Y''(t)-t Y(t) = b^{\frac{4}{3}} g(b ^{\frac{2}{3}} t) Y(t), \qquad t \in (- \infty, s_0 b ^{-2/3} ],
\end{equation}
with $Y(t)=w(s)$.
Let us denote by $\mathrm{Ai}$ the solution of the Airy equation 
\[
 Y''(t)-t Y(t)=0, \qquad t \in \R ,
\]
with the following behavior at infinity:
\begin{eqnarray*}
 \mathrm{Ai}(t) &\sim & \frac{1}{ \sqrt{ \pi}} \frac{1}{ t ^{\frac{1}{4}}}  \mathrm{exp}\left ( \frac{2}{3} t ^{\frac{3}{2}} \right )  \quad \textrm{ as }\   t \to + \infty, \\
 \mathrm{Ai}(t) &\sim & \frac{i}{ \sqrt{\pi}} \frac{1}{ (-t) ^{\frac{1}{4}}} \mathrm{exp} \left ( \frac{2i}{3} (-t) ^{\frac{3}{2}} \right ) \quad 
\textrm{ as } \ t \to - \infty.
\end{eqnarray*}
Similar expansions hold for the derivatives of $\mathrm{Ai}$. Note that with this choice, $\mathrm{Ai}$ does not vanish on the real line (see \cite{Fed1993} for details). Now we perform a fixed point argument on the equation (\ref{eqony}) to prove the existence of a solution near $\mathrm{Ai}$. We are looking for a solution of the equation (\ref{eqony}) of the form $Y(t)= \mathrm{Ai}(t) ( 1+ R(t))$ so that $R$ satisfies the integral formulation for $t \in (- \infty, s_0 b^{-2/3}]$
\begin{equation} \label{formulation}
 R(t) = \int_{ - \infty} ^t \mathrm{Ai}^2( \tau)  (h_0(\tau) - h_0(t )) \left ( b^{\frac{4}{3}} g( b^{\frac{2}{3}}\tau) (1+ R( \tau)   )\right )d \tau, 
\end{equation}
where 
\[
 h_0(\tau) = \int _{\tau} ^{\infty} \frac{d \alpha}{  \mathrm{Ai}^2( \alpha ) }.
\]
We have that the source term 
\[
S(t)= b ^{\frac{4}{3}} \int_{-\infty} ^t \mathrm{Ai}^2 (\tau) ( h_0(\tau)-h_0(t))  g(b^{\frac{2}{3}} \tau) d \tau
\]
satisfies for $b$ small and for all $k \in \{ 0,1,2 \}, t \in (- \infty,s_0 b^{-2/3} ]$,
\[
| S^{(k)}(t) | \leq  b^{1/6}, 
\]
and for all $t \in (- \infty, - 1 ]$,
\[
|S^{(k)}(t) | \leq  \frac{b ^{1/6} }{ |t| ^{1/2}}.
\]
These estimates are consequences of the asymptotic behavior for $t \leq 0$:
\[
\frac{1}{C} \frac{1}{ \langle t \rangle ^{1/4}} \leq| \mathrm{Ai}(t) | \leq C \frac{1}{ \langle t \rangle ^{1/4}},  \quad | h_0(t) | \leq C, \quad |g(t)| \leq \frac{C}{ \langle t \rangle ^2}, \quad \frac{|\mathrm{Ai}'(t)|}{ |\mathrm{Ai} (t)|} \leq C \langle t \rangle ^{1/2},
\]
and for $t \geq 0$,
\[
| \mathrm{Ai}^2(t) h_0(t) | \leq C, \quad \int_0^t |\mathrm{Ai}^2( \tau) | d \tau \leq C | \mathrm{Ai}^2(t) |, \quad  \frac{1 + t ^{1/2} } {|\mathrm{Ai}(t) ^2 |} \int_0^t \frac{ |\mathrm{Ai}(\tau) ^2 |}{ 1 + \tau ^{1/2}} d \tau \leq C,
\]
\[
|g(t)| \leq \frac{C}{\langle t \rangle ^2}, \qquad \frac{|\mathrm{Ai}'(t)|}{ |\mathrm{Ai} (t)|} \leq C \langle t \rangle ^{1/2}.
\]
Indeed, for instance, for $S(t)$, we may write for $t \leq 0$:
\begin{eqnarray*}
|S(t)| &\leq & C b^{4/3} \int_{- \infty} ^t |\mathrm{Ai} ^2 (\tau)| \frac{ d \tau}{ 1+ b^{2/3} \tau} \\
       & \leq & C b^{4/3} \int_{- \infty} ^t \frac{d \tau} { (1+ \sqrt{\tau}) (1+ b ^ {2/3} \tau)}  \\
       & \leq & C b^{4/3} \int_{- \infty} ^t \frac{d \tau} {  1+ (b ^ {4/9} \tau )^{3/2}  }  \\
       & \leq & C \frac{ b^{7/9}}{ 1+ \sqrt{|t|}} .
\end{eqnarray*}
For $t \geq 0$, we split the integral into two parts:
\begin{eqnarray*}
|S(t)| & \leq & b^{4/3} \left | \int_{- \infty} ^0 \mathrm{Ai}^ 2( \tau) ( h_0(\tau)- h_0(t)) g( b^{2/3} \tau) d \tau  \right | \\
       & & + b^{4/3} \left | \int_0 ^t  \mathrm{Ai}^2(\tau) ( h_0(\tau)- h_0(t)) g( b^{2/3} \tau) d \tau \right | \\
       & \leq & A+B.
\end{eqnarray*}
As for $t \leq 0$, the $A$ term is controlled using the behavior of $\mathrm{Ai}$ at $- \infty$:
\begin{eqnarray*}
|A| & \leq & C b^{4/3} \int_{- \infty} ^0 \frac{ d \tau}{ (1+ \sqrt{\tau})(1+ b^{2/3} \tau)} \\
    & \leq & b^{1/6}.
\end{eqnarray*}
We bound $B$ by $B \leq B_1+B_2$ and write for $t \geq 0$ using estimates mentioned above: 
\begin{eqnarray*}
B_1 & \leq & b^{4/3} \int_0^t | \mathrm{Ai}^2(\tau)| | h_0(\tau)| \frac{ d \tau}{ (1+ b ^{2/3} \tau ) ^{3/2} } \\
 & \leq & C b^{4/3} \int_0^t \frac{d \tau}{ (1+ b^{2/3} \tau )^{3/2}} \\
 & \leq & b^{1/6}.
\end{eqnarray*}
For the $B_2$ term, we have:
\begin{eqnarray*}
B_2 &\leq &b^{4/3} |h_0(t)| \int _0^t | \mathrm{Ai} ^2 (\tau) | | g(b^{2/3} \tau) | d \tau \\
    & \leq & C b^{4/3} |h_0(t)| \int_0^t |\mathrm{Ai}^2(\tau) | d \tau \\
    & \leq & b ^{1/6}.
\end{eqnarray*}
Summing the above estimates, we obtain the desired controls for $S(t)$. The derivatives of $S$ are controlled in the same way. This suggests to perform the fixed point in the metric space 
\begin{eqnarray*}
 E&=& \Big \{  u  \in \mathcal {C} ^2  (- \infty, s_0 b^{-2/3} ]  , \   \forall k \in \{ 0,1,2 \},  \forall t \in (- \infty, s_0 b^{-2/3} ] , \  |u^{(k)} (t) | \leq 2  b^{1/6}  ,  \\
 & & \quad  \forall t \in (- \infty, -1],  \ |u^{(k)}(t)| \leq 2 \frac{b ^{1/6}}{ |t| ^{1/2}}      \Big \} ,
\end{eqnarray*}
with the natural distance
\[
d(u,v)= \sum_{k=0}^2 \left ( \sup_{   t \leq  s_0 b ^{-2/3}}  ( | u^{(k)}(t) - v^{(k)}(t) | ) +  \sup _{t \leq -1} \left (  |t| ^{1/2} ( | u^{(k)}(t) - v^{(k)}(t)  | )  \right ) \right) .
\]
We set 
\[
\Phi : R \mapsto \left ( t \mapsto \int_{- \infty} ^t \mathrm{Ai} (\tau) ( h_0(\tau)- h_0(t)) \left ( b^{\frac{4}{3}} g( b^{\frac{4}{3}} \tau) ( 1+ R(\tau)) \right ) d \tau \right ) ,
\]
and remark that from the estimates on $S(t)$, $\Phi$ maps $E$ into itself and for $b$ small enough and $u,v \in E$,
\[
 d( \Phi( u) , \Phi(v) ) \leq \frac{1} {2} d(u,v).
\]
By the Banach fixed point theorem, we thus construct a solution of (\ref{eqony}) writing $ Y(t)= \mathrm{Ai} (t) ( 1 + R(t) ) $ with $R \in E$. In terms of variables $r,Z$, this gives a solution of (\ref{eqonz}) which writes
\begin{equation} \label{formulaz}
Z(r)=\frac{1}{ \sqrt{s'(2-br)}} \mathrm{Ai} \left ( \frac{s(2-br)}{b ^{2/3} } \right ) \left ( 1+ R \left ( \frac{s(2-br)} {b^{2/3}} \right )\right ).
\end{equation}
We obtain a fundamental system of solutions of (\ref{eqonz}) by adding the complex conjugate to the latter solution $Z$. Let us now define $\zeta_b$. For this, we first introduce $Z_1$ the solution satisfying $Z_1(0)=1, Z_1'(0)=0$. We set:
\[
 \tilde{Z}(r)=\left ( \int_{r}^{R_b} \frac{Z\hat{\Psi}_b } { \mathrm{Wr}(Z,Z_1) } \right ) Z_{1} (r)+ \left ( \int_{R_{b}^-} ^r \frac{Z_{1} \hat{\Psi}_b }{\mathrm{Wr}(Z,Z_1)} \right ) Z(r) ,
\]
and check that $\zeta_b$ that we define as $\zeta_b(r)=\tilde{Z}(r) \mathrm{exp}( -i b r^2 /4)$ satifies the conclusion of the lemma. Note that $\mathrm{Wr}(Z,Z_1)$ is the Wronskian of $Z$ and $Z_1$ and is constant in $r$. Let us verify that $\zeta_b$ satisfies the boundary conditions $\zeta_b'(0)=0, \zeta_b \in \dot{H}^1$ or equivalently $\tilde{Z}'(0)=0, \tilde{Z} (r) \mathrm{exp}(-ib r^2 /4) \in \dot{H}^1$. Since $\hat{\Psi}_b$ is supported in $[R_b^-, R_b]$, for $r \geq R_b$, 
\[
 \tilde{Z}(r)=Z(r) \int_{R_b^-}^{R_b} \frac{Z_{1} \hat{\Psi}_b }{\mathrm{Wr}(Z, Z_1)}.
\]
From the asymptotic properties of $\mathrm{Ai}$, there exists a real number $D \neq 0$ independant of $b$ such that
\[
 |Z(r) |^2 \sim \frac{D}{ b^{\frac{2}{3}} r}, \quad \textrm{as} \quad  r \to + \infty ,
\]
so that by setting 
\[
 \Gamma_b = \frac{D}{ | \mathrm{Wr}(Z, Z_1) | ^2 b^{\frac{2}{3}}} \left |\int_{R_b^-}^{R_b} Z_{1} \hat{\Psi}_b  \right |^2,
\]
we get 
\[
 \lim_{|y| \to \infty} |y| |\zeta_b(y)|^2 = \lim_{r \to \infty} r |\tilde{Z}(r)|^2 =\Gamma_b.
\] 
Let us prove (\ref{gammabestimate}). We evaluate the Wronskian of $Z, Z_1$ at $0$ and remark that 
\[ \frac{2}{3} (s(2))^{\frac{3}{2}} = \frac{\pi}{2}, \]
to first deduce the equivalent for some $C>0$,
\begin{equation} \label{equivalentz}
  Z'(0) \sim C b^{1/6} e^{ \frac{\pi}{2b}},
\end{equation}
and then 
\[
 |\mathrm{Wr}(Z, Z_1)| \sim C b^{\frac{1}{6}} e^{  \frac{\pi}{2b}} .
\]
To estimate the inner product between $Z_1$ and $\hat{\Psi}_b$, let us introduce the solution $J$ to $J''-J=0, J(0)=1, J'(0)=0$ i.e. $J(r)= \mathrm{cosh} (r)$ which is formally the limit of $Z_1$ when $b$ goes to $0$. Using the expression of $\tilde{\Psi}_b$ and an integration by parts, the quantity
\[
\left |  \int_{R_b-}^{R_b} Z_{1} \hat{\Psi}_b \right | = \left | \int  Z_{1} \left ( \tilde{P}_b '' - \tilde{P}_b + \frac{b ^2 r^2}{4}  \tilde{P}_b - \tilde{P}_b ^5 \right ) \right | =\left | \int  Z_{1} \tilde{P}_b ^5 \right |
\]
converges by Lebesgue theorem to $|(J, Q^5)|>0$ since $Q,J>0$ and this shows (\ref{gammabestimate}). Now, let us prove the last point:
\[
 \int \left | \partial_r \left ( \tilde{Z} (r) e^{- \frac{i b r^2}{4}}  \right ) \right  |^2 dr \leq \Gamma_b^{1- C \eta}.
\]
We write 
\begin{eqnarray*}
 \partial_r \left ( \tilde{Z} (r) e ^{- i \frac{ b r^2}{4}} \right ) &= & \left ( \int_r ^{R_b} \frac{Z \hat{\Psi}_b }{\mathrm{Wr(Z, Z_1)}}  \right ) \partial_r \left ( Z_{1} e^{-i \frac{br^2}{4}} \right ) +\left ( \int_{R_b^-} ^{r} \frac{Z_{1} \hat{\Psi}_b }{\mathrm{Wr(Z, Z_1)}}  \right ) \partial_r \left ( Z e^{-i \frac{br^2}{4}} \right ) \\
                                                                                   & =& A(r) +B(r).
\end{eqnarray*}
At this point, we need bounds on $Z_1$ before the turning point. But $Z_1$ may be written as $Z_1= 2\mathrm{Re} ( \alpha Z)$ with
\[
\alpha = \frac{\overline{Z'(0)}}{\mathrm{Wr}(Z,\overline{Z})}.
\]
Using the behavior of $\mathrm{Ai}$ and $\mathrm{Ai'}$ at $- \infty$, we get 
\[
\mathrm{Wr}(Z,\overline{Z}) = \lim_{r \to + \infty} \mathrm{Wr}(Z,\overline{Z}) (r)=  i b^{1/3} \lim_{r \to + \infty} \mathrm{ Im } \ \mathrm{Ai} \left ( \frac{s(2 - br)}{b ^{2/3}}  \right )\overline{\mathrm{Ai}'}\left ( \frac{s(2 - br)}{b ^{2/3}}  \right ) = i D b^{1/3},
\]
for some real constant $D \neq 0$. This together with (\ref{equivalentz}) and  (\ref{formulaz}) allow us to deduce for $r \leq 2/b$:
\[
 | Z_1(r) |+ |Z_1'(r)|+ |Z_1''(r)| \leq  C e^{ \frac{\pi} {2b}(1+C \eta) }.
\]
These estimates with $|\hat{\Psi}_b| \leq \mathrm{exp}(-(1-C \eta) \pi / 2b )$ provide
\begin{eqnarray*}
 \int |A(r) |^2 dr &\leq& e^{ - \frac{2 \pi}{b}(1-C \eta) }   \left ( \int _{ 0}^{R_b} |  \partial_r Z_1 | ^2 + \int 
_{r \leq R_b} b^2 r^2 | Z_{1} |^2 \right )  \\
                          & \leq &  e^{ - \frac{2 \pi}{b}(1-C \eta) } \left(  \frac{ 4}{b^2}  e ^{ \frac{\pi}{b}(1+C \eta)} +4  e ^{ \frac{\pi}{b}(1+C \eta)}  \right )  \\
                         & \leq & \Gamma_b^{1- C \eta}.
\end{eqnarray*}
On the other hand,
\begin{eqnarray*}
  \int |B(r) |^2 dr &\leq&  e^{ -\frac{\pi}{b}(1-C \eta)} \int_{ R_b^-}^{\infty}  \left | \partial _r Z- \frac{ibr}{2} Z \right |^2 
\end{eqnarray*}
Computing with (\ref{formulaz}), we show that the above integral is finite with 
\[
 \left  |  \partial_r Z(r) - \frac{ibr}{2}  Z(r) \right | \leq C \frac{b^{\alpha} }{ r^{3/2}}
\]
for some $\alpha \in \R$ so that we obtain
\[
 \int | B(r)|^2 \leq \Gamma_b^{1-C \eta} 
\]
and this proves (\ref{zetabnorm}). For the uniqueness of $\zeta_b$, taking a generic solution $\hat{Z}$ of (\ref{eqonz}), writing down the variation of constants formula for this solution and imposing the boundary conditions $\hat{Z}'(0)=0, \hat {Z}(r) \mathrm{exp} ( -i b r^2 /4) \in \dot{H}^1 $, we easily obtain $\hat{Z} = \tilde{Z}$.

\end{proof}

\begin{rem}
 In our case, we will only need the one dimensional version of the lemma since from the assumption of radial symmetry, the problem will be reduced to a one dimensional equation. For studying problems without assumption of symmetry, we need to prove a multidimensional result; this is sketched in \cite{MerRap2004}.  Here, we have only detailed the one dimensional result. Note that if the dimension $d$ is equal to $2$, the equation 
\[
 Z''(r) + \frac{d-1}{r} Z'(r) - Z(r) + \frac{b^2 r^2}{4} Z(r)=0,
\]
which is the $d$-dimensional analogue of (\ref{eqonzh}), has two turning points of order $1/2$ and $2/b$ when $b$ tends to $0$. If the dimension is bigger than $2$, the equation is similar to the one dimensional one with only one turning point of order $2/b$. These 
facts may be seen by looking at the equation for $W(r)=r^{(d-1)/2}Z(r)$:
\[
 W''(r) +q(r) W(r)=0, \quad \textrm{where} \quad q(r)= \frac{b^2 r^2}{4}-1 + \left ( \frac{d-1}{2} - \frac{(d-1)^2 }{4} \right ) \frac{1}{r^2},
\]
and the study of the sign and vanishing points of $q(r)$.
\end{rem}

\begin{pro}[\textbf{Spectral property}, \cite{MerRap2005b} ] \label{spectral}
Let $\mathcal L_1$ and $\mathcal L _2$ be the one dimensional Schr\"odinger operators defined by 
\[
\mathcal L_1 = - \partial ^2_y +10 y Q^3 \partial_y Q, \qquad \mathcal L_2 = -\partial_y^2 +2 y Q^3  \partial_y Q,
\]
and $H$ the quadratic form defined for $\var \in H^1$ ,
\[
H (\var, \var)= ( \mathcal L_1 \var_1, \var_1) + ( \mathcal L_2 \var_2, \var_2).
\]
Then there exists a contant $\delta >0$ such that we have the property
\begin{eqnarray*}
\forall \var \in H^1( \R), \ \  H(\var, \var) & \geq & \delta    \left ( \int | \partial_y \var (y)  |^2 dy + \int | \var |^2 e^{- |y|} dy\right ) -\frac{1}{\delta}  \Big (  (\var_1, Q)^2 +
(\var_1, y^2 Q)^2 \\
 & & +(\var_1, yQ)^2+(\var_2, \Lambda Q)^2+(\var_2, \Lambda^2 Q)^2 + (\var_2, \partial_y Q)^2 \Big ).
\end{eqnarray*}
\end{pro}

\begin{rem}
We may claim a similar property in all dimensions but this is proved only in the case of the dimension $1$ using the explicit formula of the ground state $Q$. In this paper, we will only use the one dimensional version.
\end{rem}

\section{Precise statement of the result} \label{precise statement}

\subsection{Statement of the Theorem}

In the sequel, we will often see a radial function $f(x)$ on $M$ as a function $f(r)$ on $[0, \rho)$; in terms of Sobolev spaces, this representation is not isometric and the function $h$ measures this fact. We introduce the Sobolev spaces for $s >0$~:
\begin{eqnarray*}
H^s_{\mathrm{r}}& = &\left \{  f \in L^2 (M) , \  f \ \textrm{radial}, \  f \in \mathrm{Domain} \left ((I- \Delta)^{\frac{s}{2}} \right ) \right \} \\
                & = & \left \{  f \in \mathcal D'((0, \infty)), \ h^{\frac{1}{2}} (I- \Delta) ^{\frac{s}{2}} f \in L^2(0, \rho) \right \}.
\end{eqnarray*} 

Let us now give the precise statement of the result formulated in the introduction.

\begin{thm} \label{theoreme1}
Let $(M,g)$ be a rotationally symmetric surface satisfying (\ref{growth}). Then there exists an open subset $\mathcal P$ of $H^2_{\mathrm{r}}(M)$ such that if $u_0$ is in $\mathcal P$, then the corresponding solution $u(t)$ blows up in finite time $T < \infty$ on a set $\{ x \in M, \ r(x)=r(T) \}$ for some $r(T) \in (0, \rho)$ at the log log speed. More precisely, $u$ satisfies the following properties. There exist parameters $\lambda (t)>0, \  r(t) \in (0,a),\  \gamma (t) \in \R,\ b(t) \in \R$, and  $\tilde{u}(t), u^{\ast} \in L^2(M), \var (t) \in L^2_{\mathrm{loc} }(\R) \cap \dot{H}^1(\R)$ such that \\

\noindent 1) Decomposition of the solution: 
\[
u(t,r)=\frac{1}{\sqrt{\lambda(t)}}  \tilde{Q}_{b(t)} \left (\frac{r-r(t)}{ \lambda (t)}  \right ) e^{i \gamma(t)}  + \tilde{u}(t,r)
\]
with the convergence
\[
\tilde{u}(t)  \to u^{\ast}  \textrm{ in } L^2 (M) , \textrm{ as } t \textrm{ goes to } T.
\]
2) $u$ blows up on a curve: the function $r(t)$ has a limit $r(T) \in (0,\rho) $ when $t$ goes to $T$ and 
\[
|u(t)|^2 \to \| Q\|_{L^2(\R)} ^2  \delta_{r(T)} + | u^{\ast} |^2 , \quad  \textrm{as} \ t \to T \quad \textrm{in the sense of measures}
\]
where $\delta_{r(T)}$ is the normalized Riemannian measure on the curve $\{ x \in M, r(x)=r(T) \}$.
3) $u$ blows up at the log log speed: there exists $C>0$ such that for all $t \in [0,T)$,
\[
\frac{1}{C} \left ( \frac{\log |\log (T-t)| }{T-t} \right )^{1/2} \leq \| \nabla u(t) \|_{L^2(M)} \leq C  \left ( \frac{\log |\log (T-t)| }{T-t} \right )^{1/2}.
\]
\end{thm}

\subsection{Description of the set of blow up initial data}  To prove our theorem and without loss of generality, we will make the assumption that $(M,g)$ is one of the two model space $(\R^2, g)$ or $(S^2, g)$. We also assume $h(1)=1$ so that the singular curve will be around $r=1<\rho$. 

\medskip

For convenience, we will note $d \mu_{ \lambda, r}(y)$ or $d \mu(y)$ if the dependance in $\lambda, r$ is explicit the measure
\[
d \mu(y) = h(\lambda y + r)\mathbf{1}_{y \geq -\frac{r}{\lambda}} dy .
\]
For $\alpha^{\ast}>0$, we introduce the set $\mathcal P(\alpha^{\ast})$ of initial data $u_0 \in H^2_{\mathrm{r}}(M)$ radially symmetric of the form 
\[
 u_0(r)= \frac{1}{\sqrt{\lambda_0}} \tilde{Q}_{b_0} \left (\frac{r-r_0}{ \lambda _0} \right) e^{i \gamma _0} + \tilde{u}_0(r),
\]
with $\tilde{u}_0 \in H^2_{\mathrm{r}}(M)$  and with the following properties: \\
\textbf{A1.} Closeness of $r_0$ to $1$:
\[
 |r_0-1| < \alpha ^ \ast,
\]
\textbf{A2.} Closeness of $\tilde{Q}_{b_0}$ to $Q$:
\[
 0 < b_0< \alpha ^{\ast},
\]
\textbf{Orthogonality conditions:}
\[
  ( \mathrm{Re} \ \var_0, |y|^2 \Sigma ) + ( \mathrm{Im} \ \var_0, |y|^2 \Theta ) = 0 ,
\]
\[
  ( \mathrm{Re} \ \var_0, y \Sigma ) + ( \mathrm{Im} \ \var_0, y \Theta ) = 0,
\]
\[
  ( \mathrm{Re} \ \var_0, \Lambda \Theta) - ( \mathrm{Im} \ \var_0, \Lambda \Sigma ) = 0,
\]
\[
  ( \mathrm{Re} \ \var_0, \Lambda^2 \Theta)- ( \mathrm{Im} \ \var_0, \Lambda ^2 \Sigma) =0,
\]
where
\[
  \var_0(y)= 
\left \{ 
\begin{array}{ccl}
 \sqrt{\lambda_0} e^{- i \gamma_0} u_0(\lambda_0 y+r_0) - \tilde{Q}_{b_0}(y)  & \textrm{if} & y \geq - \frac{r_0}{ \lambda_0}, \\
 0 & \textrm{if} & y < - \frac{r_0}{ \lambda_0},
\end{array}
\right.
\]
and
\[
 \Sigma= \mathrm{Re} ( \tilde{Q}_{b}), \qquad \Theta = \mathrm{Im} ( \tilde{Q}_b),
\]
\textbf{A3.} Smallness of $\var_0$ in a weighted norm:
\[
    \int \left | \partial_y \var_0 (y) \right |^2 \mu_{ \lambda_0, r_0}(y) dy + \int_{|y| \leq \frac{10}{b_0}} | \var_0(y)|^2 e^{-|y|} dy < \Gamma^{\frac{6}{7}}_{b_0},
\]
\textbf{A4.} Control of the energy and localized momentum:
\[
 \lambda_0 ^2 |E_0|+  \lambda_0 \left | \mathrm{Im} \int \nabla \psi \cdot \nabla u_0  \overline{u}_0 \right |^2 < \Gamma^{10}_{b_0},
\]
where
\[
\psi(r) = \left \{ 
\begin{array}{rcl}
1 & \textrm{if} & \frac{1}{2} \leq r \leq \frac{3}{2}, \\
0 & \textrm{if} & r \leq \frac{1}{4} \textrm{ and } r \geq 2,
\end{array}
\right .
\]
\textbf{A5.} Log-Log regime:
\[
 0 < \lambda_0 < \exp \left ( - \exp \left( \frac{8 \pi}{9 b_0} \right ) \right ),
\]
\textbf{A6.} $L^2$ smallness of $\tilde{u}_0$:
\[
 \| \tilde{u}_0 \|_{L^2} < \alpha ^{\ast}
\]
\textbf{A7.} $H^{\frac{1}{2}}$ and $H^{\frac{k}{2}}$ estimates, $k=2,3,4$, outside the singular curve:
\[
\begin{array}{rcl}
\displaystyle{\| u_0 \|_{H^{\frac{1}{2}} (|r-1| > 1/2)} }& < & \displaystyle{(\alpha^{\ast})^{\frac{1}{4}} },\\
 \displaystyle{\| u_0 \|_{H^{\frac{k}{2}}(|r-1| > 1/2) }} & < &\displaystyle{ \frac{1}{\lambda_0 ^{ k-2 }}}.
\end{array}
\]
\textbf{A8.} $H^2$ smallness outside the blow-up curve:
\[
\|  u_0 \|_{H^2(|r-1| > \frac{1}{32} )} < \alpha^{\ast},
\]
\begin{rem}
In rescaled variable, the Sobolev space $H^2_{\mathrm{r}}(M)$ defined above is transformed into a Sobolev space with weighted $\mu(y)$; if $u_0 \in \mathcal P( \alpha ^{\ast})$, then since $\tilde{u}_0 \in H^2_ {\mathrm{r}}$, $\var_0$ belongs to:
\[
H^2_{r_0, \lambda_0} = \left \{ \var \in \mathcal D'  (-\frac{r_0}{\lambda_0}, \frac{\rho-r_0}{\lambda_0}  ), \quad  \var, \   \partial_y \var, \  \partial^2_y \var \in L^2  ((- \frac{r_0}{\lambda_0}, \frac{\rho-r_0}{\lambda_0}  ), d\mu)  \right \}.
\]
In particular, the first term in the left hand side of $\textbf{A3}$ is well defined. Moreover, with $\textbf{A1}$, $\textbf{A2}$ and $\textbf{A5}$, the second term in $\textbf{A3}$ is also well defined.
\end{rem}

\begin{pro}
 The set $\mathcal P(\alpha^{\ast})$ is nonempty and open in $H^{2}_{\mathrm{r}}(M)$.
\end{pro}

\begin{proof}
The fact that the set is non empty is similar to the flat case. It consists in choosing $(\lambda_0, r_0, b_0)$ close to $(0,1,0)$ so that 
$\textbf{A1}, \textbf{A2}, \textbf{A5}$ and next taking $\var_0(y) = \nu f$  for a well localized even function $f$ and a parameter $\nu$ small to choose so that 
$\textbf{A3}$ and $\textbf{A4}$ hold. The support condition on $f$ then implies that $u$ is supported near $r=1$ and so $\textbf{A7}, \textbf{A8}$.  \\
The fact that $\mathcal P(\alpha ^{\ast})$ is open reflects the stability with respect to initial data of the log log regime.  Let us detail this point. Let $u_0 \in \mathcal P( \alpha^{\ast})$ and $u \in H^{2}_{\mathrm{r}} (M)$ such that $\|u_0-u\|_{H^2} \leq \eta$, for some $\eta >0$ to be chosen later. Then $u_0$ and $u$ write
\begin{eqnarray*}
u_0(r)&=&\frac{1}{\sqrt{\lambda_0}} e^{i \gamma_0} \left ( \tilde{Q}_{b_0} \left ( \frac{r-r_0}{\lambda_0}  \right ) + \var_0 \left ( \frac{r-r_0}{\lambda_0} \right ) \right ) ,\\
u(r)&=& \frac{1}{\sqrt{\lambda_0}} e^{i \gamma_0} \left ( \tilde{Q}_{b_0} \left ( \frac{r-r_0}{\lambda_0}  \right ) + \var \left ( \frac{r-r_0}{\lambda_0} \right ) \right ) ,
\end{eqnarray*}
for some $\var$ and $\var_0, \lambda_0, r_0, \gamma_0, b_0$ satisfying \textbf{A1}, \textbf{A8} and the orthogonality conditions. The function $\var$ certainly does not satisfy the orthogonality conditions but by modulation theory, we can slightly modify the parameters $\lambda_0, r_0, \gamma_0, b_0$ to remove this problem. This is state in the following lemma. First, we introduce a notation. For $\delta>0$, let 
\[
V_{\delta}= \{v \in H^1_{r_0, \lambda _0}, \| v-Q \| _{H^1_{r_0, \lambda _0} } < \delta \}.
\]
					  
\begin{lem}
There exist $\delta>0$, a neighboohood $V$ of $(1,0,0,0)$ in $(0, \infty) \times \R^3$ and a function $( \lambda, r, \gamma, b): V_{\delta} \to V$ such that for every $v \in V_{\delta}$, the function $\var$ defined by 
\[
\var (y)= \lambda ^{\frac{1}{2} }(v) e^{-i \gamma (v)} v( \lambda(v)y + r(v)) - Q_{b(v)} (y), 
\]
satisfies the orthogonality conditions
\begin{eqnarray*}
  (  \var_1, |y|^2 \Sigma ) + ( \var_2, |y|^2 \Theta )& = &0 ,  \\
  ( \var_1, y \Sigma ) + (  \var_2, y \Theta ) &=& 0, \\
  ( \var_1, \Lambda \Theta) - ( \var_2, \Lambda \Sigma ) &=& 0,  \\
    ( \var_1, \Lambda^2 \Theta)- ( \var_2, \Lambda ^2 \Sigma) &=&0, 
\end{eqnarray*}
where 
\[
\var_1= \mathrm{Re} \var, \qquad \var_2= \mathrm{Im} \var.
\]
\end{lem}
The proof of this lemma is classical and relies on a perturbation argument near the point $(\lambda, r, \gamma, b, v)=(1,0,0,0,Q)$ using the implicit function theorem. See \cite{MerRap2003} for details. Now, we apply the preceding lemma with the function 
\[
v(y)= \lambda_0^{\frac{1}{2}} e^{-i \gamma_0} u (\lambda_0 y+r_0)
\]
which is in $V_{\delta}$ (from \textbf{A3}, \textbf{A6}) for some $\delta >0$ if $\eta, \alpha^{\ast}$ are small enough. Setting
\[
\tilde{\lambda}=\lambda(v) \lambda_0, \quad  \tilde{r}= r(v) \lambda_0+ r_0, \quad \tilde{\gamma}=\gamma_0+ \gamma(v), \quad \tilde{b}=b(v),
\]
we deduce that the function 
\[
\tilde{\var}(y) =\sqrt{\lambda(v)} e^{- i \gamma (v)} v ( \lambda(v) y+ r(v)) - \tilde{Q}_{b(v)}(y) = \sqrt{\tilde{\lambda}} e^{-i \tilde{\gamma}} u(\tilde{\lambda} y+ \tilde{r}) - \tilde{Q}_{b} (y)
\]
satisfies the orthogonality conditions (\ref{orth1})-(\ref{orth4}). Moreover, by continuity of $(\lambda, r, \gamma, b)$, if $\eta$ is small, the parameters $(\tilde{\lambda}, \tilde{r}, \tilde{\gamma}, \tilde{b})$ are close to $(\lambda_0,r_0,\gamma_0,0)$ so that \textbf{A1}-\textbf{A8} are true. We deduce $v \in \mathcal P( \alpha^{\ast})$. Therefore, $\mathcal P( \alpha^{\ast})$ is open. 
\end{proof}

Let $u_0 \in \mathcal P(\alpha^{\ast})$ and $u(t)$ the corresponding solution. Since $\mathcal P(\alpha^{\ast})$ is open in $H^2_{\mathrm{r}}$ and $u$ is a continuous function of time, $u(t)$ stays in $\mathcal P(\alpha ^{\ast})$ at least for a small time and in particular, we have the decomposition:
\begin{equation}\label{decomposition}
u(t,r)= \frac{1} {\sqrt{\lambda (t)}} \tilde{Q}_{b(t)} \left ( \frac{r-r(t)}{\lambda (t)} \right ) e^{i \gamma (t)} +\tilde{u}(t,r) 
\end{equation}
with 
\[
\tilde{u}(t,r) = \frac{1} {\sqrt{\lambda (t)}} \var \left (t, \left ( \frac{r-r(t)}{\lambda (t)} \right ) \right ) e^{i \gamma (t)}.
\]
As for $\var_0$, we extend $\var (t)$ by $0$ for $ y < - \lambda (t) / r(t)$. Moreover, by the local theory, the following weaker estimates hold: for $\delta>0$ to be chosen later, \\
\textbf{B1.} 
\[
 |r(t)-1 | < (\alpha^{\ast})^{1/2} , 
\]
\textbf{B2.} 
\[
0<b(t)< (\alpha^{\ast})^{1/8},
\]
\textbf{Orthogonality conditions.}
\begin{equation} \label{orth1}
  ( \mathrm{Re} \ \var (t) , |y|^2 \Sigma ) + ( \mathrm{Im} \ \var (t), |y|^2 \Theta ) = 0,
\end{equation}
\begin{equation} \label{orth2}
  ( \mathrm{Re} \ \var (t) , y \Sigma ) + ( \mathrm{Im} \ \var (t) , y \Theta ) = 0,
\end{equation}
\begin{equation} \label{orth3}
  ( \mathrm{Re} \ \var (t) , \Lambda \Theta) - ( \mathrm{Im} \ \var (t) , \Lambda \Sigma ) = 0,
\end{equation}
\begin{equation} \label{orth4}
  ( \mathrm{Re} \ \var (t) , \Lambda^2 \Theta)- ( \mathrm{Im} \ \var (t), \Lambda ^2 \Sigma) =0,
\end{equation}
\textbf{B3.} 
\[
\mathcal E(t):= \int | \partial _y \var (t)  |^2 \mu_{ \lambda (t), r(t)} dy + \int_{ |y| \leq \frac{10}{b(t)}}  |\var (t) |^2 e^{-|y|} dr < \Gamma_{b(t)}^{3/4}, 
\]
\textbf{B4.} 
\[
\lambda^2(t) |E_0| <\Gamma_{b(t)}^2,
\]
\textbf{B4'.} 
\[
\lambda  (t) \left | \mathrm{Im} \left ( \int \nabla  \psi \cdot \nabla u(t) \overline{u}(t)  \right )   \right | <\Gamma_{b(t)}^2,
\]
\textbf{B5.} 
\[
0< \lambda (t) < e^{- e^{\frac{\pi}{10 b(t)}}}, \label{B5}
\]
\textbf{B6.}
\[ 
\| \tilde{u}(t)\|_{L^2} < (\alpha^{\ast})^{1/10},
\]
\textbf{B7.} For $k=2,3,4$,
\[
\begin{array}{rcl}
\displaystyle{\|u(t)\|_{H^{\frac{1}{2}}( \{ x , |r(x)-1| > \frac{1}{2})}} &<& \displaystyle{(\alpha^{\ast})^{\frac{1}{10}}}, \\
\displaystyle{\| u(t) \|_{H^{\frac{k}{2}}(x, |r(x)-1| > \frac{1}{2})}} &< &\displaystyle{\frac{1}{\lambda (t) ^{ k-2+ (5-k) \delta}}}.
\end{array}
\]
We denote $t_1=t_1(\delta)>0$ the maximal time for which the estimates \textbf{B1}-\textbf{B7} hold. 
\begin{pro} \label{prop1}
If $\delta$ is small enough then the estimates \textbf{B1}-\textbf{B7} hold for all time $t \in [0, T)$ i.e. $t_1=T$.
\end{pro}

A big part of this paper is devoted to prove proposition \ref{prop1}. For this, we will show that for $t \in [0, t_1)$, we can obtain better estimates than  \textbf{B1}-\textbf{B7} and so necessarily $t_1=T$. More precisely, we will prove the following. \\
\textbf{C1.} 
\[ |r(t)-1 | \leq (\alpha^{\ast})^{\frac{2}{3}} , \]
\textbf{C2.} 
\[
0<b(t) \leq (\alpha^{\ast})^{1/5},
\]
\textbf{C3.}
\[
\int | \partial _y \var (t)  |^2 \mu_{ \lambda (t), r(t)} dy + \int_{ |y| \leq \frac{10}{b(t)}}  |\var (t) |^2 e^{-|y|} dy \leq \Gamma_{b(t)}^{4/5}, 
\]
\textbf{C4.}
\[
\lambda^2(t) |E_0| \leq\Gamma_{b(t)}^4,
\]
\textbf{C4'.}
\[
\lambda  (t) \left | \mathrm{Im} \left ( \int \nabla \psi \cdot \nabla u(t) \overline{u}(t)  \right )   \right | \leq \Gamma_{b(t)}^4,
\]
\textbf{C5.}
\[
0< \lambda (t) \leq e^{- e^{\frac{\pi}{5 b(t)}}},
\]
\textbf{C6.}
\[ 
\| \tilde{u}(t)\|_{L^2} \leq (\alpha^{\ast})^{1/5},
\]
\textbf{C7.} For $k=2,3,4$,
\[
\begin{array}{rcl}
\displaystyle{\|u(t)\|_{H^{\frac{1}{2}}(|r-1| > \frac{1}{2})}} &\leq & \displaystyle{(\alpha^{\ast})^{\frac{1}{5}}}, \\
\displaystyle{\| u(t) \|_{H^{\frac{k}{2}}(|r-1| > \frac{1}{2})}} & \leq &\displaystyle{\frac{1}{2 \lambda (t) ^{ k-2+ (5-k) \delta}}}.
\end{array}
\]

\bigskip

\textbf{Sketch of the proof.} Let us sketch the proof of the theorem. Following the work \cite{Rap2006} in the Euclidean case, the idea is first to use the radial assumption of the manifold and the functions to rewrite the equation in these coordonates $(r, \theta)$ and see the new equation as a one dimensional equation and then use the log log $L^2$-critical theory. In the new variables, the Laplace operator is not exactly the one dimensional Laplacian $\partial_r^2$ and there is an other term that we will treat as a perturbation; for a radial function $u$, 
\[
\Delta_M u =\partial_r^2 u  + \frac{h'}{h} \partial_r u.
\]
The stability of the log log regime suggests that this term will be negligeable. Then one may see the derivation of the log log speed as a consequence of modulation theory and linearization of conservation laws or more generally remarquable identities that solutions of Schr\"odinger equation satisfy. First, by modulation theory, we may write our solution $u(x)=u(r)$ as 
\[
u(x)= \frac{1}{ \sqrt{ \lambda (t)}} \left ( \tilde{Q_b} ( \frac{r-r(t)}{\lambda (t)} ) + \var(t, \frac{r-r(t)}{\lambda (t)} ) \right ) e^{i \gamma (t)} ,
\] 
for some parameters $\lambda (t), r(t), b(t), \gamma (t)$ and a rest $\var$ small in some sense. The study of the modulation parameters does not rely on conservation laws and consists in multiplying the equation for $u$ by some well chosen quantities to extract relations between the parameters under differential forms and the rest $\var$. To exploit these relations, we need dispersive properties of $\var$ that are consequences of the linearization of a virial type identity for the one dimensional equation (i.e. in $r$ variable) and conservation laws. The linearization of the virial will make appear a quadratic form which is positive except for a finite number of negative directions. Two arguments allow to treat these directions. First, from modulation, we may impose a certain number of orthogonality conditions to vanish some order one in $\var$ inner products and by linearizing conservation laws, we deduce estimates on the other negative directions. Note that the radial symmetry is an essential assumption to see the two dimensional equation as a one dimensional one but we do not need any other symmetry, except the classical ones: phase and translation in time. In particular, no translation in space is required to have a momentum conservation law. This approach is the same like for the study of stability of ground state for NLS or solitary waves for KdV: we expand in $\var$ a conservation law, vanish order one inner products and use the positivity of the second order term to bound the rest in the decomposition of the solution. When linearizing the energy, we will need a smallness of the critical norm $H^{1/2}$ of the rest $\tilde{u}$ outside the location of singularities. In the case where the manifold is non compact, we may use the good dispersive behavior of the linear flow to see this smallness as a consequence of the $H^{1/2}$ local smoothing effect. In the general case, we use a strategy again based on (almost) conservation laws and developped in \cite{RapSze2009}. It consists in introducing an $H^2$ energy and proving that it is in some sense subcritical and this allows to first deduce an $H^2$ estimate and then the $H^{1/2}$ smallness.

\medskip

The paper is organized as follow. In section \ref{eqforvar}, we linearize the equation around the modulated ground state $\tilde{Q}_b$ and thus deduce the equation satisfied by the rest $\var$. In the next section, we derive controls on some inner products using conservation laws: conservation of mass, energy and momentum. The proof consists in linearizing the conservation laws and deduce estimates on the dominant terms. Note that for this part, we need the smallness of the critical norm. In the section \ref{estiparadiff}, we treat the finite dimensional part and derive estimates on the parameters under differential form. In section \ref{virialloc}, we give the virial estimates which allow to treat the infinite dimensional part $\var$. Combining the two previous sections, we may integrate the differential inequalities to deduce the first estimates on the parameters. This is done in section \ref{sectionestimates}. In the next section, we proved the refined virial estimate to obtain the lower bound on the blow up rate. In section \ref{sectionnorm}, we prove the smallness of the $H^{1/2}$ norm outside the blow up curve. Once we have proved all the bootstrap conclusions, we show that these estimates are sufficient to conclude the proof of the theorem.

\section{\texorpdfstring{Linearization of the equation}{Linearization of the equation}} \label{eqforvar}

In this section, we write the equation (\ref{nls}) in term of $\var$. We introduce the change of time variable
\begin{equation} \label{s0}
s= s_0 + \int_{0}^t \frac{d \tau}{ \lambda^2(\tau)}, \qquad, \textrm{ avec } s_0=e^{ \frac{3 \pi}{4 b_0}}.
\end{equation}
The corresponding final time is $s_1=s(t_1)$. Remark that $s_1$ may eventually be infinite depending on the value of $t_1$ and the behaviour of $\lambda(t)$. Then, $\tilde{Q}_{b}$ satisfies
\[
\partial_y^2 \tilde{Q}_b + \frac{ \lambda (s)}{ \mu(y)}\partial_y \tilde{Q}_b - \tilde{Q}_b + ib \Lambda \tilde{Q}_b + \tilde{Q}_b | \tilde{Q}_b |^4= - \tilde{\Psi}_b,
\]
where
\[
\tilde{\Psi}_b= \Psi_b  - w \lambda \partial_y \tilde{Q}_b, \qquad w=w(t,y)=\frac{h'( \lambda (t)y + r(t))}{h (\lambda (t) y +r(t))}
\]
In the rest of the paper, we will often remove the $b$ parameter and write for instance $\tilde{Q}$ for $\tilde{Q}_b$. Thus starting from (\ref{nls}), a direct computation gives 
\begin{eqnarray}
  \label{eq1} \partial_s \Sigma + \partial_s \var _1 - M_{-}( \var) + b \Lambda \var_1 &=& \left ( \frac{\lambda_s}{\lambda} +b \right ) \Lambda \Sigma + \tilde{\gamma}_s \Theta + \frac{r_s}{ \lambda}  \partial_y \Sigma \\ \nonumber
                                                                                             & &+ \left( \frac{\lambda_s}{\lambda} + b \right ) \Lambda \var_1 + \tilde{ \gamma}_s \var_2 + \frac{r_s}{\lambda}  \partial_y \var_1  \nonumber \\ 
                                                                                             & & + \mathrm{Im} (\tilde{\Psi}) - R_2 ( \var)  \nonumber \\  
   \label{eq2} \partial_s \Theta + \partial_s \var _2 + M_{+}( \var) + b \Lambda \var_2 &=& \left ( \frac{\lambda_s}{\lambda} +b \right ) \Lambda \Theta - \tilde{\gamma}_s \Sigma + \frac{r_s}{ \lambda}  \partial_y\Theta \\ \nonumber
                                                                                             & &+ \left( \frac{\lambda_s}{\lambda} + b \right ) \Lambda \var_2 - \tilde{ \gamma}_s \var_1 + \frac{r_s}{\lambda} \partial_y\var_2   \\ \nonumber
                                                                                            & & - \mathrm{Re} (\tilde{\Psi}) + R_1 ( \var) \\ \nonumber
\end{eqnarray}
where we denote $\tilde{\gamma}(s)=-s+\gamma(s), \  M=(M_{+}, M_{-})$ the operators:
\begin{eqnarray*}
 M_{+}(\var)&=&- \partial_{y}^2 \var_1 - \lambda  w \partial_y \var_1 + \var_1 - \left( \frac{4 \Sigma^2}{ |\tilde{Q}|^2} +1 \right) |\tilde{Q}|^{4} \var_1 -  4 \Sigma \Theta  |\tilde{Q}|^{2} \var_2 , \\
 M_{-}(\var)&=& -\partial_y^2 \var_2- \lambda w \partial_y \var_2 + \var_2 - \left( \frac{4 \Theta^2}{|\tilde{Q}|^2} +1 \right) |\tilde{Q}|^{4} \var_2 - 4 \Sigma \Theta |\tilde{Q}|^2\var_1 ,
\end{eqnarray*}
and $R_1, R_2$ the nonlinear terms defined by
\begin{eqnarray*}
 R_1(\var)&=& (\var_1+ \Sigma) | \var+ \tilde{Q}|^{4}- \Sigma | \tilde{Q}|^{4}- \left( \frac{4 \Sigma^2}{|\tilde{Q}|^2} +1 \right) |\tilde{Q}|^{4} \var_1 - 4 \Sigma \Theta  |\tilde{Q}|^2 \var_2 ,\\
R_2(\var) &=& (\var_2+ \Theta) | \var+ \tilde{Q}|^{4} - \Theta | \tilde{Q}|^{4}- \left( \frac{4 \Theta^2}{ |\tilde{Q}|^2} +1 \right) |\tilde{Q}|^{4} \var_2 - 4 \Sigma \Theta  |\tilde{Q}|^2  \var_1.
\end{eqnarray*}
Note that $M$ is essentially the linearized operator close to the modified ground state $\tilde{Q}_b$ for the quintic equation posed in $\mathbb R$ plus order one terms that we hope to be negligeable. Formally, in the limit $b \to 0$, $M$ tends to the linearized operator around $Q$, $L=(L+, L-)$, of the quintic equation in $\R$ where
\[
L_+=- \partial_y^2+1-5 Q^4, \qquad L_-=- \partial_y^2 +1-Q^4.
\]
\section{Linearization of the conservation laws}  \label{sectionconservation}
In the following proposition and in the rest of the paper, $\delta( \alpha ^{\ast})$ is a quantity that tends to $0$ as $\alpha ^{\ast}$ goes to $0$.
\begin{pro}
For all $s \in [s_0, s_1)$, we have the following controls \\

\noindent (\romannumeral 1) Estimate induced by mass conservation:
\begin{equation}
 d_0 b^2(s) + \int | \tilde{u} (t)|^2 \leq (\alpha ^{\ast})^{\frac{1}{2}}.  \label{mass}
\end{equation}
(\romannumeral 2) Estimate induced by energy conservation:
\begin{equation} 
\left| 2(\var_1, \Sigma ) + 2 ( \var_2, \Theta)- \int | \partial_y \var |^2 \mu(y) dy +5 \int_{ |y| \leq \frac{10}{b}} Q^4 |\var_1|^2 +  \int_{ |y| \leq \frac{10}{b}} Q^4 | \var_2 |^2  \right | \label{energy}
\end{equation}
\begin{flushright}
$ \leq \Gamma_b^{1- C \eta} + \delta ( \alpha^{\ast}) \mathcal E (t) . $
\end{flushright}
 (\romannumeral 3) Estimate induced by the smallness of the localized momentum:
\begin{equation}
| (  \var_2 , \partial_y \Sigma ) | \leq \delta( \alpha ^{\ast}) \mathcal E^{\frac{1}{2}}(t) + \Gamma_{b}^2. \label{momentum}
\end{equation}
\end{pro}

\begin{proof} \textbf{(\romannumeral 1) Conservation of the mass.} The proof of the estimate (\ref{mass}) consists in using the conservation of the mass and expanding in terms of $\alpha ^{\ast}$ the equality $\|u(t)\|_{L^2}^2= \|u_0\|_{L^2}^2$. We write using polar coordonates and the change of variable $y=\frac{r- r(t)}{\lambda (t)}$, 
\[ \|u_0\|_{L^2}^2= \int |\tilde{Q}(y)|^2  \mu_{\lambda(t),r(t)} (y) dy + 2 \mathrm{Re} \int \tilde{Q}(y) \overline{\var} (t,y) \mu_{\lambda(t),r(t)} (y) dy + \| \tilde{u} \|_{L^2}^2.
\]
To treat the first term, we first use that the weight $\mu$ is close to $h(1)=1$ if $ |y| \leq 10/b$. From \textbf{B1} and \textbf{B5}, we have if 
$|y| \leq 10/b$, $ |  \lambda (t)  y + r(t) -1 | \leq  \mathrm{inf} \{ (a-1)/2, 1 \}$ so that $ | h(\lambda (t) y + r(t) )- 1 | \leq C | \lambda (t) y +r(t)-1 | \leq (\alpha ^{\ast})^{1/2}$
and thus
\[
\int |\tilde {Q}(y) |^2 \mu (y) dy = \int |\tilde {Q}(y) |^2  dy + \mathcal O ( ( \alpha^{\ast}) ^{\frac{1}{2}})  .
\]
Then, recall that the mass of $\tilde{Q}$ verifies
\[
\|\tilde{Q} \|_{L^2}^2 = \| Q\|_{L^2}^2 + d_0 b^2+ \mathcal O (b^4).
\]
From $b \leq (\alpha^{\ast} )^{1/8}$, we deduce
\[
\int |\tilde{Q}(y) |^2 \mu(y) dy = \| Q\|_{L^2}^2 + b^2(t) + \mathcal O (( \alpha ^{\ast})^{\frac{1}{2}}).
\]
The second term is estimated from Cauchy-Schwarz and the control on $\mathcal E (t)$ given by \textbf{B3}. This gives using $ | \tilde{Q}(y)| \leq C \mathrm{exp} ( - |y|/2)$,
\[
\left |2 \mathrm{Re} \int \tilde{Q}(y) \overline{\var} (t,y) \mu (y) \right | \leq  C \int_{ |y| \leq 10/b} e ^{ -\frac{|y|}{2}}  | \var| dy \leq \frac{C}{ \sqrt{b}} \mathcal E (t) \leq  (\alpha^{\ast}) ^{\frac{1}{2}}.
\]
Summing all the above estimates, we obtain
\[
 \|u_0\|_{L^2}^2= \|Q\|_{L^2}^2 + d_0 b^2(t) + \|\tilde{u} \|_{L^2}^2 + \mathcal O ((\alpha ^{\ast})^{1/2}).
\]
But since the mass of the initial data is close to the mass of $Q$:
\[ 
0 \leq \|u_0\|_{L^2} -\|Q\|_{L^2} \leq \alpha ^{\ast},
\]
we can conclude 
\[
d_0 b^2(t) + \|\tilde{u}\|_{L^2}^2 \leq (\alpha ^{\ast})^{\frac{1}{2}}.
\]
In particular, this proves \textbf{C2} and \textbf{C6}.

\medskip

\textbf{ (\romannumeral 2) Conservation of the energy.} We expand in $\var$ the equality $E(u(t))=E(u_0)$. A direct computation gives
\begin{equation} \label{energy2}
 E(u_0)=\frac{C}{\lambda^2(t)} \sum_{k=0}^6 A_k 
\end{equation}
where
\begin{eqnarray*}
A_0&=& \frac{1}{2} \int |\partial_y \tilde{Q} |^2  \mu(y) dy- \frac{1}{6}  \int |\tilde{Q} |^6 \mu(y) dy , \\
A_1&=&  \mathrm{Re} \int \partial_y \tilde{Q} \partial_y \overline{\var} \mu(y) dy- \int | \tilde{Q}|^4 \mathrm{Re} \left ( \tilde{Q} \overline{\var} \right ) \mu(y) dy, \\ 
A_2&=&\frac{1}{2} \int |\partial_y \var |^2 \mu(y) dy- \frac{1}{2} \int |\tilde{Q} |^4 |  \var |^2 \mu(y) dy -2 \int | \tilde Q |^2 \mathrm{Re}^2 \left ( \tilde{Q} \overline{ \var} \right ) \mu(y) dy, \\
A_3&=&-2 \int | \tilde Q|^2 |\var|^2 \mathrm{Re} \left ( \tilde Q \overline{\var} \right) \mu(y) dy -\frac{4}{3} \int \mathrm{Re}^{3} \left (  \tilde Q \overline{ \var} \right ) \mu(y) dy, \\
A_4&=&- \frac{1}{2} \int | \tilde Q |^2 | \var|^4 \mu(y) dy -2 \int |\var|^2 \mathrm{Re}^2 \left ( \tilde Q \overline{\var} \right) \mu(y) dy, \\
A_5&=& -\int  | \var|^4 \mathrm{Re} \left ( \tilde Q \overline{\var} \right ) \mu(y) dy, \\
A_6&=& -\frac{1}{6} \int | \var|^6 \mu(y) dy .
\end{eqnarray*}
Now, we estimate each $A_k$. Remark that two of these terms are non localized in space, in the sense that the integrande is not multiplied by a power of $\tilde{Q}$, namely
\[
\frac{1}{2} \int | \partial_y \var |^2 \mu(y) dy, \qquad -\frac{1}{6} \int | \var|^6 \mu(y) dy.
\]
These terms will require a special treatment.  For the other ones,  we will use from \textbf{B5}  the approximation $ \mu(y) \approx h(r(t))$: for $|y| \leq 10 /b$,
\[
  |\mu(y) - h(r(t)) | \leq  C \Gamma _b.
\]
For $A_0$, we find essentially the one dimensional energy of $\tilde{Q}_b$. Indeed, by \textbf{B1} and \textbf{B5}, 
\begin{eqnarray*}
 A_0 &=& E ( \tilde{Q} _b)  (h(r(t)) + \mathcal O (\Gamma_b) )  \\
   &=& \mathcal O( \Gamma_b ^{1- C \eta} ).
\end{eqnarray*}
For the linear term, using again $\mu(r) \approx h(r(t))$ and after integration by parts in the first term, we get 
\[
 A_1=-\left ( h(r(t)) + \mathcal O ( \Gamma_b) \right ) \left (  \mathrm{Re} \int \left (\partial_{yy}^2 \tilde{Q}_b \overline{\var} + | \tilde{Q} _b |^4 \tilde{Q}_b \overline{\var} \right )  dy \right )
\]
Using the equation on $\tilde{Q}$, we deduce 
\[
 A_1=-\left ( h(r(t)) + \mathcal O ( \Gamma_b) \right )  \left ( \mathrm{Re} \left (  \int \tilde{Q}_b  \overline{\var} \mu(y) dy \right ) - b(t)  \mathrm{Im} \left ( \int \Lambda \tilde{Q}_b \overline{\var} \right )+ \mathrm{Re}  \left ( \int \Psi _b \overline{\var}  dy \right )  \right ) .
\]
The second term is zero by the orthogonality condition (\ref{orth3}) so that 
\[
 A_1=-\left ( h(r(t)) + \mathcal O ( \Gamma_b) \right )  \left ( ( \Sigma , \var_1) + ( \Theta , \var_2) +  \mathrm{Re}  \left ( \int \Psi \overline{\var} dy \right ) \right ) .
\]
At first sight, the third term is of size $\Gamma_b^{1/2} \mathcal E ^{1/2} (t)$ but in fact, we can improve this bound. We state the general result in the following lemma which will be used
several times in the sequel.
\begin{lem} \label{psi}
For every $k, p \in \mathbb N$, 
\[
 \left | \int \var y^ p \partial^k_y \Psi_b \right | \leq \delta( \alpha ^{\ast}) \mathcal E (t) + \Gamma_b^{1- C \eta}. 
\]
\end{lem}

\begin{proof}
We first prove by density that for all $\var \in H^1$ and for all $y \in [- 2/b, 2/b]$, 
\begin{equation} \label{densityargument}
|\var (y) | \leq \left ( \int_{-2/b} ^{2/b} | \var (y) |^2 e^{- |y|} dy \right )^{1/2} + \frac{C}{\sqrt{b}} \left ( \int_{- \frac{2}{b}} ^{\frac{2}{b}} | \partial_y \var (y) |^2 dy \right )^{1/2}.
\end{equation}
Let $\var $ smooth and compactly supported and $y_0 \in [-2/b,2/b]$ depending on $b$ and $\var$ such that 
\[
  | \var (y_0 )|\leq \left ( \int _{-2/b} ^{2/b} | \var (y)  |^2 e^{- |y|} dy \right )^{1/2}.
\]
By a contradiction argument, this point always exists since $\int e^{-|y|}  dy >1$. Then by writing for all $y \in [-2/b,2/b]$,
\[
 \var (y) = \var (y_0) + \int_{y_0} ^y \partial_ y \var (y) dy ,
\]
we obtain (\ref{densityargument}) by Cauchy-Schwarz. We apply this to our situation: we reintroduce the function $\mu$ since $|y| \leq 2/b$ to first have:
\[
 | \var (y) | \leq \mathcal E(t) ^{1/2} +\frac{C}{\sqrt b }  \left (  \int _{ - \frac{2}{b}} ^{ \frac{2}{b}} | \partial_y \var |^2 dy \right ) ^{1/2} \leq \left (1+ \frac{C}{ \sqrt{b}} \right ) \mathcal E(t)^{1/2} 
\]
and then the lemma is proved by integration and using $ |y| ^p | \partial _ k  \Psi_b (y) | \leq \Gamma_b ^{1/2}$. 
\end{proof}
 In particular:
\[
 \left | \mathrm{Re} \int \Psi_b \overline{\var} dy \right | \leq  \delta( \alpha ^{\ast}) \mathcal E (t) + \Gamma_b^{1- C \eta}, 
\]
and then 
\[
 \left |  A_1+ h(r(t)) \left  (   ( \Sigma, \var_1) + ( \Theta, \var_2) \right ) \right | \leq \Gamma_b^{1- C \eta}  + \mathcal E (t).
\]
For the quadratic term, we use the uniform closeness of $\tilde{Q}$ to $Q$ (\ref{closeness}):
\begin{eqnarray*}
 A_2&=& \frac{1}{2} \int | \partial_y \var |^2 \mu(y) dy  +( h(r(t)) + \mathcal O (\Gamma_b) ) \left ( \frac{1}{2}  \int Q^4 | \var|^2 dy +2 \int Q^2 \mathrm{Re} ^2  \left (Q \overline{\var}  \right ) dy \right )  ,
\end{eqnarray*}
so that 
\[
\left |  A_2 -\frac{1}{2} \int | \partial_y \var |^2  \mu(y) dy  - h(r(t))  \left ( \frac{5}{2} \int _{ |y| \leq \frac{10}{ b}}  Q^4  | \var_1  |^2     + \int_{ |y| \leq \frac{10}{b}} Q^4 | \var_2 |^2    \right ) \right | \leq \delta( \alpha ^{\ast} )\mathcal E(t) + \Gamma_b^{1-C \eta}.
\]
The terms $A_3, A_4$ are treated the same way. For instance, for $A_4$, we may write
\begin{eqnarray*}
 |A_4| &\leq& \| \var\|^2_{L^{\infty}( |y| \leq 10 /b )}  \int_{ |y|  \leq \frac{10}{b(t)}} |\tilde {Q}|^2 | \var |^2 dy \\
       & \leq & \| \var\|^2_{L^{\infty}( |y| \leq 10 /b )}  \int_{|y|  \leq \frac{10}{b(t)}} e^{-|y| } |\var |^2 dy.
\end{eqnarray*}
Now, we have to control the $L^{\infty}$ norm of $\var$. For this, we apply the Sobolev inequality to $\phi_1 \var$ where $\phi_1$ is a cut-off function satisfying 
\[
  \phi_1 (y)= 
\left \{ 
\begin{array}{rcl}
 0 & \textrm{if} & y \geq \frac{1}{2 \lambda}, \\
 1 & \textrm{if} & y \leq \frac{1}{4 \lambda},
\end{array}
\right.
\]
and the following bound on the derivative $\| \partial_y \phi_1 \|_{L^{\infty} } \leq \lambda$, 
\begin{eqnarray*}
 \| \var  \|_{L^{\infty}\left ( |y| \leq \frac{10}{b}\right )} &\leq&  \| \var  \phi_1 \|_{L^{\infty}} \\
                                                    & \leq & \| \partial_y ( \phi_1 \var) \|_{L^2} ^{\frac{1}{2}}  \| \phi _1\var \|_{L^2} ^{\frac{1}{2}}  \\
                                                    & \leq &\left (  \int_{y \leq \frac{1}{2\lambda}} | \partial_y \var |^2 + \lambda | \var |^2 \right ) ^{\frac{1}{4}} \left (  \int _{y  \leq \frac{1}{ 2 \lambda}} | \var |^2 \right)^{ \frac{1}{4} } 
\end{eqnarray*}
Reintroducing the measure $\mu(y)$ which is close to $h(r(t))$ if $|y| \leq 1/(2\lambda)$ and using \textbf{B3}, \textbf{B5}, \textbf{B6}, we obtain the bound 
\begin{eqnarray*}
 \| \var  \|_{L^{\infty}\left ( |y| \leq \frac{10}{b}\right )} &\leq &\left (  \int_{y \leq \frac{1}{2\lambda}} (| \partial_y \var |^2  + \lambda | \var |^2 ) \mu(y) dy \right ) ^{\frac{1}{2}} \left (  \int _{y  \leq \frac{1}{ 2 \lambda}} | \var |^2 \mu(y) dy  \right)^{ \frac{1}{2} } \\
 & \leq &  \delta ( \alpha ^{\ast }) \mathcal E^{1/4} (t)+ \Gamma_b^{1-C \eta} .
\end{eqnarray*}
Therefore, $A_3$ and $A_4$ are negligible:
\[
 |A_3| + | A_4| \leq \delta ( \alpha^{\ast}) \mathcal E(t) + \Gamma_b^{1-C eta}.
\]
For $A_5$, we first write by Cauchy-Schwarz inequality and then the $L^{\infty}$ bound found above~:
\[
 |A_5| \leq ( h(r(t)) + \mathcal O (\Gamma_b) ) \left ( \int_{|y| \leq \frac{1}{10b(t)}} | \var |^8  dy \right)^{\frac{1}{2}} \left( \int  | \tilde{Q}|^2 | \var|^2  \right)^{\frac{1}{2}} \leq \delta ( \alpha^{\ast}) \mathcal E(t) + \Gamma_b^{1-C \eta}.
\]
We now turn to the estimate of the non-localized term $A_6$. We cannot proceed as before since we have no control of the $L^{\infty}$ norm of $\var$ on the whole space. We split the space into two areas, one is localized near the singular curve $r \approx 1$ and the other one near $0$ and infinity. This splitting writes 
\[
 \int |\tilde{u}|^6 \leq \int | \chi_1 \tilde{u} |^6 + \int | \chi_2 \tilde{u} |^6,
\]
where $\chi_2$ is a radial cut-off localized near $r =1$ and $\chi_1$ outside:
\[
 \chi _1 (r) = \left \{ 
\begin{array}{rcl}
 0 & \textrm{if} & r \in [ \frac{3}{4}, \frac{5}{4} ] , \\
1 &  \textrm{if} &r\in [0, \frac{1}{2}] \cup [ \frac{3}{2}, \rho)
\end{array}
\right .
,
\quad 
 \chi _2 (r) = \left \{ 
\begin{array}{rcl}
 0 & \textrm{if} & r \in [0, \frac{1}{4}] \cup [ \frac{7}{4}, \rho)  , \\
1 &  \textrm{if} & r  \in [ \frac{1}{2}, \frac{3}{2} ] 
\end{array}
\right .
.
\]
For the first integral, we use the 2D Gagliardo-Nirenberg inequality and the $H^{1/2}$-smallness estimate near the blow-up curve
\[
 \int |\chi_1 \tilde{u} |^6 \leq \| \nabla ( \chi_1 \tilde{u} ) \|_{L^2}^2  \| \chi_1 \tilde{u}  \|_{H^{1/2}}^4 \leq \delta(\alpha^{\ast})  \left( 1+ \| \nabla \tilde{u} \|_{L^2}^2  \right).
\]
For the second integral, we use interpolation: 
\[
 \int | \chi_2 \tilde{u} |^6 \leq  \| \chi_2 \tilde{u} \|^2_{H^1} \| \chi_2 \tilde{u} \|_{L^2}^4 \leq \delta(\alpha^{\ast}) \left ( 1+  \| \nabla \tilde{u} \|_{L^2}^2  \right ).
\]
We deduce from the last estimates that
\begin{eqnarray*}
 |A_6| &\leq & \delta( \alpha^{\ast}) \lambda ^2 (t) \left ( 1+  \| \nabla \tilde{u} \|_{L^2}^2  \right ) \\
      & \leq & \lambda^2(t) + \delta( \alpha^{\ast}) \int | \partial_y \var |^2 \mu(r) dr \\
       & \leq &  \Gamma_b + \delta ( \alpha^{\ast} ) \mathcal E(t).
\end{eqnarray*}
Summing and using \textbf{B4} and $|h(r(t)) | \leq 2$, we conclude the proof of (\ref{energy}). \\

\textbf{(\romannumeral 3) Estimate induced by the smallness of the localized momentum.} We expand in $\var$ the local momentum
\[
 \lambda(t)\mathrm{Im} \int \nabla \psi \cdot \nabla u(t) \overline{u}(t)  .
\]
Recall that $\psi$ is such that 
\[
\psi(r) = \left \{ 
\begin{array}{rcl}
1 & \textrm{if} & \frac{1}{2} \leq r \leq \frac{3}{2}, \\
0 & \textrm{if} & r \leq \frac{1}{4} \textrm{ and } r \geq 2.
\end{array}
\right .
\]
so that $\pa _y \psi ( \lambda (t) y + r(t))=1$ if $y \in \mathrm{Supp}( \tilde{Q})$. Using polar coordonates and change of variables $y=\frac{r-r(t)}{ \lambda(t)}$, this gives 
\begin{eqnarray*}
\lambda(t)\mathrm{Im} \int \nabla \psi \cdot \nabla u(t) \overline{u}(t) & =& C \mathrm{Im} \int \pa_y \tilde{Q} \overline{\tilde{Q}} \mu(y) dy + C  \mathrm{Im} \int \pa _y \tilde{Q} \overline{ \var} \mu(y) dy  \\
                                                                                                                      & &+ C \mathrm{Im} \int \partial_y \var \overline{\tilde{Q}} \mu(y) dy  +C  \int \pa _y \psi (\lambda (t) y + r(t))\pa_y \var \overline{\var} \mu(y) dy .
\end{eqnarray*}
Expanding $\mu(y)$ as before:
\[
  | \mu(y) - h(r(t)) | \leq \Gamma_b ^2,
\]
and observing that $\tilde{Q}$ is even, we have that the first term is bounded by $\Gamma_b^2$. Again by splitting the function $\mu(y)$, the second term is such that
\[
\left | \mathrm{Im} \int \partial_y \tilde{Q} \overline{\var} \mu(y) dy - h(r(t)) \left ( ( \partial_y \Sigma, \var_2) + ( \partial_y \Theta, \var_1 ) \right ) \right | \leq \Gamma_b^2.
\]
Moreover, using (\ref{closeness}) we also have
\[
 | ( \pa_y \Theta, \var_1) | \leq \|\pa_y \Theta e^{ \frac{y}{2}} \|_{L^2} \| \var_1 e^{-\frac{y}{2}} \|_{L^2( -10/b, 10/b)}  \leq \delta( \alpha^{\ast}) \mathcal E ^{ \frac{1}{2}} (t) ,
\]
and then 
\[
 \left | \mathrm{Im} \int \partial_y \tilde{Q} \overline{\var} \mu(y) dy - h(r(t)) \left (  \partial_y \Sigma , \var_2  \right ) \right | \leq \delta( \alpha ^{\ast}) \mathcal E^{1/2} + \Gamma_b.
\]
The third term in the expansion of the localized momentum is estimated like the second one. For the last one, using Cauchy-Schwarz and the control of $ \mathcal E(t)$:
\begin{eqnarray*}
\left |  \int \pa _y \psi \pa \var \overline{\var} \mu(y) dy \right | & \leq & \left ( \int | \pa_y \var |^2 \mu(y) dy \right )^{\frac{1}{2}}  \left ( \int | \var| ^2 \mu(y) dy \right )^{\frac{1}{2}} \\
                                                                            & \leq & \mathcal E^{\frac{1}{2}}(t) \|\tilde{u} \|_{L^2} \\
                                                                            & \leq & \delta( \alpha^{\ast}) \mathcal E^{\frac{1}{2}}(t).
\end{eqnarray*}
Finally, we deduce from \textbf{B4'}:
\[ 
| ( \partial_y \Sigma , \var_2) | \leq \delta( \alpha ^{\ast}) \mathcal E^{\frac{1}{2}}(t) + \Gamma_{b}^2.
\]
\end{proof}

\section{Estimates on the parameters under differential form} \label{estiparadiff}
In this section, using our choice of orthogonality conditions (\ref{orth1})-(\ref{orth4}), we deduce estimates involving the geometrical parameters or more precisely their derivatives.
 
\begin{pro}
For every $s \in [s_0, s_1)$, we have the following estimates
\begin{equation} \label{parameter}
  \left | \frac{\lambda _s }{\lambda} + b \right |+ |b_s| + \left |  \frac{r_s}{\lambda}  \right |   \leq   C \mathcal E (t) + \Gamma_b ^{1-C \eta}   .
\end{equation}
\begin{equation} \label{gamma}
 \left | \tilde{\gamma}_s  - \frac{ (\var_1, L_+( \Lambda ^2 Q))}{ \| \Lambda Q \|_{L^2}^2 } \right | \leq \delta( \alpha ^{\ast}) \mathcal E^{\frac{1}{2}} (t) + \Gamma_b^{1-C \eta} .
\end{equation}
\end{pro}

\begin{proof}
\textbf{Estimate for $\frac{\lambda_s}{\lambda}+b$}. We take the inner product of (\ref{eq1}) with $|y|^2 \Sigma$ and (\ref{eq2}) with $|y|^2 \Theta$ and sum the two equalities. Next using the orthogonality conditions (\ref{orth1})-(\ref{orth4}) and several integrations by parts, we can group terms together to obtain
\begin{eqnarray*}
& & \left ( \frac{ \lambda_s }{\lambda}+b \right ) \|y \tilde Q \|_{L^2}^2 = -\frac{1}{2} \partial_s \left ( \| y \tilde Q \|_{L^2}^2 \right ) - \mathrm{Im} ( \lambda w  \partial_y \tilde{Q}, y^2 \tilde{Q}  ) +  \mathrm{Im} ( \var, y^2 \tilde {\Psi})  \\
 & & -\left ( \frac{\lambda_s}{\lambda}+b\right ) \mathrm{Re} (\var, \Lambda ( y^2 \tilde{Q} ) ) - \frac{r_s}{\lambda} \mathrm{Re} ( \var, \partial_y (y^2 \tilde{Q} ) ) +b_s \mathrm{Re} ( \var, y^2  \partial_b \tilde{Q}) \\
 & &    + \tilde{\gamma}_s \mathrm{Im} ( \var, y^2 \tilde{Q} ) - \mathrm{Im} ( \lambda w  \partial_y \var, y^2 \tilde{Q} )   + (R_1(\var), y^2 \Theta) - (R_2( \var),y^2 \Sigma) . 
\end{eqnarray*}
Now, we estimate each term. First, for the left hand side, we have the lower bound
\[
  \left | \left ( \frac{ \lambda_s }{\lambda}+b \right ) \|y \tilde Q \|_{L^2}^2 \right | \geq C \left | \frac{\lambda_s}{\lambda}+ b \right |.
\]
Next for the term $\partial_s \|y \tilde Q\|_{L^2}^2 $, we write
\begin{equation} \label{term1}
\partial_s \| y\tilde{Q} \|_{L^2}^2 = b_s \mathrm{Re} ( \partial_b \tilde{Q}, y^2 \tilde{Q})  . 
\end{equation}
Using that $\partial_b \tilde{Q}$ is close to $\frac{i}{4} y^2 Q$ and $y^2 \tilde{Q}$ is close to $y^2 Q$ and $\mathrm{Re}( \frac{i}{4} y^2 Q, y^2 Q)=0$, we can bound (\ref{term1}) by 
$\delta( \alpha^{\ast}) |b_s|$. For the term $\mathrm{Im} (\var, y^2 \tilde{\Psi})$, we use Lemma \ref{psi} and the smallness of $\lambda$ \textbf{B5} to deduce
\[
| \mathrm{Im}  (\var, y^2 \tilde{\Psi}) |  \leq  \Gamma_b^{1-C \eta} +  \mathcal E(t).                                                       
\]
For terms appearing with the function $w$, namely
\[
 \mathrm{Im} ( \lambda w \partial_y \tilde{Q}, y^2 \tilde{Q} ), \qquad \mathrm{Im} ( \lambda w   \partial_y \var, y^2 \tilde{Q} ),
\]
we use the bound \textbf{B5} on $\lambda$ and the lower bound for $|y| \leq 10/b$:
\begin{equation} \label{minorationw}
| w(y)| =\left | \frac{h'(\lambda (t) y + r(t) )}{h(\lambda (t) y + r(t) )}  \right | \leq  \frac{ 2| h'(1)| }{ h(1)} 
\end{equation}
to estimate the first of these terms by
\[
  | \mathrm{Im} ( \lambda w  \partial_y \tilde{Q}, y^2 \tilde{Q}) | \leq \Gamma_b.
\]
For the second, using again (\ref{minorationw}) and also Cauchy-Schwarz, we have 
\[
 | \mathrm{Im} ( \lambda w  \partial_y \var, y^2 \tilde{Q} ) | \leq \Gamma_b \mathcal E^{\frac{1}{2}}(t) \leq \Gamma_b.
\]
Now, for terms 
\[
 \mathrm{Re} (\var, \Lambda ( y^2 \tilde Q ) ), \ \mathrm{Re} ( \var, \partial_y (y^2 \tilde Q ) ), \  \mathrm{Re} ( \var, y^2  \partial_b \tilde Q ),  \ 
\mathrm{Im} ( \var, y^2 \tilde Q ) ,
\]
we estimate by $\mathcal E^{\frac{1}{2}}(t)$ using the exponential decay of $\tilde{Q}$ and (\ref{closeness}). It remains to consider the nonlinear in $\var$ terms. These terms are treated exactly the same way than $A_k,  k=1, ..., 6$ in the step of conservation of energy 
in the section \ref{sectionconservation}.  This gives
\[
 | (R_1(\var),y^2 \Theta)| +   |(R_2(\var),y^2 \Sigma)| \leq \mathcal E(t).
\]
Finally, putting together all these considerations, we obtain the first estimate
\begin{equation} \label{lambda}
 \left | \frac{\lambda_s}{\lambda}+b \right | \leq \mathcal E ^{\frac{1}{2}}(t) \left (  \left | \frac{\lambda_s}{\lambda}+b \right | + \left | \frac{r_s}{\lambda}  \right | 
 + |\tilde{\gamma}_s  |  \right )+ \delta (\alpha^{\ast}) |b_s| +\mathcal E(t) + \Gamma_b^{1- C \eta} .
\end{equation}
\textbf{Estimate for $\frac{r_s}{\lambda}$.} Taking the inner product of (\ref{eq1}) with $y \Sigma$ and (\ref{eq2}) with $y \Theta$ and sum the two equalities, we obtain
\begin{eqnarray*}
& &  - \frac{\| \tilde Q  \|_{L^2}^2}{2}   \frac{r_s}{\lambda} =  \mathrm{Im} ( \lambda w  \partial_y \tilde{Q} , y \tilde{Q}) - \mathrm{Im} ( \var, y \tilde{\Psi}) + \frac{r_s}{\lambda}  \mathrm{Re} (\var, \partial_y (y \tilde Q  ))    \\
 & & - \tilde{\gamma}_s \mathrm{Im} ( \var, y \tilde{Q} ) -b_s \mathrm{Re} ( \var, y \partial_b \tilde{Q}) + \left ( \frac{\lambda_s}{\lambda}+b \right ) \mathrm{Re} (\var, \Lambda (y \tilde{Q} ))  \\
  & &  +\mathrm{Im} ( \lambda w \partial_y \var, y \tilde{Q} )  -(R_1(\var), y \Theta) + ( R_2 ( \var), y \Sigma). \\
\end{eqnarray*}
As before, we estimate each term of this relation. The same type of consideration than for $\lambda_s / \lambda + b$ yields the estimate
\begin{equation} \label{r}
 \left | \frac{r_s}{\lambda} \right | \leq \mathcal E^{\frac{1}{2}} (t) \left ( \left | \frac{\lambda_s}{\lambda}+b \right | + \left | \frac{r_s}{\lambda}\right | 
 + | \tilde{\gamma}_s  | +|b_s|    \right )  + \mathcal E(t) + \Gamma_b^{1-C \eta}  .
\end{equation}
\textbf{Estimate for $b_s$.} We take the inner product of (\ref{eq1}) with $- \Lambda  \Theta $ and (\ref{eq2}) with $ \Lambda \Sigma $ and sum the two equalities to obtain
\begin{eqnarray}
& & b_s \mathrm{Im} ( \partial_b \tilde{Q} , \Lambda \tilde{Q} )= 2  \mathrm{Re} ( \var,  \tilde{Q}- \tilde{\Psi} ) - \mathrm{Re} (  \tilde{\Psi} ,\Lambda \tilde{Q})+ \mathrm{Re} ( \lambda w  \partial_y \var,  \Lambda \tilde{Q})  \label{virial} \\
& & + b_s \mathrm{Im}(\var, \Lambda \partial_b \tilde{Q}) -\left ( \frac{\lambda_s}{\lambda}+b \right ) \mathrm{Im} ( \var, \Lambda^2 \tilde{Q})- \frac{r_s}{\lambda}   \mathrm{Im} ( \var, \partial_y \Lambda \tilde{Q} )  \nonumber \\
& & +\tilde{\gamma}_s \mathrm{Re}( \var, \Lambda \tilde{Q})   + (R_1( \var),\Lambda  \Sigma) + ( R_2(\var), \Lambda  \Theta) .  \nonumber 
\end{eqnarray}
All terms are treated the same way than before except the term $2  \mathrm{Re} ( \var, \tilde{Q} -\tilde{\Psi} ) $ which needs a special treatment. Using the conservation of energy written in (\ref{energy2})
\[
\frac{\lambda^2 E_0}{C}  = \sum_{k=0}^6 A_k,
\]
and the equality for $A_1$:
\[
 2  \mathrm{Re} ( \var, \tilde{\Psi} - \tilde{Q} ) = -\frac{2}{h(r(t))} A_1 + \mathcal O ( \Gamma_b),
\]
we may write
\[
 2  \mathrm{Re} ( \var, \tilde{\Psi} - \tilde{Q} ) = -\frac{2}{h(r(t))}  \left ( \frac{\lambda^2 E_0}{C}  - \sum_{k \neq 1} A_k \right )  + \mathcal O ( \Gamma_b).
\]
Therefore, using that $h(r(t)) \approx 1 >0$, the different estimates of the $A_k$ established in the proof of (\ref{energy}), and the equality 
\[
E( \tilde{Q}) =  O ( \Gamma_b^{1- C \eta}),
\]
which follows from (\ref{energyestimate}), we successively have 
\begin{eqnarray*}
  \left | 2  \mathrm{Re} ( \var, \tilde{\Psi} - \tilde{Q} ) \right | &\leq & \lambda^2 | E_0 | + E ( \tilde Q) + \mathcal E(t) + \Gamma_b \\
                                                                 & \leq & \lambda^2 |E_0| +\mathcal E(t) + \Gamma_b^{1-C \eta}  .
\end{eqnarray*}
Moreover, the approximation $\partial_b \tilde{Q_b} \sim -i y^2 /4 \tilde{Q_b}$ gives after integration by parts:
\begin{equation} \label{signvirial}
\mathrm{Im} ( \partial_b \tilde{Q} , \Lambda \tilde{Q} ) = \frac{1}{4} \int y^2 |\tilde{Q_b} |^2.
\end{equation}
We thus have 
\begin{equation} \label{b}
|b_s| \leq \mathcal E^{\frac{1}{2}}(t)  \left (  \left | \frac{\lambda_s}{\lambda}+b \right |+ \left | \frac{r_s}{\lambda} \right | + |b_s| +| \tilde{\gamma}_s | \right ) 
 +  \lambda^2 |E_0|+ \mathcal E(t) + \Gamma_b^{1- C \eta}. 
\end{equation}
\textbf{Estimate for $\tilde{\gamma}_s$.}
We take the inner product of (\ref{eq1}) with $\Lambda^2 \Theta$ and (\ref{eq2}) with $- \Lambda ^2 \Sigma$:
\begin{eqnarray*}
 & &  \tilde{\gamma} _s  \| \Lambda Q \|_{L^2}^2 - (\var_1, L_{+}(  \Lambda^2 Q)) =b_s \left ( \mathrm{Im} ( \partial_b \tilde Q, \Lambda ^2 \tilde Q) - \mathrm{Im} (\var, \partial_b \Lambda^2 \tilde Q )\right ) \\
 & & + \left ( \frac{ \lambda_s}{\lambda} +b \right ) \mathrm{Im} ( \Lambda ^2 \tilde Q, \Lambda \tilde Q)   + \mathrm{Re} ( \tilde{\Psi}_b, \Lambda ^2 \tilde Q) - \mathrm{Re} ( \lambda w \partial _y \tilde Q, \Lambda ^2 \tilde Q) \\
 & &  -  \frac{r_s}{ \lambda}  \mathrm{Im} ( \partial_y \tilde Q, \Lambda ^2 \tilde Q)+ \left ( \frac{ \lambda_s}{\lambda} +b \right ) \mathrm{Im} ( \var, \Lambda^3 \tilde Q ) + \tilde{\gamma}_s \mathrm{Re} ( \var, \Lambda ^2 \tilde Q) \\
 & &  +  \frac{r_s}{\lambda} \mathrm{Im} ( \var, \partial_y \Lambda^2 \tilde Q) - b \mathrm{Im} ( \var, \Lambda ^3 \tilde Q) +\left ( ( M_+( \var),\Lambda^2  \Sigma)- ( \var_1, L_+ \Lambda^2  Q) \right ) \\
 & & +( M_-(\var), \Lambda ^2  \Theta) - ( R_1( \var), \Lambda ^2  \Sigma)  -( R_2(\var), \Lambda ^2  \Theta) . 
\end{eqnarray*}
We use the previous estimates together with
\[
     \left | (M_{+} (\var), \Lambda ^2 \Sigma ) - ( \var_1, L_{+} \Lambda ^2 Q) \right |      +   \left | ( M_{-}( \var), \Lambda ^2 \Theta ) \right |          \leq \delta( \alpha^{\ast}) \mathcal E^{1/2}  (t)                         
\]
 to deduce
\begin{equation*}
  \left |   \tilde{\gamma} _s  \| \Lambda Q \|_{L^2}^2 - (\var_1, L_{+}(  \Lambda^2 Q)) \right | \leq C |b_s| + \delta ( \alpha^{\ast}) \left |\frac{\lambda_s}{\lambda} + b \right  |  + \delta( \alpha ^{\ast}) \left | \frac{r_s}{\lambda}  \right | 
\end{equation*}
\begin{flushright} 
   $\displaystyle{ + \left ( \left |\frac{\lambda_s}{\lambda} + b \right  | + |\tilde{\gamma}_s| + \left | \frac{r_s}{\lambda} \right | \right ) \mathcal E^{\frac{1}{2}}(t)+ \mathcal E^{\frac{1}{2}}(t) \delta( \alpha^{\ast}) + \Gamma_b^{1- C \eta}}$
\end{flushright}
and that we can also rewrite to make appear the term $\tilde{\gamma}_s \| \Lambda Q \|_{L^2}^2 - ( \var_1, L_{+} ( \Lambda^2 Q))$ as 
\begin{eqnarray} \label{gamma2}
 \left |   \tilde{\gamma} _s  \| \Lambda Q \|_{L^2}^2 - (\var_1, L_{+}(  \Lambda^2 Q)) \right | & \leq & C |b_s| +\delta( \alpha^{\ast})  \Big ( \left |\frac{\lambda_s}{\lambda} + b \right  | + \left | \frac{r_s}{\lambda} \right |   + \big (  \tilde{\gamma}_s   \| \Lambda Q \|_{L^2}^2 \\
 & & -( \var_1, L_+( \Lambda^2 Q) \big )  \Big ) + \delta( \alpha^{\ast}) \mathcal E^{\frac{1}{2}}(t) + \Gamma_b^{1- C \eta}. \nonumber
\end{eqnarray}

Now, summing (\ref{lambda}), (\ref{r}), (\ref{b}) and $\nu$(\ref{gamma2})  for $\nu$ small enough (to make the term $C |b_s|$ in (\ref{gamma2}) go in the left side) and using that $\mathcal E(t)$ goes to $0$ as $\alpha^{\ast}$ tends to $0$, we first have
\begin{equation} \label{esti10}
\left | \frac{\lambda_s}{\lambda}+b \right | + \left | \frac{r_s}{\lambda} \right | + |b_s|+ \left | \tilde{\gamma}_s \| \Lambda Q\|_{L^2}^2 -( \var_1, L_+( \Lambda^2 Q)) \right | 
\end{equation}
\begin{flushright}
 $ \displaystyle{ \leq \delta ( \alpha^{\ast}) \mathcal E ^{\frac{1}{2}}(t) + \Gamma_b^{1- C \eta}  + \lambda^2 | E_0|,}$
\end{flushright}
and this gives (\ref{gamma}). Next, by summing the estimates (\ref{lambda}), (\ref{r}), (\ref{b}) that do not involve the phase parameter $\tilde{\gamma}_s$ and injecting (\ref{gamma}) in this new estimate, we get the control of the other parameters (\ref{parameter}).
\end{proof}

\section{\texorpdfstring{$H^1$ virial estimate}{H1 virial estimate}} \label{virialloc}

In the spirit of \cite{MerRap2003}, \cite{Rap2006}, to derive dispersive properties of the rest $\var$, we would like to use the virial identity for the one dimensional Schr\"odinger equation. However, the classical virial identity is only defined in $\Sigma:=\{ u \in H^1, \int d(N,x)^2 |u|^2 < \infty \}$ (where $N$ is one of the pole) but a formal computation will show that we may extend this identity to $H^1$. The curvature of the manifold, which is reflected in the additional term $h'/h \partial_r u$ in the Laplace operator, is treated as a perturbation due to the smallness of the parameter $\lambda$. Let us give the computation in the case $u \in \Sigma$. In this case, we see $u$ as a function of $r$ and an almost solution of the Euclidean equation  
\[
i \partial_t u + \partial_{r}^2  u \sim - |u|^4 u, 
\] 
for which we apply the classical virial identity 
\[
 \frac{d}{dt} \mathrm{Im} \int \partial_r u \overline{u} r dr = 4 E_0.
\]
Next, we expand in $\var$ this relation and switch in $(s,y)$ variables to obtain 
\[
A_0+A_1+A_2 = 4 \lambda ^2 E_0,
\]
where 
\begin{eqnarray*}
A_0&=& \frac{d}{ds} r(s) \mathrm{Im} \int \partial_r  u \overline{u} + \mathrm{Im} \int y \partial_y  \tilde{Q_b} \overline{\tilde{Q_b}} , \\
A_1&=& 2 \frac{d}{ds} \mathrm{Im} (\Lambda \tilde{Q_b}, \var) , \\
A_2&= &\frac{d}{ds} \mathrm{Im} \int \partial_y \var \overline{ \var} .
\end{eqnarray*}
But $r \sim 1 $ and the momentum of $u(r)$ is almost constant since $u(r)$ is an almost solution of the one dimesional equation so that after integration by parts in the second term, we get $A_0= D b_s + \mathcal O (\Gamma_b ^{1-C \eta})$ for a constant $D>0$. $A_1$ is zero by orthogonality conditions. For the third term, we use the equation satisfied by $\var$ that we may see by the smallness of the parameters and removing the nonlinear part as 
\[
i \partial_s \var  + L  \var = \Psi.
\]
This gives 
\[
A_2 = H( \var, \var) - 2 \mathrm{ Re} \int \var \Lambda \overline{ \Psi} +  \mathcal O( \Gamma_b^{1- C \eta}).
\]
Summing $A_0,A_1,A_2$ and using the smallness of $\Psi$ and the coercive property of $H$ modulo some negative directions that we control by orthogonality conditions and conservation laws, we obtain the virial estimate:
\[
C b_s \geq \mathcal E(t) - \Gamma_b ^{1- C \eta}.
\]
Since this estimate is well defined for $u \in H^1$, we expect it to hold in the general case. This is stated in the following proposition.
\begin{pro}[$H^1$ \textbf{virial estimate}]
 There exist $C>0$ and $\delta>0$ such that for all $s \geq 0$,
\begin{equation} \label{virial2}
 b_s \geq \delta \mathcal E(t) - \Gamma_b^{1-C \eta}.
\end{equation}
\end{pro}

\begin{proof}
 We go back to the relation (\ref{virial}): 
\begin{eqnarray}
& & b_s \mathrm{Im} ( \partial_b \tilde{Q} , \Lambda \tilde{Q} )= 2  \mathrm{Re} ( \var,  \tilde{Q}- \tilde{\Psi} ) - \mathrm{Re} (  \tilde{\Psi} ,\Lambda \tilde{Q})+ \mathrm{Re} (\lambda w \partial_y \var,  \Lambda \tilde{Q})   \\
& & + b_s \mathrm{Im}(\var, \Lambda \partial_b \tilde{Q}) -\left ( \frac{\lambda_s}{\lambda}+b \right ) \mathrm{Im} ( \var, \Lambda^2 \tilde{Q})- \frac{r_s}{\lambda}   \mathrm{Im} ( \var, \partial_y \Lambda \tilde{Q} )  \nonumber \\
& & +\tilde{\gamma}_s \mathrm{Re}( \var, \Lambda \tilde{Q}) -  + (R_1( \var),\Lambda  \Sigma) + ( R_2(\var), \Lambda  \Theta) .  \nonumber 
\end{eqnarray}
To prove (\ref{virial2}), we need in particular to extract the quadratic in $\var$ term in the right hand side of the last equality that we write as 
\[
 b_s \mathrm{Im} ( \partial_b \tilde Q , \Lambda \tilde Q) = F+G,
\]
where 
\[
 F=2  \mathrm{Re} ( \var,  \tilde{Q} -\tilde{\Psi})+ \left ( (R_1( \var),\Lambda \Sigma) + ( R_2(\var), \Lambda \Theta) \right ) =F_1+F_2.
\]
First, the $G$ term is easily estimated using the control of the parameters proved in the previous section and estimates
\[ 
 |(\var, \partial^{\alpha}  \tilde Q)|+|(\var, \partial_b \Lambda \tilde Q)|+|( \tilde{\Psi} ,\Lambda \tilde{Q})| \leq \mathcal E(t)+ \Gamma_b^{1- C \eta} .
\]
This gives that $G$ is negligible:
\[
 |G| \leq \delta( \alpha^{\ast}) \mathcal E(t) + \Gamma_b^{1-C \eta} .
\]
Now we focus on the $F$ term where we want to extract the quadratic in $\var$ term. Let us first study the $F_1$ term. We use the conservation of energy that we can write according to (\ref{energy2}):
\[
 \frac{ \lambda^2 E_0}{C} = \sum_{k=0}^6 A_k.
\]
Remark that we have already seen
\[
 A_0= \mathcal O( \Gamma_b^{1- C \eta}),
\]
\[
 A_1=-h(r(t)) \mathrm{Re}( \var, \tilde Q - \tilde \Psi) + \mathcal O( \Gamma_b^{1- C \eta}),
\]
and
\begin{eqnarray*}
 A_2+A_3+A_4+A_5+A_6 &=& \frac{1}{2} \int | \partial_y \var |^2 \mu(y) dy -\frac{5 }{2} \int Q^4 \var_1^2 dy - \frac{1}{2} \int Q^4 \var_2^2 \\
                     & & + \mathcal O( \Gamma_b^{1- C \eta})+ \delta(\alpha^{\ast}) \mathcal E(t).
\end{eqnarray*}
Therefore, we may write $F_1$ as
\[
 F_1=\frac{-2}{ h(r(t))} \left ( \frac{\lambda^2 E_0}{C}  - \frac{1}{2} \int | \partial_y \var |^2 \mu(y) dy + \frac{5}{2} \int Q^4 \var_1^2+ \frac{1}{2} \int Q^4 \var_2^2 
\right ) + \delta( \alpha^{\ast})\mathcal E (t)+ \Gamma_b ^{1- C \eta} .
\]
Next, by expanding $R_1(\var)$ and $R_2(\var)$ and replacing $\tilde{Q_b}$ by $Q$ by (\ref{closeness}), we obtain 
\[
 F_2=(R_1(\var), \Lambda \Sigma) +(R_2(\var), \Lambda \Theta)  =(\var_1^2, 5 Q^4 + 10 y Q^3 \partial_y Q)+ (\var_2^2, Q^4 +2 yQ^3 \partial_y Q)  + \delta( \alpha^{\ast}) \mathcal E(t).
\]
We sum and obtain 
\begin{multline}
 F = \int | \partial_y \var |^2  \frac{\mu(y)}{ h(r(t))}  dy +\left ( ( \mathcal L_1 \var_1, \var_1)+ ( \mathcal L_2 \var_2, \var_2) - \int | \partial_y \var |^2 \right ) - 2 \frac{\lambda ^2 E_0}{ C h(r(t))}   \\
+ \mathcal O (\Gamma_b ^{1-C \eta}) + \delta( \alpha^{\ast}) \mathcal E(t)  .
\end{multline}
Remark that since the term $|\partial_y \var|^2$ is not localized in space, we cannot make the approximation
\[
  \mu(y) \sim h(r(t))
\]
with a good error. Moreover, we will need to control from below the quantity $( \mathcal L_1 \var_1, \var_1)+ ( \mathcal L_2 \var_2, \var_2) $ by $\mathcal E(t)$ plus some inner products. For this, we will need to localize $\var$ to reintroduce the measure $\mu(y)$. 
These two facts  suggest to introduce $\phi_3$ a cut-off such that
\[
 \phi_3 (t,y)=
\left \{
\begin{array}{rcl}
 1 & \textrm{if}& |y| \leq \Gamma_b ^{-5} \\
0 & \textrm{if} & |y| \geq 2 \Gamma_b ^{-5}
\end{array}
\right . ,
\]
with $0 \leq \phi_3 \leq 1$ and the bound $| \partial _y \phi_3 | \leq \Gamma_b^5$. Then we write
\[
 \int | \partial_y \var  |^2 \mu(y) dy = \int | \partial_y \var |^2 \phi_3^2  \mu(y) dy+\int | \partial_y \var  |^2 (1- \phi_3^2) \mu(y) dy.
\]
The second term will not be a problem since it is positive. For the first term, we write that the quantity 
\[
 A:=  \left | \int \phi_3^2   | \partial_y \var |^2  \mu(y) dy - \int | \partial_y (\phi_3 \var) |^2 h(r(t)) dy  \right | 
\]
is bounded by
\begin{eqnarray*}
 A &\leq&  \left | \int \phi_3^2 | \partial_y \var |^2 h(r(t)) dy - \int | \partial_y ( \phi_3 \var ) |^2 h(r(t)) dy \right |  \\ 
                   & & + \left | \int \phi_3^2   | \partial_y \var |^2  \mu(y) dy - \int \phi_3^2   | \partial_y \var |^2   h(r(t)) dy \right |   \\
                                                                     & \leq & \left | \int \phi_3^2   | \partial_y \var |^2  \left(h(\lambda(t) y +r(t))-h(r(t))  \right) dy  \right |  + C \| \partial_y \phi_3 \|_{L^{\infty}}^2 \| \var \|_{H^1 ( - \Gamma_b^{-5}, \Gamma_b^{-5} )}^2  \\
                                                                   & \leq & \delta( \alpha^{\ast}) \mathcal E(t)+\Gamma_b
\end{eqnarray*}
since when $|y|  \leq \Gamma_b^{-5},\  | h(\lambda(t)y+r(t))-h(r(t))| \leq \delta (  \alpha^{\ast} )$ according to \textbf{B5} and $ 1/2 \leq \mu(y)$. At this point, for the $F$ term, we obtain
\[ 
F \geq \int | \partial_y (\phi_3\var)  |^2 +\left ( ( \mathcal L_1 \var_1, \var_1)+ ( \mathcal L_2 \var_2, \var_2) - \int | \partial_y \var |^2 \right ) -\frac {\lambda^2 E_0}{h(r(t))}- \delta( \alpha^{\ast}) \mathcal E(t) - \Gamma_b^{1- C \eta}.
\]
We may also localize the expression in parenthesis with a good error. Indeed, let us give the argument for the term $y Q^3 \partial_y Q \var_1^2$:
\begin{eqnarray*}
 \left | \int y Q^3 (\partial_y Q) \var_1^2 - \int y Q^3 (\partial_y Q) (\phi_3 \var_1 )^2 \right | &\leq & \int _{ |y| \geq \Gamma_b^{-5}} |y| Q^3 |\partial_y Q|  \var_1^2  \\ 
                                                       & \leq &  \Gamma_b \mathcal E(t) \leq \delta(\alpha ^{\ast}) \mathcal E(t).
\end{eqnarray*}
Thus with \textbf{A4}, we have:
\[
 F \geq ( \mathcal L_1 ( \phi_3 \var_1), \phi_3 \var_1 ) + ( \mathcal L_2 ( \phi_3 \var_2), \phi_3 \var_2) - \delta( \alpha^{\ast})\mathcal E(t)  - \Gamma_b^{1-C \eta}.
\]
Here, we use the spectral property (\ref{spectral}) to obtain a lower bound on $H( \phi_3 \var, \phi_3 \var) $:

\begin{multline} \label{innerprod}
 H( \phi_3 \var, \phi_3 \var)   \geq  \delta \left ( \int | \partial_y (\phi_3 \var )|^2 dy + \int | \phi_3 \var |^2 e^{- |y|}dy \right )-\frac{1}{\delta}  \Big (  (\phi_3 \var_1, Q)^2 +
(\phi_3 \var_1, y^2 Q)^2  \\ 
+(\phi_3 \var_1, yQ)^2  
  +( \phi_3 \var_2, \Lambda Q)^2+(\phi_3 \var_2, \Lambda^2 Q)^2 + (\phi_3 \var_2, \partial_y Q) \Big ). 
\end{multline}
First, by the property of the support of $\phi_3$, we have
\[
 \int | \partial_y (\phi_3 \var )|^2 dy + \int | \phi_3 \var |^2 e^{- |y|}dy \geq  \mathcal E(t) - \Gamma_b.
\]
Secondly, using orthogonality conditions and the estimates (\ref{mass}), (\ref{energy}), (\ref{momentum}) induced by conservation laws, we can bound each of the six inner products in (\ref{innerprod})  by $\delta( \alpha^{\ast}) \mathcal E(t)$. Indeed, let us give the argument for the 
inner products $(\phi_3 \var_1, Q)$ and $(\phi_3 \var_1, y^2 Q)$. For the first, we use the conservation of energy (\ref{energy}) to write
\begin{eqnarray*}
 |(\phi_3 \var_1, Q)| & \leq&  | (\var_1, Q)| +| ((1- \phi_3) \var_1, Q)|   \\
                                   & \leq & \leq  | (\var_1, Q-  \Sigma)| + | ( \var_1,  \Sigma )| + \Gamma_b \\ 
                       & \leq & \delta( \alpha^{\ast}) \mathcal E^{\frac{1}{2}}(t)+ \delta(\alpha^{\ast}) \mathcal E^{\frac{1}{2}} (t) + \Gamma_b^{1- C \eta}  \\
                      & \leq &  \delta(\alpha^{\ast}) \mathcal E^{\frac{1}{2}} (t) + \Gamma_b^{1- C \eta}  .
\end{eqnarray*}
For the second, we use the orthogonality condition (\ref{orth1}). This yields
\begin{eqnarray*}
| (\phi_3 \var_1, y^2 Q)| &\leq&  | ( 1- \phi_3 ) \var_1 , y^2 Q) | + | ( \var_1, y^2 Q)| \\
                          &\leq  &  \delta( \alpha ^{\ast}) \mathcal E^{\frac{1}{2}}(t)+| ( \var_1, y^2 (Q-  \Sigma))|+ |(\var_1, y^2 \Sigma)|    \\
                          & \leq & \delta( \alpha ^{\ast}) \mathcal E^{\frac{1}{2}}(t) +\delta( \alpha ^{\ast}) \mathcal E^{\frac{1}{2}}(t) + |(\var_2, y^2 \Theta)| \\
                         & \leq & \delta( \alpha ^{\ast}) \mathcal E^{\frac{1}{2}}(t).
\end{eqnarray*}
All these considerations give 
\[
 H( \phi_3 \var, \phi_3 \var)  \geq  \delta \mathcal E(t), 
\]
and then
\[
 F \geq \delta \mathcal E(t) - \Gamma_b^{1- C \eta} .
\]
We obtain (\ref{virial}) by summing $F$ and $G$ for $\alpha^{\ast}$ small enough and using the sign of the quantity $\mathrm{Im} ( \partial_b \tilde{Q} , \Lambda \tilde{Q} )$ proved in (\ref{signvirial}).

\end{proof}

\section{Estimates on geometrical parameters} \label{sectionestimates}
In this section, we integrate estimates proved in the two previous sections. This will give us informations on $b, \lambda$ and $r$. Using the virial estimate:
\[ \left | \frac{\lambda_s}{\lambda} +b \right | + |b_s| \leq \Gamma_{b}^{1/2} , \]
we get 
\[
\frac{d}{ds} ( \lambda^2 e^{5 \pi / b} ) = 2 \lambda^2 e ^{ 5 \pi /b} \left ( \frac{\lambda_s}{\lambda} +b -b - \frac{5 \pi b_s}{ 2 b^2} \right ) \leq - b \lambda^2 e^{ \frac{5 \pi}{b} } \leq 0.
\]
Therefore, $\lambda^2 e^{ 5 \pi / b }$ is a decreasing function of $s$ and so of $t$, which gives using \textbf{A5}:
\[
 \frac{ \lambda^2(t) | E_0| }{\Gamma_b^4(t)} \leq |E_0| \lambda ^2(t) e^{5 \pi / b(t)}  \leq |E_0| \lambda^2(0) e^{ 5 \pi / b(0)} <1
\]
and this proves \textbf{C4}.

\medskip

\textbf{Estimate for $b$.} 
From the virial estimate (\ref{virial2}), we obtain a differential inequality 
\[
 b_s \geq - \Gamma_b.
\]
that we solve by dividing by $b^2$ and remarking that it implies 
\[
\frac{d}{ds} \left ( e^{ 3 \pi / 4b} \right ) \leq 1.
\]
Integrating this in time and using our choice for $s_0$ (\ref{s0}), we get 
\[
 e^{ 3 \pi / 4b}  \leq s,
\]
and this gives the lower bound on $b$:
\begin{equation} \label{boundb}
b \geq \frac{3 \pi}{4 \log s}.
\end{equation}

\medskip

\textbf{Estimate for $\lambda$.}  Using the estimate (\ref{parameter}) on the parameter $\lambda_s / (\lambda + b)$ and the control on $\Gamma_b$, we get
\[
 \left | \frac{\lambda_s}{\lambda} +b  \right | \leq \Gamma_b^{1/2}  \leq \frac{b}{3} .
\]
We deduce from the the lower bound on $b$
\[
 - \frac{ \lambda_s}{\lambda}  \geq \frac{2b}{3} \geq \frac{\pi}{ 2 \log s}.
\]
We integrate this inequality to deduce 
\[
 - \log \lambda(s) \geq - \log \lambda (s_0) + \int _{s_0}^s \frac{ \pi}{ 2 \log  \theta} d \theta,
\]
and thus 
\begin{eqnarray*}
 - \log \lambda(s) & \geq  & - \log \lambda (s_0) + \frac{ \pi}{2} \int _{s_0}^s  \frac{\log \theta - 1}{ \log ^2 \theta} d \theta  \\
                   & \geq & - \log \lambda (s_0) + \frac{\pi}{2}  \left [ \frac{ \theta}{ \log \theta} \right ]_{s_0} ^s \\
                   & \geq & - \log  \lambda ( s_0) - \frac{ \pi }{2} \frac{s_0}{ \log s_0} + \frac{\pi}{2} \frac{s}{ \log s}.
\end{eqnarray*}
Using \textbf{A5} and our choice of $s_0$ (which is large if $\alpha^{\ast}$ is small enough), this gives
\begin{equation} \label{boundonlambda1}
 - \log \lambda(s)  \geq -\frac{1}{2} \log \lambda (s_0) + \frac{ \pi}{2} \frac{s}{ \log  s}
\end{equation}
and this prove for all $s \in [s_0,s_1)$,
\begin{equation} \label{estimatelambda}
 \lambda (s) \leq \sqrt{\lambda_0} e^{- \frac{ \pi}{2} \frac{s}{ \log s}}.
\end{equation}
The lower bound on $b$ (\ref{boundb}) implies in particular
\[
 b(s) \geq \frac{ \pi}{5} \frac{1}{ \log s - \log \log s}
\]
which is equivalent to
\[
 \frac{s}{ \log s} \geq  \exp { \frac{\pi}{ 5b(s)} }.
\]
This last inequality together with (\ref{estimatelambda}) yield:
\[
 \lambda \leq \sqrt{ \lambda_0} \mathrm{exp} \left ( - \frac{ \pi}{2} \mathrm{exp} \left ( \frac{ \pi}{ 5 b} \right )  \right )
\]
and therefore from the smallness of $\lambda_0$ given by \textbf{A5}, we obtain the upper bound on $\lambda$ \textbf{C5}.
\textbf{Estimate on $r$.} The estimate (\ref{parameter}) on $r_s/ \lambda$ shows that if $\alpha ^{\ast}$ is small,
\[
  \left | \frac{r_{\lambda}}{ \lambda} \right | \leq 1 .
\]
And thus, using the upper bound on $\lambda$ (\ref{parameter}) and the estimate on $\lambda_0$ \textbf{A5}, we get
\[
  | r(s)- r_0| \leq \int_{s_0}^s |r_s| ds \leq \int_{s_0}^s \lambda( \theta) d \theta \leq \sqrt{ \lambda_0} \int_2 ^{+ \infty} \exp \left ( -  \frac{\pi}{2} \frac{ \theta}{ \log \theta}  \right ) d \theta \leq \alpha^{\ast}.
\]
And by \textbf{A1},
\[ 
|r(t)-1 | \leq \alpha^{\ast},
\]
and we obtain \textbf{C1}.

\section{Smallness of the localized momentum}

The goal of this section is to show \textbf{C4'}: 
\[
 \lambda (t) \left | \mathrm{Im} \left ( \int \overline{u}(t)  \nabla   \psi \cdot \nabla u(t) dx \right ) \right | \leq \Gamma_b^4.
\]
We first multiply the equation (\ref{nls}) by $\frac{1}{2}  \Delta \psi \overline{u} + \nabla \psi \cdot \nabla \overline{u}$ and take the real part and then after several integration by parts
\[
\frac{1}{2} \partial_t \mathrm{Im} \int \overline{u} \ \nabla \psi \cdot \nabla u  =  \int | \partial_ r u|^2 \partial_r^2 \psi h dr d \theta - \frac{1}{4} \int |u|^2 \Delta^2 \psi + \frac{1}{3} \int |u|^6 \Delta \psi .
\]
We use the boundedness of $\psi$ and its derivative to control the two first terms in the right hand side of the above equality and the Sobolev type inequality from Proposition \ref{sobolevradial} to control the $L^6$ norm of $u$. This gives 
\begin{eqnarray*}
  \left | \frac{1}{2} \partial_t \mathrm{Im} \int \overline{u} \ \nabla \psi \cdot \nabla u \right | &\leq &\| \nabla u \|_{L^2}^2 + \|u\|_{L^2}^2 + \| u\|_{H^{\frac{1}{3}}} ^6 \\
                                                                                                       & \leq & 2 \| u\|_{H^1}^2 + \| u\|_{L^2} ^3 \| u\| _{H^1}^{2} \\                             
                                                                                                   & \leq &  C \|u\|_ {H^1} ^2 .
\end{eqnarray*}
Thus using the decomposition of $u$ in term of $\tilde{Q}$ and $\var$ to write $\| \nabla u \| _ {L^2} \leq 1/ \lambda $, we deduce
\begin{equation} \label{moment}
 \left | \frac{1}{2} \partial_t \mathrm{Im} \int \overline{u} \ \nabla \psi \cdot \nabla u \right | \leq  \frac{1}{ \lambda ^2 (t)}.
\end{equation}
We then integrate (\ref{moment}) in time between $0$ and $t$ and use 
\[
 \int_0^t \frac{d \tau}{ \lambda^2 (\tau)} = \int_{s_0} ^s d \theta \leq s ,
\]
to obtain for all $t \in [0, t_1)$,
\begin{equation} \label{ineg2}
 \lambda (t) \left | \mathrm{Im} \int  \nabla \psi \cdot \nabla u \overline{u} \right | \leq \lambda (t) \left | \mathrm{Im} \int  \nabla \psi \cdot \nabla u_0 \overline{u_0} \right |+ \lambda (t) s(t).
\end{equation}
But, $\lambda (t) e^{6 \pi / b(t)} \leq \lambda_0  e^{6 \pi / b_0} $. Indeed, 
\begin{eqnarray*}
 \partial_s ( \lambda   e^{ 6 \pi / b} ) = \lambda e^{ 6 \pi / b} \left ( \left ( \frac{\lambda_s }{\lambda} + b \right ) - b - \frac{6 \pi b_s}{b^2}  \right )
\end{eqnarray*}
and the term in parenthese is negative (equivalent to $-b$) since by (\ref{parameter}) and \textbf{B3}:
\[
 \left | \frac{\lambda_s}{\lambda}+ b \right | +  \left | \frac{b_s}{b^2} \right | \leq \Gamma_b^{\frac{1}{2}} .
\]
We deduce using the estimate of the localized momentum at time $t=0$ \textbf{A4}:
\begin{eqnarray}
  \frac{\lambda (t) \left |\mathrm{Im}  \int  \nabla \psi \cdot \nabla u_0 \overline{u_0} \right |}{\Gamma_b ^5}  &\leq & \lambda (t) \left | \mathrm{ Im} \int \nabla \psi \cdot \nabla u_0  \overline{u_0} \right | e^\frac{ 6 \pi }{b(t)} \nonumber \\
                                                                                                  & \leq & \lambda _0 \left | \mathrm{ Im} \int \nabla \psi \cdot \nabla u_0  \overline{u_0} \right | e ^\frac{6 \pi}{b_0} <1. \label{ineg} 
\end{eqnarray}
Moreover from (\ref{boundb}) and (\ref{boundonlambda1}), we may show
\[
 s(t) \lambda (t)  \leq \Gamma_b ^5 .
\]
Injecting the last inequality together with (\ref{ineg}) into (\ref{ineg2}), we get \textbf{C4'}.

\section{Refined Virial estimate}
The virial estimate (\ref{virial2}) allowed us to get the bound (\ref{boundb}) on $b(s)$
\[
 b(s) \geq \frac{ 3 \pi}{4 \log s}.
\]
With this estimate, we will deduce the log log upper bound on the blow up rate
\[
\| \nabla u(t) \|_{L^2} \leq C \left ( \frac{ \log | \log (T-t) | } {T-t} \right )^{\frac{1}{2}}.
\]
If we want to prove the log log lower bound on the blow up rate, we need to have the converse inequality in (\ref{boundb}) and this requires to improve the virial estimate (\ref{virial2}). This refinement consists in finding an expansion at first order of the rest $\var$ and this term will be $\zeta_b$ or more precisely a troncated version of $\zeta_b$. Therefore, we choose a suitable parameter $A>0$ such that in the area $|y| \leq A$, the  radiation is close to $\var$ and then  get a virial type estimate for the new variable
\[
\tilde{ \var} = \var- \tilde{\zeta}_b,
\]
where $\tilde{\zeta}_b$ is the troncated radiation:
\[
 \tilde{\zeta} _b = \chi_A \zeta _b  = \tilde{\zeta} _{\mathrm {Re}} +i \tilde{\zeta}_{\mathrm{Im}}.
\]
and where  $\chi_A$ is a radial cut-off localized between $0$ and $2A$: $\chi_A(y)=\chi( y /A)$ with
\[
\chi(y)=
\left \{
 \begin{array}{rcl}
  1 & \textrm{ if } & |y| \leq 1 \\
  0 & \textrm{ if } & |y| \geq 2 \\
 \end{array}
\right .
.
\]
We choose $A$ as
\begin{equation} \label{choicea}
 A=e^{ \frac{a}{ \pi b(t)}},
\end{equation}
where $a>0$ is a small constant to be chosen later.  Using that $\chi_A (y) =1$ if $ |y| \leq 2/ b$, we deduce the equation satisfied by the truncated radiation $\tilde{\zeta}_b$:
\[
 \partial_y^2 \tilde {\zeta} _b- \tilde{\zeta}_b+ ib \Lambda \tilde{\zeta}_b = \Psi_b + F,
\]
where 
\[
 F=( \partial_y^2 \chi_A) \zeta_b + 2 (\partial_y \chi_A)( \partial_y \zeta_b) + i by (\partial _y \chi_A)  \zeta_b.
\]
Note that $\mathrm{Supp} \ F \subset \{ y, \  A \leq |y| \leq 2A \}$. Moreover, it is not difficult to see from Lemma \ref{lemzeta} that $\tilde{\zeta}_b$ satisfies 
\begin{equation} \label{zetatilde3}
 \int | \partial_y \tilde{\zeta}_b |^2 \leq \Gamma_b^{1- C \eta}, \qquad \int | \tilde{\zeta} _b|^2  \leq \Gamma_b^{1- C \eta}.
\end{equation}

 As for the first virial estimate, we may predict what kind of estimate we may hope for $\tilde{\var}$. Indeed, the equation for $\tilde{\var}$ is essentially:
\[
i \partial_s \tilde{\var} + L \tilde{\var} = F,
\]
so that by rewriting the virial identity if it is defined, we obtain roughly:
\[
b_s \geq H (\tilde{\var}, \tilde{\var}) + \mathrm{Re}(  \tilde{\var}, \Lambda \tilde{\zeta}) + \mathrm{smaller \ terms}.
\]
In particular, we will need to estimate the inner product $(  \tilde{\var}, \Lambda \tilde{\zeta})$. Let us now state the precise virial estimate for the variable $\tilde{\var}$.
\begin{lem}[\textbf{Refined virial estimate}] \label{refined}
 There exist constants $ \delta_1>0, c>0$ such that the following holds. There exist $\eta ^{\ast}, a^{\ast} >0$ such that for all $\eta \in (0, \eta^{\ast}), a \in (0, a^{\ast})$, there exists $b^{\ast}(\eta^{\ast}, a^{\ast})>0$ such that for all $ |b| \leq b^{\ast}(\eta^{\ast}, a^{\ast})$, and for all $s \in [s_0, s_1)$, 
\begin{equation} \label{refined2}
 \frac{d}{ds} f_1(s) \geq \delta_1  \tilde {\mathcal  E} (s) + c \Gamma_{b(s)} - \frac{1}{\delta_1} \int_{A \leq |y| \leq 2A}  | \var | ^2, 
\end{equation}
where $f_1(s) \sim C b$:
\begin{equation}
 f_1(s) = \frac{b}{4} \| y \tilde{Q}_b \|_{L^2}^2 + \frac{1}{2} \mathrm{Im} \left ( \int   y \partial_y \tilde{\zeta}  \overline{\tilde{\zeta}}  \right ) + ( \var_2, \Lambda \tilde{\zeta}_{\mathrm{re}}) - ( \var_1, \Lambda \tilde{\zeta}_{\mathrm{Im}}),
\end{equation}
and
\[
 \tilde {\mathcal E}(t)=\int | \partial_y \tilde {\var}(t,y) |^2 \mu(y) dy + \int_{|y| \leq \frac{10}{b(t)}} | \tilde{\var}(t,y) |^2e^{ - |y|} dy.
\]
\end{lem}
\begin{proof}

In a first step, as for the first virial estimate, the proof consists in multiplying the $\var$-equation by some suitable quantities; namely (\ref{eq1}) by $- \Lambda( \Theta + \tilde{\zeta}_{\mathrm{Im}})$ and (\ref{eq2}) by $\Lambda ( \Sigma + \tilde{\zeta}_{\mathrm{Re}})$ and summing. We do not rewrite this equality to avoid surcharging the text. This relation is essentially the same as the one found to prove the log log lower bound for the $L^2$-critical equation (see this relation in \cite{MerRap2006}) plus additionnal terms induced by the functions $w$ and $\mu$. Among these terms, there is only one which is non localized and which enter in the quadratic form $H$. The others are localized and are controled using (\ref{zetatilde3}) and for $|y| \leq 10/b,\  \mu(y) \sim 1$ and $w(y) \leq C$. After estimating these terms and controlling the negative directions of $H$ as in the proof of the first virial estimate, we obtain 
\begin{equation}
 \frac{d}{ds} f_1(s) \geq \delta_1 \left ( \int | \partial_y \tilde {\var}  |^2  \mu(y) dy + \int _{ |y| \leq \frac{10}{b}} | \tilde{\var}|^2 e^{ - |y|} dy \right ) - \Gamma_b ^2 +\mathrm{Re}(\tilde{\var}, \Lambda F).
\end{equation}
We have to estimate the inner product $\mathrm{Re}( \tilde{\var}, \Lambda F)$. To deduce the first virial estimate, we have used the bound
\[
  \mathrm{Re} ( \var, \Lambda \Psi )  \geq - \Gamma_b^{1- C \eta} .
\]
Here we may obtain a better estimate by showing for some $c>0$,
\begin{equation} \label{flux}
  \mathrm{Re} ( \tilde{\var}, \Lambda  F )  \geq c \Gamma_b - \frac{1}{c}  \int _{A \leq |y| \leq 2A } |\var|^2  .
\end{equation}
The proof of (\ref{flux}) does not make appear the measure $\mu$ and therefore is the same than the Euclidean $L^2$-critical case \cite{MerRap2006}. Let us recall quickly the argument, for details see \cite{MerRap2006}. It consists in splitting
\begin{equation} \label{split}
 \mathrm{Re} ( \tilde{\var}, \Lambda  F ) = \mathrm{Re} ( \var, \Lambda  F ) - \mathrm{Re} ( \tilde{\zeta}, \Lambda  F ).
\end{equation}
For the first term, we use Cauchy-Schwarz and the bound
\[
\int_A^{2A} | \Lambda F | ^2 \leq \Gamma_b
\]
consequence of properties of $\zeta_b$ (Lemma \ref{lemzeta}) to deduce for all $\delta_2 >0$,
\[
 | \mathrm{Re} ( \var, \Lambda  F ) | \leq \Gamma_b ^{\frac{1}{2} } \left ( \int_{ A \leq |y| \leq 2A}  | \var |^2 \right )^{\frac{1}{2}} \leq \delta_2 \Gamma_b + \frac{1}{\delta_2} \int _{ A \leq |y| \leq 2A}  | \var |^2 .
\]
The treatment of the second term consists in writing $F$ for $ A \leq  |y| \leq 2A$ as 
\[
 F=\partial_y^2 \tilde{\zeta} - \tilde{\zeta}+ ib \Lambda \tilde{\zeta} ,
\]
and using polar coordonates and integration by parts. Then from (\ref{zetatilde3}): 
\[
- \mathrm{Re} ( \tilde{\zeta}, \Lambda  F ) \geq c \Gamma_b.
\]
Finally, by summing and taking $\delta_2$ small enough, we obtain the desired estimate:
\[
\frac{d}{ds} f_1(s) \geq \delta_1 \left ( \int | \partial_y \tilde {\var}  |^2  \mu(y) dy + \int _{ |y| \leq \frac{10}{b}} | \tilde{\var}|^2 e^{ - |y|} \right ) + c \Gamma_b  - \frac{1}{c} \int_{ A \leq |y| \leq 2A}  | \var |^2.
\]
\end{proof}

Remark that $f_1(s) \sim Cb$ so that the new virial estimate is almost an estimate of the type:
\[
 b_s \geq \delta_1 \tilde{\mathcal E} (t) + c \Gamma_b -\frac{1}{\delta} \int_{ A \leq |y| \leq 2A} | \var |^2,
\]
and this inequality is an improvement of the first virial estimate if we are able to have a good estimate on the localized $L^2$-norm of $\var$. We denote by $\phi_4$ a smooth, increasing and radial cut-off satisfying $0 \leq \phi_4 \leq 1$ with $\phi_4'(y) \geq 1/2$ if $y \in [1,2]$ and:
\[
\phi_4(y)=
 \left \{ 
\begin{array}{ccl}
 0 & \textrm{if} & 0 \leq y \leq 1/2 \\
1 & \textrm{if} & y \geq 3 \\
\end{array}
\right . .
\]
\begin{lem}[\textbf{Localized $L^2$-norm of $\var$}] For all $s \in [s_0, s_1)$, 
\begin{equation} \label{dispersion}
 \frac{d}{ds} \left ( \frac{1}{r(s)} \int \phi_4 \left (  \frac{y}{A} \right )| \var|^2 \mu(y) dy \right ) \geq \frac{b}{8} \int _{A \leq |y| \leq 2A} | \var|^2 - \Gamma_{b} ^{\frac{a}{2}} \int | \partial_y \var |^2 \mu(y)dy  - \Gamma_{b}^2.
\end{equation}
\end{lem}

\begin{proof}
Let us denote by $L$ the left hand side of the above inequality. Then 
\[
 L= - \frac{r_s}{r(s)^2} \int \phi_4 (\frac{y}{A}) | \var |^2 \mu(y) dy + \frac{1}{r(s)} \frac{d}{ds} \left ( \int \phi_4(\frac{y}{A}) |\var| ^2 \mu(y) dy \right ).
\]
Using that $u$ satisfies the equation (\ref{nls}), we have after computations
\begin{eqnarray*}
 \frac{d}{dt} \int \phi_4 \left (  \frac{r - r(t) }{\lambda A} \right )| u|^2 dx & =& -\frac{1}{ \lambda ^2 A} \int |u|^2 \phi_4'\left (\frac{r- r(t)}{\lambda A} \right ) \left ( \frac{r_s}{\lambda} + \frac{r-r(t)}{\lambda} \left( \frac{\lambda_s}{\lambda} + \frac{A_s}{A} \right ) \right ) h dr d \theta \\
   & & + \frac{2 }{ \lambda A} \mathrm{Im} \int u \partial_r \overline{u} \phi_4' \left ( \frac{r-r(t)}{\lambda A} \right ) h dr d\theta.
\end{eqnarray*}
Remark that with the definition of $A$ and the support of $\tilde{Q}_b$, 
\[
\tilde{Q}_b \left ( \frac{r-r(t)}{\lambda (t) }   \right ) =0 \quad \textrm{if} \quad \frac{ r - r(t)}{A \lambda (t)} \geq \frac{1}{2},
\]
so that we can replace $u$ by
\[
\frac{1}{\sqrt{\lambda}} \var \left (t, \frac{r-r(t)}{\lambda} \right ) e^{i \gamma}
\]
in the above formula and coming back to the $(s,y)$ variable, this gives after writing $\frac{\lambda_s}{\lambda}=(\frac{\lambda_s}{\lambda} + b) -b $,
\begin{eqnarray*}
 \frac{d}{ds} \int \phi_4 ( \frac{y}{A})  | \var|^2 \mu(y) dy & =& -\frac{1}{ A} \int | \var |^2 \phi_4'(\frac{y}{A}) \left ( \frac{r_s}{\lambda} + y \left ( \frac{\lambda_s}{\lambda} +b + \frac{A_s}{A} \right ) \right ) \mu(y) dy  \\
  & & + \frac{2}{A } \mathrm{Im} \int \var \partial_y \overline{\var} \phi_4'( \frac{y}{A}) \mu(y) dy + b \int | \var |^2 \phi_4' ( \frac{y}{A}) \frac{y}{A} \mu(y) dy \\
  & & =B_1+B_2+B_3.
\end{eqnarray*}
Using estimates on parameters (\ref{parameter}), (\ref{gamma}), and $|A_s/A| \leq |b_s  /b^2| \leq \Gamma_b^a$, we have,
\[
| B_1| \leq \Gamma_b ^{\frac{a}{2}} \int  \phi_4' (\frac{y}{A}) | \var |^2 \mu(y) dy .
\] 
For $B_2$, we write
\begin{eqnarray*}
| B_2| &\leq & \frac{1}{A} \left ( \int \phi_4' (\frac{y}{A}) | \var |^2\mu(y) dy + \int | \partial_y \var |^2 \phi_4' (\frac{y}{A}) | \mu(y) dy \right ) \\
      &\leq  & \Gamma_b ^{\frac{a}{2}} \left ( \int \phi_4' (\frac{y}{A}) | \var |^2  \mu(y) dy + \int | \partial_y \var  |^2 \mu(y) dy \right )  \\
      &  \leq & \Gamma_b^{a/2} \int \phi_4' (\frac{y}{A}) | \var|^2 \mu(y) dy + \Gamma_b^{a/2} \int | \partial_y \var |^2 \mu(y) dy.
\end{eqnarray*}
Remark that with our choice of $\phi_4$, $y/A \geq 1/2$ and this provides a lower bound for $B_3$:
\[
B_3 \geq \frac{b}{2} \int \phi_4'(\frac{y}{A})      |\var| ^2     \mu(y) dy.
\]
Thus, we find using $\mu(y) \sim 1$ for $\alpha^{\ast}$ small and $\phi_4'(x) \geq 1/2$:
\begin{equation} \label{bi}
 B_1+B_2+B_3 \geq \frac{b}{8} \int_{ A \leq |y| \leq 2A} | \var |^2 dy.
\end{equation}
Moreover from (\ref{parameter}), 
\[
\left | \frac{r_s}{r^2} \int \phi_4( \frac{y}{A}) | \var |^2 \mu(y) dy \right | \leq \lambda \int | \tilde{u} |^2 \leq \Gamma_b^2. 
\] 
The estimate above and (\ref{bi}) conclude the proof of the lemma. 
\end{proof}

\begin{lem}[\textbf{Lyapunov functional}] \label{lya2}
 There exists a functional $\mathcal J$ and a constant $c >0$ such that for every $s \in [s_0, s_1)$,
\begin{equation} \label{lya}
 \partial_s \mathcal J \leq -c b (s) \left (  \Gamma_{b(s)} + \tilde{ \mathcal E}(s) + \int _{A \leq |y|\leq 2A} | \var|^2 \right ) ,
\end{equation}
and 
\begin{equation} \label{orderb2}
 | \mathcal J(s)-d_0 b^2 | \leq \delta ( \alpha^{\ast}) b^2 ,
\end{equation}
where $d_0>0$ is given by (\ref{supercriticalmass}).
Moreover, $\mathcal J$ is given by the following expression:
\begin{eqnarray*}
 \mathcal J (s) & =& -\frac{\delta_1}{16} \left ( b \tilde{f}_1 - \int _0^b \tilde{f}_1+ b  \mathrm{Im} (\var, \Lambda \tilde{\zeta}) \right ) +\frac{1}{r(s)} \int \left ( 1- \Phi_4( \frac{y}{A}) \right )  |\var |^2 \mu(y) dy \\
                          & & + \frac{1}{r(s)} \left (  \int |\tilde{Q}_b |^2- \int |Q|^2 +2 \mathrm{Re} \int \var \overline{ \tilde{Q}_b} \mu(y) dy \right ) ,
\end{eqnarray*}
where
\[
 \tilde{f}_1(s) =\frac{b}{4} \| y \tilde{Q}_b\|_{L^2}^2 + \frac{1}{2} \mathrm{Im} \left( \int y \partial_y \tilde{\zeta} \overline{\tilde{\zeta}} \right ).
\]
\end{lem}

\begin{proof}
We multiply (\ref{refined2}) by $\frac{\delta_1 b}{16}$ and sum with (\ref{dispersion}) to obtain
\begin{multline} \label{lya3}
\frac{\delta_1 b }{16} (f_1)_s + \left ( \frac{1}{r(s)} \int \phi_4 ( \frac{y}{A}) | \var |^2 \mu(y) dy \right )_s \geq \frac{\delta_1^2 b}{16} \tilde {\mathcal{E}} + \frac{C \delta_1}{16} b \Gamma_b - \Gamma_b^{\frac{a}{2}} \int | \partial_y \var |^2 \mu(y) dy \\ + \frac{b}{16} \int _{A \leq |y| \leq 2A} |\var |^2 - \Gamma_b^2. 
\end{multline}
We rearrange the left hand side to make appear the derivative of $\mathcal J$; this gives using the product rule:
\begin{equation} \label{product}
\left ( \frac{\delta_1}{16} \left ( b \tilde{f}_1 - \int_0^b \tilde{f}_1 + b  \mathrm{Im} (\var, \Lambda \tilde{\zeta}) \right ) + \frac{1}{r(s)} \int \phi_4 ( \frac{y}{A}) | \var |^2 \mu(y) dy \right )_s- \frac{\delta_1}{16} b_s \mathrm{Im} ( \var, \Lambda \tilde{\zeta}) .
\end{equation}
We use the mass conservation:
\[
 \int | \var |^2 \mu(y) dy + \int | \tilde{Q}_b |^2 \mu(y) dy + 2 \mathrm{Re} \int  \var \overline{\tilde{Q}}_b \mu(y) dy = \int | u_0 |^2 ,
\]
to write 
\begin{eqnarray*}
 \int \phi_4 ( \frac{y}{A}) | \var |^2 \mu(y) dy &= &\int |u_0|^2 - \int | \tilde{Q}_b|^2 \mu(y) dy - 2 \mathrm{Re} \int \var \overline{ \tilde{Q}_b} \mu(y) dy \\
   & & - \int \left ( 1- \phi_4(\frac{y}{A} ) \right )| \var |^2 \mu(y) dy;
\end{eqnarray*}
so that (\ref{product}) becomes
\[
 - \mathcal J_s - \frac{\delta_1}{16} b_s \mathrm{Im} ( \var, \Lambda \tilde{\zeta} ) - \left ( \frac{1}{r(s)} \left ( \int |u_0|^2 - \int | \tilde{Q}_b |^2 + \int |Q|^2 \right ) \right ) _s ,
\]
that we may rewrite according to (\ref{parameter}), \textbf{B3}, \textbf{B5}, and $\| \Lambda \tilde{\zeta} \|_{L^2} \leq \Gamma_b^{3/8}$ (see Lemma \ref{lemzeta}),
\[
 - \mathcal J_s + \mathcal O( \Gamma_b ^{9/8} ).
\]
In the right hand side of (\ref{lya3}), the only term that we need to control is 
\[
 \Gamma_b^{\frac{a}{2}} \int | \partial_y \var |^2 \mu(y) dy;
\]
that we may bound by (if $a \geq 4 C \eta$),
\begin{eqnarray*}
 \Gamma_b^{\frac{a}{2}} \int | \partial_y \var |^2 \mu(y) dy & \leq & \Gamma_b^{\frac{a}{2}} \left ( \int | \partial_y \tilde{\zeta} | ^2 + \int | \partial_y \tilde{\var} |^2 \mu(y) dy \right ) \\
                                                             & \leq & \Gamma_b^{\frac{a}{2}} \left (  \Gamma_b^{1- C \eta} + \int | \partial_y \tilde{\var} |^2 \mu(y) dy \right ) \\
                                                             & \leq & \Gamma_b^{1+\frac{a}{4}} + \Gamma_b^{\frac{a}{2}} \int | \partial_y \tilde{\var} |^2 \mu(y) dy.
\end{eqnarray*}
This shows (\ref{lya}). The last thing to do is to prove that $\mathcal J$ is of order $b^2$ (\ref{orderb2}). We use that $\tilde{Q}_b$ has a supercritical mass \ref{supercriticalmass} so that 
\begin{eqnarray*}
 \mathcal J(s)&=&d_0 b^2 + b^2 \delta( \alpha^{\ast }) + \frac{1}{r(s)} \left (  2 \mathrm{Re} \int \var \overline{   \tilde{Q}_b} \mu(y)dy +\int \left ( 1- \phi_4 ( \frac{y}{A}) \right ) | \var |^2 \mu(y) dy. \right ) \\
  & & - \frac{ \delta_1}{16} \left ( b \tilde{f}_1 - \int_0^b \tilde{f}_1 + b \mathrm{Im} ( \var, \Lambda \tilde{\zeta} ) \right ) .
\end{eqnarray*}

Now, we prove
\begin{eqnarray*}
 \int \left ( 1- \phi_4( \frac{y}{A} )\right ) | \var |^2 \mu(y) dy &\leq& \delta( \alpha^{\ast}) b^2.
\end{eqnarray*}
Indeed, let us denote by $\mathcal C$ the ring $\{x \in M, - 2 A \lambda +r(t) \leq r(x) \leq 2A \lambda + r(t) \}$ and $\phi$ a radial, smooth, positive cut-off such that $\phi(x)=1$ if $x \in \mathcal C $, $\phi(x) =0$ if $r(x) \geq 2A \lambda +(A \lambda )^{1/2} + r(t)$ or $ r(x) \leq -2A \lambda -(A \lambda )^{1/2}  +r(t)$ with the bound $\| \nabla \phi \|_{L^{\infty}} \leq C / (\lambda A)^{1/2}$. Then by Cauchy-Schwarz inequality, radial Sobolev embedding (see Proposition \ref{sobolevradial}) $H^{1/4} \hookrightarrow L^4  $ and interpolation, we have
\begin{eqnarray*}
 \int \left ( 1- \phi_4( \frac{y}{A}) \right ) | \var |^2 \mu(y) dy  \leq   \int_{|y| \leq 3A} | \var |^2 &\leq & \int_{\mathcal C} | \tilde{u} (x) |^2 dx  \\
                                                                  & \leq &  \mathrm{Vol} ( \mathcal C) ^{\frac{1}{2}} \| \phi \tilde{ u} \|_{L^4} ^2 \\
                                                                  & \leq & \left ( \int_{-2A \lambda + r(t)}^{2A \lambda + r(t)} h(r) dr  \right ) ^{\frac{1}{2}} \| \phi \tilde{u} \|_{H^{\frac{1}{4}}}^2  \\
                                                                 & \leq & C (A \lambda)^{\frac{1}{2}} \| \phi \tilde{u} \|_{L^2} ^{ \frac{3}{2}} \| \nabla ( \phi \tilde{u}) \|_{L^2} ^{\frac{1}{2}} \\
                                                                 & \leq &   C (A \lambda)^{\frac{1}{2}}   \left ( \| \nabla \phi \|_{L^{\infty}} ^{\frac{1}{2}} \| \tilde{u} \|_{L^2} ^{\frac{1}{2}} + \| \nabla \tilde{u} \|_{L^2}^{\frac{1}{2}} \right ).
\end{eqnarray*}
But, using on the one hand the bound on the derivative of $\phi$ and on the other hand 
\[
\| \nabla \tilde{u} \|_{L^2} = \frac{1}{\lambda} \left ( \int |\var|^2 \mu(y) dy \right ) ^{\frac{1}{2}} \leq \frac{1}{\lambda} \mathcal E ^{\frac{1}{2}} ,
\]
we obtain from \textbf{B3} and \textbf{B5} 
\[
\int \left ( 1- \phi_4( \frac{y}{A}) \right ) | \var |^2 \mu(y) dy  \leq C (A \lambda )^{\frac{1}{4}} + A^{\frac{1}{2} } \mathcal E \leq \delta( \alpha^{\ast}) b^2.
\]
Moreover, we have easily from (\ref{zetatilde3}):
\[
 |b \mathrm{Im} ( \var, \Lambda \tilde{\zeta})| +\left | \mathrm{Re} \int \var \overline{\tilde{Q}}_b \mu(y) dy \right  | \leq \delta(\alpha^{\ast}) b^2.
\]
The last term, namely 
\[
 \frac{\delta_1}{16} \left ( b \tilde{f}_1 - \int _0^b \tilde{f}_1 \right )
\]
is also less than $\delta(\alpha^{\ast})b^2$ using the fact that we can choose $\delta_1$ arbitrarily small in Lemma \ref{refined}.   
This proves Lemma \ref{lya2}.
\end{proof}
The following lemma is a consequence of the previous Lemmas; it gives in particular \textbf{C3} and the upper bound on $b$ which will allow us to prove the log log lower bound.
\begin{lem}
 There exists a constant $C>0$ such that for all $s \in [s_0, s_1)$,
\be \label{upperboundb}
 b(s) \leq \frac{4 \pi}{ 3 \log s} ,
\ee
\be \label{boundE}
 \mathcal E (s) \leq \Gamma_{b(s)}^{\frac{4}{5}},
\ee
\be \label{gammavar}
 \int _{s_0}^s \left ( \Gamma_{b ( \sigma) }+ \tilde{\mathcal E}  (\sigma)  \right ) d \sigma \leq C \alpha^{\ast}.
\ee
\end{lem}

\begin{proof}
 The proof of the upper bound on $b$ is a consequence of (\ref{lya}) and (\ref{orderb2}); indeed it imply:
\[
 \frac{d}{ds} e^{\frac{ 5 \pi}{4} \sqrt{ \frac{d_0}{ \mathcal J}}} =-\frac{5 \pi}{8} \sqrt{d_0} \mathcal J_s \mathcal J^{- \frac{3}{2} } e^{\frac{ 5 \pi}{4} \sqrt{ \frac{d_0}{ \mathcal J}}} \geq 1.
\]
We integrate the last inequality in time between $s$ and $s_0=e^{\frac{3 \pi}{4 b_0}}$ and since $\mathcal J(s_0) \sim d_0 b ^2 (s_0)$, we have
\[
 e^{\frac{ 5 \pi}{4} \sqrt{ \frac{d_0}{ \mathcal J}}} \geq s+( e^{\frac{ 5 \pi}{4} \sqrt{ \frac{d_0}{ \mathcal J(s_0)}}}-s_0) \geq s;
\]
or equivalently
\[
\sqrt{ \frac{\mathcal J }{d_0}} \leq \frac{5 \pi}{4} \frac{1}{ \log (s)}.
\]
We conclude using again (\ref{orderb2}) and this proves the upper bound on $b$. Let us prove (\ref{boundE}). We will need to have a more precise error in the approximation $ \mathcal J \sim d_0 b^2$. Here the measure $\mu(y)dy$ does not play an significant role so that the improvement of (\ref{orderb2}) is the same than the Euclidean case.  One can show that (see \cite{MerRap2006} for a detailed proof)
\[
  \frac{1}{C} \mathcal E - \Gamma_b^{1- Ca} \leq \mathcal J(s) - f_2(b(s)) \leq C A^2 \mathcal E  + \Gamma_b^{1- C a},
\]
where $f_2$ is the function defined by
\[
 f_2(b)=\left ( \int |\tilde{Q}_b | - \int Q^2 \right ) - \frac{\delta_1}{32} \left ( b \tilde{f}_1(b) - \int_0^b \tilde{f}_1(v) dv \right )
\]
and satisfying 
\[
 \frac{d f_2}{d b^2} (0)  >0.
\]
Let $s \in [s_0, s_1)$. If $b_s(s) \leq 0$ then by the virial estimate (\ref{virial2}), (\ref{boundE}) holds. If $b_s(s) >0$, then by continuity, $b$ increases on an interval $[s_2,s]$ for some $s_2 \geq s_0$ and
we choose for $s_2$ the smallest time for which $b$ increases on $[s_2,s]$. Using the decay of $\mathcal J$ and the growth of $f_2$ near $b=0$, we have
\begin{eqnarray*}
\frac{1}{C} \mathcal E (s) & \leq & \mathcal J(s) - f_2(b(s)) + \Gamma_{b(s)}^{1-  C a} \\
      & \leq & J(s_2)-f_2(b(s_2)) + \Gamma_{b(s)}^{1-  C a} \\
      & \leq & C A^2(s_2) \mathcal E(s_2)+ \Gamma_{b(s_2)}^{1- Ca} +\Gamma_{b(s)}^{1-  C a} \\
\end{eqnarray*}
But we will prove that
\be \label{decay}
\mathcal E(s_2) \leq \Gamma_{b(s_2)}^{\frac{6}{7}}.
\ee
Indeed, if $s_2=s_0$, then (\ref{decay}) holds from \textbf{A3} and if $s_2 >s_0$ then $b_s(s_2)=0$ and (\ref{decay}) follows from the virial estimate (\ref{virial2}). Finally, using $b(s_2) \leq b(s)$,
\begin{eqnarray*}
 \frac{1}{C} \mathcal E(s) & \leq & C A^2(s_2)  \Gamma_{b(s_2 )}^{\frac{6}{7}} + \Gamma_{b(s_2)}^{1- Ca} +\Gamma_{b(s)}^{1-  C a} \leq \Gamma_{b(s)}^{\frac{5}{6}}.
\end{eqnarray*}
This proves (\ref{boundE}). For (\ref{gammavar}), we divide by $\sqrt{\mathcal J}$ the inequality (\ref{lya}) and we integrate between $s_0$ and $s$; this yields in particular
\[
 c \int_{s_0}^s b  \left ( \Gamma_b + \tilde{\mathcal E}  \right ) \leq \sqrt{\mathcal J (s_0)}- \sqrt{\mathcal J (s)}.
\]
To conclude (\ref{gammavar}), we use $\mathcal J(s_0) \leq C b^2(s_0)$ and \textbf{A2}.
\end{proof}

\section{Smallness of the critical norm}  \label{sectionnorm}

\medskip

 We now prove that the critical norm of the solution is small outside the blow up curve $r \sim 1$. This is the last part to prove Proposition \ref{prop1}. This smallness has been proved in \cite{Rap2006} in the Euclidean case using Strichartz estimates and local smoothing properties of the flow. Indeed, in $\mathbb R^2$ or more generally in $\mathbb R^N$, it is a classical fact that the solutions of the linear Schr\"odinger equation have locally one half a derivative more than the initial data. This smoothing effect relies on the good dispersive behavior of the free equation in $\R^N$. In the case $a = \infty$, the manifold is non compact and we could use local smoothing properties for such manifolds. However, if $a < \infty$, the manifold is compact and this regularization is false because of lack of dispersion due to the discrete spectrum of the Laplace operator. The method we use to overcome this difficulty follows \cite{RapSze2009} where the use of an almost conserved quantity at level $H^2$ allows to get estimate on the $H^2$ norm of the solution and then propagate this information to have estimates on $H^s$ norms for $s \leq 2$ and so deduce the smallness for $s=1/2$. The counterpart of this method is that it requires more regularity on the solution and so the stability holds in $H_{\mathrm{rad}}^2$ and no longer in $H^1_{\mathrm{rad}}$ where it is more natural. 
 
\medskip

To prove the smallness of the solution in the area $|r-1| > 1/2$, we will split the problem into two parts: firstly, the smallness in the region $r \in [1/4,1/2] \cup [3/2,7/4]$ and secondly the smallness  for $|r-1| > 3/4$. In the first area, we will constantly use radial Sobolev inequalities reflecting the fact that in this area, the equation looks like the one dimensional $L^2$-critical equation. In the second area, the point is the introduction of the $H^2$ pseudo energy. 

\medskip

Let us recall that we have the bootstrap assumptions \textbf{B7}: for all $t \in [0,t_1)$

\begin{eqnarray}
\| u(t) \|_{H^{\frac{1}{2}}(|r-1| > \frac{1}{2})} &\leq& (\alpha^{\ast})^{\frac{1}{10}}, \label{assumption1/2}\\
\|u(t) \|_{H^{1} (|r-1| > \frac{1}{2})} & \leq &\frac{1}{\lambda (t) ^{3 \delta}} ,\label{assumption1}\\
\|u(t) \|_{H^{\frac{3}{2}} (|r-1| > \frac{1}{2})} & \leq &\frac{1}{\lambda (t)^{1 + 2 \delta}} ,\label{assumption3/2}\\
\|u(t) \|_{H^{2} (|r-1| > \frac{1}{2})} & \leq &\frac{1}{\lambda (t)^{2 + \delta}}. \label{assumption2}
\end{eqnarray}
and we want to prove the improved bounds $\textbf{C7}$: for all $t\in [0,t_1)$

\begin{eqnarray}
\| u(t) \|_{H^{\frac{1}{2}}(|r-1| > \frac{1}{2})} &\leq& (\alpha^{\ast})^{\frac{1}{5}}, \label{conclusion1/2}\\
\|u(t) \|_{H^{1} (|r-1| > \frac{1}{2})} & \leq &\frac{1}{2\lambda (t)^{3 \delta}} , \label{conclusion1}\\
\|u(t) \|_{H^{\frac{3}{2}} (|r-1| > \frac{1}{2})} & \leq &\frac{1}{2\lambda (t)^{1 + 2 \delta}} ,\label{conclusion3/2}\\
\|u(t) \|_{H^{2} (|r-1| > \frac{1}{2})} & \leq &\frac{1}{2\lambda (t)^{2 + \delta}}. \label{conclusion2}
\end{eqnarray}

 The main points to prove the $H^{1/2}$ smallness outside the curve is the $L^2H^1$ integrability of the solution outside $r \sim 1$ and an almost monotonicity property of the parameter $\lambda$. These two properties are stated in the following lemmas.
\begin{lem} The parameter $\lambda$ is almost decreasing in the sense that: 
\begin{equation} \label{monotonicity}
\textrm{for all} \  s_2, s_3\in [s_0, s_1) \ \textrm{with} \ s_2 < s_3,  \quad   \lambda (s_3) \leq 3 \lambda (s_2).
\end{equation}
\end{lem}

\begin{proof}
 This is a consequence of the first virial estimate (\ref{virial2}) and (\ref{parameter}); indeed, combining these two inequalities, we deduce:
\[
  \frac{\lambda_s}{\lambda} + b \leq C \Gamma_b^{1-C \eta} + b_s.
\]
Using the sign and the smallness of $b$ (\ref{boundb}) and \textbf{B2}; we find by integrating the last inequality
\begin{eqnarray*}
 \int _{s_2}^{s_3} \frac{\lambda_s}{\lambda} &\leq &b(s_3)-b(s_2) +C  \int_{s_2}^{s_3} \left (-b + \Gamma_b^{1-C \eta} \right )  \leq b(s_3)-b(s_2) -\frac{C}{2} \int_{s_2}^{s_3} b \leq 1.
\end{eqnarray*}
Computing the left hand side, we get (\ref{monotonicity}).
\end{proof}

Let us now state the second lemma. Let $\chi$ be a positive and smooth cut-off function such that $\chi=1$ on $[2/32,30/32]$ and $\chi=0$ outside $[1/32,31/32]$. Then the following lemma holds. 
\begin{lem} \label{smallnessL2H1}
There exists a constant $C >0$ such that for all $t \in [0,t_1)$,
\[
\int_0^{t} \| \chi  u \|_{H^1}^2 \leq C \alpha ^{\ast}.
\] 
\end{lem}

\begin{proof}
This lemma is a consequence of (\ref{gammavar}). Indeed, switching in polar coordonates and in $s$ variable, we obtain for $t<t_1$,
\begin{equation} \label{L2H1}
\int_0^{t}  \| \chi u \|_{H^1} ^2 \leq \int_0^{t} \|\tilde{u} \|_{H^1} ^2 \leq t \| \tilde{u} \|_{L^{\infty} L^2}^2 +  \int_{s_0}^{s} \int | \partial_y \var |^2 \mu(y) dy .
\end{equation}
But on the one hand, from the upper bound on $\lambda$ (\ref{estimatelambda}) and since $s_0$ can be made arbitrary large ,
\begin{equation} \label{finitetime}
t = \int_{s_0}^{s} \lambda ^2(\tau) d \tau \leq \lambda_0 \int_2 ^{\infty} e^{- \frac{2 \pi}{3} \frac{s} { \log s}} \leq \alpha^{\ast},
\end{equation}
and in particular $t_1 \leq \alpha ^{\ast} < \infty$. On the other hand, from (\ref{gammavar}), the second term in the right hand side of (\ref{L2H1}) is less than $C\alpha^{\ast}$. We conlude thanks to the $L^2$ boundedness of $\tilde{u}$.
\end{proof}

\medskip

Let us now introduce the $H^2$ pseudo-energy. An important fact, that we will prove later, is that this energy is controled from below by the $H^2$ norm and therefore it allows us to get estimates on this norm. 

\begin{lem}[\textbf{$H^2$ pseudo-energy}]
Let $E_2$ be the following functional
\[
E_2(u)=\int | \Delta u |^2 -3 \int | \nabla u|^2 |u|^4 -2 \mathrm{Re} \int ( \nabla u)^2 |u|^2 \overline{u}^2.
\]
then $E_2$ is an almost conserved quantity in the sense that there exists a constant $C>0$ such that for all $t \in [0,t_1)$,
\begin{equation} \label{almost}
\frac{d}{dt} E_2(u) \leq \frac{C}{{\lambda(t)}^{6 + \frac{3 \delta}{2}}}.
\end{equation}
\end{lem}

%\begin{equation} \label{almost}
%\frac{d}{dt} E_2(u) \leq C \left (  \int | \nabla u|^2 |u|^3 | \Delta u| + %\int | \nabla u |^2 |u|^8 \right ).
%\end{equation}
%\end{lem}

\begin{proof}
We first apply the operator $\nabla$ to the equation (\ref{nls}). Then we multiply by $\nabla \Delta \overline{u} + \nabla ( |u|^4 \overline{u})$, integrate and take the imaginary part to obtain
\[
0= \mathrm{Re} \int \nabla u_t \nabla \Delta \overline{u} + \mathrm{Re} \int \nabla u_t \nabla (|u|^4 \overline{u})=A+B.
\]
We integrate by part the $A$ term:
\[
A= - \mathrm{Re} \int u_t \Delta^2 \overline{u} = -\frac{1}{2} \frac{d}{dt} \int | \Delta u|^2.
\]
Using the identity 
\[
\nabla(|u|^4 \overline{u})= \nabla ( \overline{u}^3 u^2)=3 \overline{u}^2 \nabla \overline{u} u^2 + 2 u \nabla u  \overline{u}^3,
\]
we split $B$ into two terms $B_1$ and $B_2$ with
\[
B_1=3 \mathrm{Re} \int \nabla u_t \nabla \overline{u} |u|^4,
\]
that we rewrite as
\[
B_1=\frac{3}{2} \left (  \frac{d}{dt} \int | \nabla u|^2 |u|^4 -\int |\nabla u|^2 \frac{d}{dt}|u|^4 \right ).
\]
Using the expression of $u_t$ in (\ref{nls}), we get $\partial_t |u|^4= - 4 |u|^2 \mathrm{Im} (\overline{u} \Delta u)$ and thus
\[
B_1=\frac{3}{2} \frac{d}{dt} \int | \nabla u |^2 |u|^4 + 6 \int | \nabla u|^2 |u|^2 \mathrm{Im} ( \overline{u} \Delta u) .
\] 
With the same method, we obtain for $B_2$,
\[
B_2= \frac{d}{dt} \mathrm{Re}  \int ( \nabla u)^2 |u|^2 \overline{u}^2 + 3\mathrm{Re} \int (\nabla u)^2  \overline{u}^2 u (i \Delta \overline{u} + i |u|^4\overline{u} ) - \mathrm{Re} \int (\nabla u)^2  \overline{u}^3( i \Delta u +i |u|^4 u).
\]
Summing $A, B_1$ and $B_2$, we get the identity:
\begin{eqnarray*}
\frac{d}{dt} E_2(u(t)) &=&12  \int | \nabla u|^2 |u|^2 \mathrm{Im} (\overline{u} \Delta u)+ 6 \mathrm{Re} \int ( \nabla u )^2 \overline{u}^2 u(i \Delta \overline{u} + i |u|^4 \overline{u}) \\
  & & - 2\mathrm{Re} \int (\nabla u)^2 \overline{u}^3 ( i \Delta u + i u |u|^4),
\end{eqnarray*}
so that we deduce
\begin{equation} \label{almost2}
\frac{d}{dt} E_2(u) \leq C \left (  \int | \nabla u|^2 |u|^3 | \Delta u| + \int | \nabla u |^2 |u|^8 \right ).
\end{equation}

 Let us now estimate the right hand side of (\ref{almost2}). 
 
\medskip

\textbf{Outside poles.} We start by the region outside poles where we may use radial Sobolev embeddings of Proposition \ref{sobolevradial} as in a one dimensional setting. First we prove
\begin{equation} \label{outside}
  \int_{\frac{1}{4} < r< \frac{7}{4}}  | \nabla u |^2 |u|^3  |\Delta u| + \int_{\frac{1}{4} < r< \frac{7}{4}}  | \nabla u |^2 |u|^8 \leq \frac{C}{\lambda ^{6+ \frac{3\delta }{2} }}.
\end{equation}
Indeed, using Cauchy-Schwarz, radial Sobolev inequality, we have successively for the first term
\begin{eqnarray*}
  \int_{\frac{1}{4} < r< \frac{7}{4}}  | \nabla u |^2 |u|^3  |\Delta u| & \leq & \| \nabla u\|_{L^2} \| \Delta u \|_{L^2} \| \nabla u \| _{L^{\infty}( \frac{1}{4} < r < \frac{7}{4}} )\| u\| ^3_{L^{\infty}( \frac{1}{4} < r < \frac{7}{4})}
 \\
                                                                           & \leq & \| \nabla u\|_{L^2} \| \Delta u \|_{L^2} \|  \Delta u \|_{L^2}^{\frac{1}{2}} \|  \nabla u \| _{L^2}^{\frac{1}{2}} \| u \|^{\frac{3}{2}} _{L^2}   \| \nabla u \|^{\frac{3}{2}} _{L^2}.
\end{eqnarray*}
Using mass conservation, control of the $H^1$ norm of the solution $\| \nabla u \|_{L^2} \leq \lambda ^{-1}$ and the bootstrap assumption (\ref{assumption2}) on the $H^2$ norm, we get
\[
 \int_{\frac{1}{4} < r< \frac{7}{4}} | \nabla u |^2 |u|^3 | \Delta u|  \leq  \frac{1}{\lambda ^{ 6 +  \frac{3\delta }{2}}} .
\]
The second term is treated in the same way; we find
\[
 \int_{\frac{1}{4} < r< \frac{7}{4}}  | \nabla u |^2 |u|^8 \leq \| \nabla u\|_{L^2}^2 \| u\|_{L^2}^4  \|\nabla u \|_{L^2}^4 \leq  \frac{1}{ \lambda ^{6 }}.
\]
Summing the two above estimate, we obtain (\ref{outside}). 

\medskip

\textbf{Near poles.} We now prove that near poles, the estimate is better:
\begin{equation} \label{near}
  \int_{|r-1| > \frac{3}{4}} | \nabla u |^2 |u|^3  |\Delta u| + \int_{|r-1| > \frac{3}{4}}   | \nabla u |^2 |u|^8 \leq \frac{C}{\lambda ^{4+ 46 \delta }}.
\end{equation}
By Cauchy-Schwarz, 2D Sobolev embeddings $H^{1+ \delta} \hookrightarrow L^{\infty}$, $H^{\frac{1}{2} } \hookrightarrow L^4$ and interpolation between (\ref{assumption1})  and (\ref{assumption2}), the left hand side in (\ref{near}) is bounded by
\begin{eqnarray*}
 & &  \| \nabla u \|_{L^4(|r-1| > \frac{3}{4})}^2 \|\Delta u \|_{L^2} \| u\|^3_{L^{\infty}(|r-1| >\frac{3}{4})} + \| \nabla u \|_{L^{2}(|r-1| > \frac{3}{4}) }^2 \| u\|_{L^{\infty}(|r-1| > \frac{3}{4})}^8 \\
 & \leq & \frac{1}{ \lambda ^{2+ \delta}} \| u\|_{H^{\frac{3}{2}} (|r-1| > \frac{3}{4})}^2 \| u\| _{H^{1+ \delta}(|r-1| > \frac{3}{4})}^3 + \frac{1}{\lambda ^{6 \delta}} \| u\|_{H^{1+ \delta}(|r-1| > \frac{3}{4})}^8 \\
  & \leq & \frac{1}{\lambda ^{2+ \delta}} \frac{1}{\lambda ^{2+ 4 \delta}} \left ( \frac{1}{\lambda ^{3 \delta (1- \delta)}}    \frac{1}{\lambda^{\delta (2+ \delta)}} \right ) ^3 +\frac{1}{\lambda ^{6 \delta} } \left ( \frac{1}{\lambda ^{3 \delta (1- \delta)}}    \frac{1}{\lambda^{\delta (2+ \delta)}} \right )^8 \\
  & \leq & \frac{1}{ \lambda ^{4 + 46 \delta}}.
\end{eqnarray*}
Summing (\ref{outside}) and (\ref{near}), we obtain (\ref{almost}).
\end{proof}

Now we prove that the $E_2$ pseudo energy is essentially the square of the $H^2$ norm. Indeed, first near poles, again by interpolation
\begin{equation*}
\int_{|r-1| > \frac{3}{4}} | \nabla u |^2 |u|^4 \leq \| \nabla u \|_{L^2({|r-1| > \frac{3}{4}})}^2 \| u \|_{H^{1+\delta}({|r-1| > \frac{3}{4}} )}^4 \leq \frac{1}{\lambda^{C \delta}},
\end{equation*}
for some $C>0$. 
In the area outside poles but including the blow up curve, we have from the radial Sobolev embeddings 
\begin{equation}
\int_{\frac{1}{4} < r < \frac{7}{4}} | \nabla u |^2 |u|^4 \leq \| \nabla u \|_{L^2(\frac{1}{4} < r < \frac{7}{4})}^4 \| u\|_{L^2(\frac{1}{4} < r < \frac{7}{4})}^2  \leq \frac{1}{ \lambda ^4}.
\end{equation}
Therefore, summing the two above inequalities and from (\ref{almost}):
\begin{equation} \label{estimateH2}
\| u\| _{{\dot H}^2}^2 \leq \int | \Delta u |^2 \leq |E_2(u)|+ \frac{1}{ \lambda ^{4}} \leq |E_2(u_0)|+ \int_0^t \frac{d \tau}{ \lambda ( \tau)^{6+ \frac{3\delta}{2}}} + \frac{1}{ \lambda ^{4}}
\end{equation} 
The energy at time $t=0$ is controlled using the previous estimates on $E_2$, $\textbf{A7}$ and the monotonicity of $\lambda$ (\ref{monotonicity}):
\[
|E_2(u_0) | \leq \int | \Delta u_0 |^2 + \frac{1}{\lambda_0^4} \leq \frac{2}{{\lambda_0}^4} \leq \frac{C}{ \lambda ^4}.
\]
To treat the $L^1$ norm in time in the right hand side of (\ref{estimateH2}), we will need the following lemma which basically say that when integrating, $\lambda ^{-1}$ behaves like the self similar rate.

\begin{lem} \label{integrale}
Let $\alpha >0$. Then for all $t \in [0, t_1)$,
\begin{equation}
\int_0^{t} \frac{d \tau}{ {\lambda (\tau)}^{\alpha}} \leq \left \{
\begin{array}{ccl}
\alpha^{\ast} &\textrm{ if }& \alpha <2  \\
\frac{ | \log \lambda |^{14}}{\lambda ^{\alpha -2}} & \textrm{ if } &\alpha  \geq 2 \\
\end{array}
\right. .
\end{equation}
\end{lem}

\begin{proof}
If $\alpha <2$, we use the uniform smallness in $\alpha^{\ast}$ of $\lambda$, $\lambda \leq \mathrm{exp} (- 1/\alpha^{\ast})$ from \textbf{B5}, \textbf{B2} and the smallness in time from \textbf{B5} and (\ref{upperboundb}), $\lambda \leq \mathrm{exp}(- s ^{\frac{3}{40}})$ and this gives:
\[
\int_{0}^t \frac{d \tau}{\lambda ^{\alpha}} = \int_{s_0}^s \lambda^{2- \alpha} d \tau \leq e^{- \frac{2-\alpha}{2 \alpha^{\ast}}}  \int_{s_0}^s e^{- \frac{2- \alpha }{2} s^{\frac{3}{40}}} \leq \alpha^{\ast}.
\]
If $\alpha \geq 2$, we use the monotonicity of $\lambda$ (\ref{monotonicity}) and  $\lambda \leq \exp ( - s ^{\frac{3}{40}})$ to write
\[
\int_{0}^t \frac{d \tau}{\lambda ^{\alpha}} = \int_{s_0}^s \frac{d \tau }{  \lambda^{\alpha-2}}  \leq C \frac{s-s_0}{\lambda^{\alpha-2}} \leq \frac{| \log \lambda |^{14}}{\lambda ^{\alpha - 2 }}.
\]
This concludes the proof of the Lemma.
\end{proof} 

Thus, Lemma \ref{integrale} allows us to write:
\[
\| u\|_{\dot{H}^2} ^2\leq \frac{C}{\lambda ^{4}}+ \frac{ | \log \lambda |^{14}} {\lambda ^{4 + \frac{3 \delta}{2}}} + \frac{1}{\lambda ^4} \leq \frac{1}{2 \lambda ^{4+ 2 \delta}}.
\]
This proves (\ref{conclusion2}). 

\medskip

Now we prove (\ref{conclusion1/2}), (\ref{conclusion1}), \ref{conclusion3/2}). These estimates rely on the following lemma.
\begin{lem} \label{interior}
Let $0 <a_2 < a_1 < b_1< b_2$ and $\chi$ a positive smooth cut-off function such that $\chi=1$ on $[a_1,b_1]$ and $\chi=0$ outside $[a_2,b_2]$. Let $v=\chi u$ and $s \in (0,\frac{3}{2}]$. Then there exists a constant $C>0$ such that for all $t \in [0,t_1)$,
\begin{equation}
\| D^s v \|_{L^{\infty}_{[0,t)} L^2} \leq C \left (\| D ^s v(0) \|_{L^2} + \|  u \|_{L^2_{[0,t)} H^{\mathrm{max}(1,s+ \frac{1}{2})}(a_2,b_2)}+ \|D^s (v  |u|^4) \|_{{L^1_{[0,t)}} L^2 }\right ).
\end{equation}
where $D^s$ is the operator  $D^s=(I- \Delta)^{s/2}$.
\end{lem}

\begin{proof}
The function $v$ satisfies the equation
\[
i \partial_t v + \Delta v -u \Delta \chi - 2 \nabla u \cdot \nabla \chi = - v  |u|^4
\]
so that 
\[
\frac{1}{2} \frac{d}{dt} \| D^s v \|_{L^2}^2 =  \mathrm{Re} \int D^s \overline{v}  D^s v_t =\mathrm{Im} \int D^s  \overline{v} \ D^s (u \Delta \chi + 2 \nabla u \cdot \nabla \chi -v |u|^4 ) .
\] 
For the first two terms, if $s \geq \frac{1}{2}$, we write using an integration by parts in the second one and distributing the derivatives
\begin{eqnarray*}
\left |\int D^s  \overline{v} \ D^s (u \Delta \chi + 2 \nabla u \cdot \nabla \chi) \right | &\leq &\| D^s v \|_{L^2} \|  u \|_{H^s(a_2,b_2)} + \| D^{s+ \frac{1}{2}} v \|_{L^2} \|  u \|_{H^{s+ \frac{1}{2}}(a_2,b_2)} \\   & \leq & \|  u\|_{H^{s+\frac{1}{2}} (a_2,b_2)} ^2. 
\end{eqnarray*}
If $s < \frac{1}{2}$, we directly write
\begin{eqnarray*}
\left |\int D^s  \overline{v} \ D^s (u \Delta \chi + 2 \nabla u \cdot \nabla \chi) \right | &\leq& \| D^s v \|_{L^2} \|  u \|_{H^s(a_2,b_2)} +  \| D^s v\| _{L^2} \|  u\|_{H^{s+ \frac{1}{2}}(a_2,b_2)} \\
      & \leq & \| u\|_{H^1(a_2,b_2)}^2.
\end{eqnarray*}
For the third term, we have easily
\begin{eqnarray*}
\left | \int D^s \overline{v} D^s( v |u|^4) \right | & \leq  & \| D^s v \|_{L^2} \| D^s ( v |u|^4)\|_{L^2} \\ 
\end{eqnarray*}
Finally, summing the above estimates, we have
\[
 \frac{1}{2} \frac{d}{dt} \|D^s v \|_{L^2}^2 \leq  \|  u\|_{ H ^{\mathrm{max}(1, s + \frac{1}{2})}(a_2,b_2)}^2 + \| D^s v \|_{L^2} \| D^s (v |u|^4 )\|_{L^2}. 
\]
We integrate the above estimate and use a basic inequality for $a,b \geq 0$ and $\var>0$,
\[
ab \leq \var a^2 + \frac{1}{\var} b^2,
\] 
to get,
\[ \| D^s v\|_{L^{\infty} L^2} ^2  \leq \| D^s v (0)\|_{ L^2}^2 +
 \|  u\|_{L^2 H^{\mathrm{max}(1, s + \frac{1}{2})}(a_2,b_2)}^2  + \var  \| D^s v \|_{L^{\infty} L^2}^2 + \frac{1}{\var} \| D^s ( v |u|^4) \|_{L^1 L^2}^2.
\]
The result follows if we choose $\var$ small enough.

\end{proof}

\textbf{Bootstrap outside poles.} Here we prove the estimates (\ref{conclusion1/2}), (\ref{conclusion1}), (\ref{conclusion3/2}) in the regions $1/4 \leq r \leq 1/2$ and $3/2 \leq r \leq 7/4$. 

\medskip

\textbf{$H^{\frac{1}{2}}$ estimate.} To prove the $H^{\frac{1}{2}}$ smallness outside poles, we first prove the $H^{\sigma}$ smallness for all $\sigma <\frac{1}{2}$. We begin by the area $1/4 <r < 1/2$, the other region will be treated in the same way. We apply Lemma \ref{interior} with $s=\sigma$, $[a_1,b_1]=[\frac{4}{32}, \frac{28}{32}]$ and $[a_2,b_2]=[\frac{3}{32}, \frac{29}{32}]$ and $\chi_1=1$ on $[a_1,a_2]$, $\chi_1=0$ outside $[a_2,b_2]$ and $v_1=\chi_1 u$:
\[
\| D ^{\sigma} v_1 \|_{L^{\infty}L^2} \leq \| D^{\sigma} v_1(0) \|_{L^2} + \|  u \|_{L^2 H^1(\frac{3}{32}, \frac{29}{32})} + \| D^{\sigma} (v_1 |u|^4)\|_{L^1L^2}.
\]
For the linear term, we use the $L^2 H^1$ smallness of $\tilde{u}$ outside the blow up curve to get from Lemma \ref{smallnessL2H1}:
\[
\|u\|_{L^2 H^1(\frac{3}{32}, \frac{29}{32})} \leq \| \chi \tilde{u} \|_{L^2H^1} \leq C (\alpha^{\ast})^{\frac{1}{2}}.
\]
For the nonlinear term, we use fractional product rule to distrubute the $\sigma$-derivative on the product and H\"older inequality with exponents $(p,q)$ such that
\[
p=\frac{2}{1- 2 \sigma}, \qquad q=\frac{1}{\sigma}, \qquad \frac{1}{p}+ \frac{1}{q}=\frac{1}{2}.
\]
Let $\tilde{\chi}$ be a cut-off such that $\tilde{\chi}=1$ on $[3/32,29/32]$ and $\tilde{\chi}=0$ outside $[2/32,30/32]$. Then using $\tilde{\chi}^4 \chi_1=\chi_1$, we obtain
\begin{eqnarray*}
\| D^{\sigma} (v_1 |u|^4) \|_{L^2} & \leq &\| D^{\sigma} (v_1 |\tilde{\chi} \tilde{u}
|^4) \|_{L^2} \leq  \| D^{\sigma} v_1 \|_{L^2}  \|\tilde{\chi} \tilde{u} \|_{L^{\infty}}^4 + \| \tilde{\chi} \tilde{u}\|_{L^{\infty}}^3 \| D^{\sigma}  (\tilde{\chi} \tilde{u}) \|_{L^q}  \| v_1 \|_{L^p}.\\
\end{eqnarray*}
Next, we apply the radial Sobolev inequality
\[
\| \tilde{\chi} \tilde{u} \|_{L^{\infty}} \leq  \| \tilde{\chi} \tilde{u}\|_{L^2} ^{\frac{1}{2}} \| \nabla ( \tilde{\chi} \tilde{u}) \|_{L^2}^{\frac{1}{2}} \leq C\|  \tilde{u} \|_{H^1(\frac{2}{32},\frac{30}{32})}^{\frac{1}{2}} ,
\]
and the embeddings $\dot{H}^{\frac{1}{2}-\sigma} \hookrightarrow L^q$,  $\dot{H}^{\sigma} \hookrightarrow L^p$ which are allowed since $\sigma < \frac{1}{2}$ and this gives 
\begin{eqnarray*}
\| D^{\sigma} (v_1 |u|^4) \|_{L^2} & \leq & \| D^{\sigma} v_1 \|_{L^2}  \| \tilde{u} \|_{H^1(\frac{2}{32}, \frac{30}{32})}^2 + \| \tilde{u}\|_{H^1(\frac{2}{32}, \frac{30}{32})}^{\frac{3}{2}} \| D^{\frac{1}{2}} (\tilde{\chi}\tilde{u}) \|_{L^2}  \| D^{\sigma} v_1 \|_{L^2}  \\
    &\leq &\| D^{\sigma} v_1 \|_{L^2} \|\tilde{u} \|_{H^1(\frac{2}{32}, \frac{30}{32})}^2.\\
\end{eqnarray*}
Finally, we obtain the Gronwall type inequality
\begin{equation} \label{gronwall}
\| D^{\sigma} v_1 \|_{L^{\infty} L^2} \leq \| D^{\sigma} v_1 (0)\|_{L^2} + \|  \tilde{u} \|_{L^2 H^1(\frac{2}{32}, \frac{30}{32})} + \int_0^{t_1} \| \tilde{u} \|_{H^1(\frac{2}{32}, \frac{30}{32})}^2 \| D^{\sigma} v_1 \|_{L^2},
\end{equation}
that we may rewrite as
\[
\| D^{\sigma} v_1 \|_{L^{\infty} L^2} \leq \| D^{\sigma} v_1 (0)\|_{L^2} + \| \tilde{u} \|_{L^2 H^1(\frac{2}{32}, \frac{30}{32})} + \| D^{\sigma}v_1 \|_{L^{\infty} L^2} \|\tilde{u} \|_{L^2 H^1(\frac{2}{32}, \frac{30}{32})}^2.
\]
Since $\| \tilde{u} \|_{L^2 H^1(2/32,30/32)} \leq C (\alpha^{\ast})^{1/2}$ again from Lemma \ref{smallnessL2H1}, we deduce from \textbf{A8}:
\[
\| D^{\sigma} v_1 \|_{L^{\infty} L^2} \leq C \left ( \| D^{\sigma} v_1 (0)\|_{ L^2}+\|  \tilde{u} \|_{L^2 H^1(\frac{2}{32}, \frac{30}{32})}\right ) \leq C,
\]
Now, we can prove the $H^{\frac{1}{2}}$-smallness of $u$ outside the blow up curve. We repeat the proof of the $H^{\sigma}$ smallness for $\sigma<1/2$ but we treat the nonlinear term differently. We reduce the support of $\chi_1$ by introducing a smooth cut-off $\chi_2$ with $0 \leq \chi_2 \leq 1$ and 
\[
\chi_2 =
\left \{
\begin{array}{rcl}
1 &\textrm{on} &  [\frac{5}{32}, \frac{27}{32}] \\
0 & \textrm{outside} & [ \frac{4}{32}, \frac{28}{32}]
\end{array}
\right. .
\]
Let $v_2= \chi_2 u$, then we apply Lemma \ref{interior} with $\sigma=1/2$,
\[
\| D^{\frac{1}{2}} v_2 \|_{L^{\infty} L^2} \leq \| D^{\frac{1}{2}} v_2(0) \|_{L^2} + \| u\|_{L^2 H^1(\frac{4}{32}, \frac{28}{32})} + \| D^{\frac{1}{2}} (v_2 |u|^4) \|_{L^1 L^2}.
\]
The nonlinear term is again controled with H\"older inequality with exponents $p=q=4$ and radial Sobolev embeddings $H^{5/8}  \hookrightarrow L^{\infty}$ and $\dot{H}^{1/4} \hookrightarrow L^4$. Note that $\chi_2 \chi_1^4 = \chi_2$ so that
\begin{eqnarray*}
\| D^{\frac{1}{2}} (v_2 |u|^4) \|_{ L^2} &\leq & \| D^{\frac{1}{2}} (v_2  |\chi_1 u|^4) \|_{ L^2} \leq \| D^{\frac{1}{2}} v_2 \|_{L^2} \| \chi_1 u\|_{L^{\infty}}^4 + \| \chi_1 u\|_{L^{\infty}}^3 \| v_2\|_{L^4} \| D^{\frac{1}{2}} (\chi_1 u) \| _{L^4} \\
                                       & \leq & \| D^{\frac{1}{2}} v_2 \|_{L^2} \|  \tilde{u}\|_{H^1(\frac{3}{32}, \frac{29}{32})}^2 + \|  \chi_1 u \|_{H^{\frac{5}{8}}}^{ 3} \|D^{\frac{1}{4}} v_2 \|_{L^2} \| D^{\frac{3}{4} } ( \chi_1 u ) \|_{L^2}.
\end{eqnarray*}
By interpolation and using the smallness of the $H^{\sigma}$ norm of $v_1=\chi_1 u$ for $\sigma =1/8,3/8$, we have
\[
\| \chi_1 u \|_{H^{\frac{5}{8}} } ^3 \| D^{\frac{3}{4} } (\chi_1 u)  \|_{L^2} \leq \left (\| \chi_1  u \|_{H^{\frac{1}{8}}}^{\frac{8}{15}} \| \chi_1 u \|_{H^1}^{\frac{7}{15}} \right )^3  \| D ^{\frac{3}{8}} (\chi_1 u ) \|_{L^2}^{\frac{2}{5}}  \| \nabla (\chi_1 u ) \|_{L^2}^{\frac{3}{5}} \leq C \|  \tilde{u} \|_{H^1(\frac{3}{32}, \frac{29}{32})}^{2} ,
\]
so that the nonlinear term is now estimated by
\[
\| D^{\frac{1}{2}} (v_2 |u|^4) \|_{ L^2} \leq \| D^{\frac{1}{2}} v_2 \|_{L^2}   \|  \tilde{u} \|_{H^1(\frac{3}{32}, \frac{29}{32})} ^2.
\]
Thus, we obtain the same type of estimate as (\ref{gronwall}):
\[
\| D^{\frac{1}{2}} v_2 \|_{L^{\infty} L^2} \leq    \| D^{\frac{1}{2}} v_2 (0)\|_{ L^2}+\| \tilde{u} \|_{L^2 H^1(\frac{3}{32}, \frac{29}{32})}+ \int_0^{t_1} \|  \tilde{u} \|_{L^2 H^1(\frac{3}{32}, \frac{29}{32})}  \| D^{\frac{1}{2}} v_2 \|_{L^2},
\]
that we solve identically to find
\[
\| D^{\frac{1}{2}} v_2 \|_{L^2} \leq \| D^{\frac{1}{2}} v_2(0) \|_{L^2} + \| \nabla \tilde{u} \|_{L^2H^1(\frac{3}{32}, \frac{29}{32})} \leq (\alpha^{\ast})^{\frac{1}{5}}.
\]
In particular, this proves the bootstrap conclusion (\ref{conclusion1/2})  in $[\frac{1}{4}, \frac{1}{2}]$. Using the same treatment, the same estimate still holds in the region $[\frac{3}{2}, \frac{7}{4}]$. 

\medskip

 \textbf{$H^{\frac{3}{2}}$ estimate.} We reduce again the support of the cut-off functions with
\[
\chi_3 =
\left \{
\begin{array}{rcl}
1 &\textrm{on} &  [\frac{6}{32}, \frac{26}{32}] \\
0 & \textrm{outside} & [ \frac{5}{32}, \frac{27}{32}]
\end{array}
\right. .
\]
and we go back to Lemma \ref{interior} to estimate the $H^{\frac{3}{2}}$ norm of $v_3=\chi_3 u$:
\[
\| D^{\frac{3}{2}} v_3 \|_{L^{\infty} L^2} \leq  \| D^{\frac{3}{2}} v_3 (0) \|_{L^2} + \| D^2 u \|_{L^2 L^2} + \| D^{\frac{3}{2}} (v_3 |u|^4)\|_{L^1 L^2}.
\]
The linear term is known by (\ref{conclusion2}) and Lemma \ref{integrale}: 
\[
 \| D^2 u \|_{L^2L^2}\leq \left ( \int _0^{t_1} \frac{1}{\lambda ^{4 + 2 \delta}} \right )^{\frac{1}{2}} \leq  \frac{ |\log \lambda |^{7}}{\lambda ^{1+ \delta}}.
\]
For the nonlinear term, we write by distributing the derivative and since $\chi_2 ^4 \chi_3=\chi_3$,
\[
\| D^{\frac{3}{2}} (v_3 |u|^4) \|_{ L^2}  \leq \|   \chi_2 u \|_{H^\frac{3}{2}} \|\chi_2 u\|_{L^{\infty}}^4.
\]
The $L^{\infty}$-norm is bounded from the $H^{\frac{1}{2}}$ smallness of $v_2 =\chi_2 u$ proved before. Indeed, by interpolation
\begin{equation} \label{interpolation}
\|\chi_2 u\|_{L^{\infty}} \leq \| \chi_2 u\|_{H^{\frac{1}{2}+ \frac{\delta}{2}}} \leq \| \chi_2 u\| _{H^{\frac{1}{2}} }^{1- \delta} \| \nabla (\chi_2  u) \|_{L^2}^{\delta} \leq \frac{1}{ \lambda ^\delta}.
\end{equation}
Moreover, 
\[
\| \chi_2 u \|_{H^{\frac{3}{2}}} \leq \|  u \|_{H^1}^{\frac{1}{2}} \| u \|_{H^2}^{\frac{1}{2}} \leq \frac{1}{\lambda ^{\frac{3}{2}+ \frac{\delta}{2}}}.
\]
Therefore, we can conclude
\[
\| D^{\frac{3}{2}} v_3\|_{L^{\infty} L^2} \leq \| D^{\frac{3}{2}} v_3(0) \|_{L^2} + \frac{| \log \lambda |^{7}}{\lambda ^{1+ \delta}} +\int_0^{t_1} \frac{1}{ \lambda ^{\frac{3}{2}+ \frac{3 \delta}{2}}} \leq \frac{1}{2 \lambda ^{1+ 2 \delta}},
\]
and this proves the bootstrap conclusion (\ref{conclusion3/2}) in $[1/4, 1/2]$ and as before the same treatment gives the result in $[3/2,7/4]$. 

\medskip

 \textbf{$H^{1}$ estimate.} Let $\chi_4$ be a positive and smooth cut-off satisfying 
\[
\chi_4 =
\left \{
\begin{array}{rcl}
1 &\textrm{on} &  [\frac{7}{32}, \frac{25}{32}] \\
0 & \textrm{outside} & [ \frac{6}{32}, \frac{26}{32}]
\end{array}
\right. 
\]
and $v_4=u \chi_4$. Again with Lemma \ref{interior} , we have
\[
\| D v_4 \|_{L^{\infty} L^2} \leq \| D v_4(0)\|_{L^2} + \| u\|_{L^2 H^{\frac{3}{2}}(\frac{6}{32}, \frac{26}{32})} + \| D (v_4|u|^4) \|_{L^1 L^2}.
\]
Using the previous estimate on $v_3 = \chi_3 u$, we deduce
\[
\|  u \|_{L^2 H^{\frac{3}{2}} (\frac{6}{32}, \frac{26}{32})} \leq \|  \chi_3 u \|_{L^2 H^{\frac{3}{2}}} \leq C \left ( \int _0^{t_1} \frac{1}{ \lambda ^{2(1 + 2 \delta)}} \right )^{\frac{1}{2}} \leq \frac{| \log \lambda |^{7}}{\lambda ^{2 \delta}}.
\]
Injecting the $L^{\infty}$ bound (\ref{interpolation}) in the nonlinear term, we obtain
\[ 
\| D (v_4 |u|^4) \|_{L^1 L^2} \leq \left \| \|\chi_3  u\|_{L^{\infty}} ^4 \| \chi_3  u \|_{H^1 } \right \|_{L^1} \leq \int_0^{t_1} \frac{1}{\lambda ^{4 \delta +1}} \leq \alpha^{\ast}.
\]
We sum the above estimates and get the desired control on the $H^1$ norm:
\[
\|D  v_4\|_{L^{\infty} L^2} \leq \| D v_4(0) \|_{L^2} + \frac{| \log \lambda |^{7}}{\lambda ^{2 \delta}} +\alpha^{\ast} \leq \frac{1}{2\lambda ^{3 \delta}}.
\]
This proves (\ref{conclusion1}) on $[1/4, 1/2]$ and as before also in $[3/2,7/4]$. 

\medskip

\textbf{Bootstrap near poles.} Using again Lemma \ref{interior}, we have for $k=1,2,3$,
\begin{equation} \label{nearpoles}
\| D^{\frac{k}{2}} u\|_{L^2(|r-1| >\frac{3}{4})} \leq \alpha^{\ast}+ \|  u \|_{L^2 {H}^{\frac{k+1}{2}}(|r-1| >\frac{1}{2})} + \left \| \|  D^{\frac{k}{2} } u\|_{L^2(|r-1| >\frac{1}{2})} \| u \|^4_{L^{\infty}(|r-1| >\frac{1}{2})} \right \|_{L^1}
\end{equation}

We now inject the bootstrap assumptions (\ref{assumption1/2}), (\ref{assumption1}) and (\ref{assumption3/2}) into the estimate (\ref{nearpoles}) and this will give the bootstrap conclusions (\ref{conclusion1/2}), (\ref{conclusion1}) and (\ref{conclusion3/2}). Indeed, first from (\ref{assumption1/2}), (\ref{assumption1}), (\ref{assumption3/2}) and Lemma \ref{integrale},
\begin{equation} \label{101}
\| D^{\frac{k+1}{2}} u \|_{L^2 L^2(|r-1| >\frac{1}{2})} \leq \left \| \frac{1}{\lambda ^{k-1 + (4-k) \delta }} \right \|_{L^2} \leq \left \{
\begin{array}{ccl}
\alpha^{\ast} &\textrm{ if } & k=1, \\ 
\frac{|\log \lambda |^{7 }}{\lambda ^{k-2+(4-k) \delta}} & \textrm{ if } & k=2,3. \\
\end{array}
\right.
\end{equation}
Moreover, by interpolation between (\ref{assumption1}) and (\ref{assumption2}), we have for some $D>0$,
\[
\| u \|^4_{L^{\infty}(|r-1| >\frac{1}{2})} \leq \| u \|^4_{H^{1+ \delta}(|r-1| >\frac{1}{2})} \leq \frac{1}{ \lambda ^{D \delta}},
\]
and therefore the nonlinear term is bounded as follow:
\begin{equation} \label{100}
\left \| \| D^{\frac{k}{2}} u\|_{L^2(|r-1| >\frac{1}{2})} \| u \|^4_{L^{\infty}(|r-1| >\frac{1}{2})} \right \|_{L^1} \leq \left \|  \frac{1}{ \lambda^{1+ C \delta}} \right \|_{L^1} \leq \alpha^{\ast}.
\end{equation}
Summing (\ref{100}) and (\ref{101}), we obtain 
\[
\| D^{\frac{k}{2}} u \|_{L^2(|r-1| > \frac{3}{4})} \leq \left \{
\begin{array}{ccl}
({\alpha^{\ast}})^{\frac{1}{5}} & \textrm{ if } & k=1, \\
\frac{1}{2\lambda ^{k-2+(5-k) \delta}} & \textrm{ if } & k=2,3, \\
\end{array}
\right.
\]
and this proves (\ref{conclusion1/2}), (\ref{conclusion1}) and (\ref{conclusion3/2}) in the region $|r-1| > 3/4$. 

\medskip

This finishes the proof of the bootstrap conclusions (\ref{conclusion1/2}), (\ref{conclusion1}) and (\ref{conclusion3/2}). 

\section{Conclusion}
Thus, we have shown that \textrm{C1}-\textrm{C7} are true until $t_1$ and then we deduce $t_1=T$ and this proves Proposition \ref{prop1}. In particular, the solution blows up since by (\ref{finitetime}) $t_1 < \infty$. Let us now deduce from the $B_i$'s the theorem \ref{theoreme1}. 

\medskip

\textbf{Proof of the log log speed.} Let us begin by proving the log log blow up speed. This relies on the integration of the differential equivalences 
\[
 b_s \sim -e ^{- \frac{1}{b}}, \qquad -\frac{\lambda_s}{\lambda} \sim b.
\]
We first prove the existence of $C>0$ such that
\begin{equation} \label{equation12}
 \frac{1}{C} (T-t) \leq \lambda ^2(t) \log | \log \lambda (t) | \leq C (T-t).
\end{equation}
To prove this, it is sufficient to show 
\[
 \frac{1}{C} \leq -\frac{d}{dt} ( \lambda^2 \log | \log \lambda | ) \leq C.
\]
Derivating this expression, using the smallness of $\lambda$ as $\alpha^{\ast}$ goes to $0$ and the equality $ \lambda _s /\lambda = \lambda \lambda _t$, it is sufficient to prove
\[
 \frac{1}{C} \leq - \frac{\lambda_s}{\lambda} \log | \log \lambda | \leq C,
\]
or since $C^{-1} b\leq -\lambda_s / \lambda \leq C b$,
\begin{equation} \label{equation11}
 \frac{1}{C} \leq b \log | \log \lambda | \leq C,
\end{equation}
First, the lower bound is a consequence of \textbf{C5}. For the upper bound, integrating the inequality 
\[
   - \frac{\lambda_s}{\lambda} \leq 2b \leq \frac{C}{ \log (s)},
\]
between $s_0$ and $s$, we deduce for $s$ large
\[
| \log \lambda (s) | \leq | \log \lambda (s_0) | + \int_{s_0}^s \frac{C}{ \log (\tau)} d \tau \leq C \frac{s}{ \log (s)}.
\]
Taking the $\log$ and using again the equivalence $b \sim (\log s)^{-1}$, we obtain for $s$ large
\[
 \log | \log \lambda | \leq \log s - \log \log s \leq C \log s \leq \frac{C}{b(s)},
\]
and this prove the upper bound in (\ref{equation11}). Then, the log log speed is deduced from (\ref{equation12}) by remarking that the function $f: \lambda \mapsto \lambda^2 \log | \log \lambda |$ is monotone near $0$ and that 
\[
\frac{1}{C} (T-t) \leq f \left (  \left ( \frac{\log | \log (T-t)| }{T-t}  \right ) ^{\frac{1}{2}} \right ) \leq C (T-t).
\]

\textbf{Proof of the blow up on a curve.} Now we prove that the location of singularity is a curve. We verify that $r(t)$ has a limit when $t$ goes to $T$. Indeed, for $t_1,t_2 \in [0,T)$ with $t_1 < t_2$,  since from log log behavior of $\lambda$ proved straight above, $1/ \lambda$ is integrable near $T$ and from (\ref{parameter}) , we have
\[
 | r(t_1)-r(t_2) | \leq \int_{t_1}^{t_2}  \left |  \frac{ d r}{dt}   dt \right | \leq \int_{t_1}^{t_2} \frac{1}{ \lambda ^2} \left | \frac{d r}{ds} \right | dt  \leq \delta (\alpha ^{\ast}) \int_{t_1}^{t_2} \frac{dt}{ \lambda},
\]
and thus $r(t)$ satisfies the Cauchy criterion near $T$ and therefore has a limit $r(T) \in [1- \alpha ^{\ast} , 1+ \alpha ^{\ast} ] 
 \subset (0,a) $ as stated in 2) of Theorem \ref{theoreme1}. 
 
 \medskip

\textbf{Proof of the convergence of $u(t)$ to $u ^{\ast}$.} Let us prove the convergence of $\tilde{u}(t)$ to $u^{\ast}$ in $L^2(M)$. This convergence does not depend on the metric and for the sake of completeness, we give the proof which is similar to the Euclidean case \cite{Rap2006}. This is a consequence of two things: first outside the blow up curve, we have 
\[
\int _0 ^T \int_{ |r-r(T)| >R}  | u| ^6 \approx  \int _0 ^T \int_{ |r-r(T)| >R}  | \tilde{u}| ^6 < \infty. 
\] 
This will allow us to extend $\tilde{u}$ until the time $T$ in $L^2(| r-r(T)| >R)$. The second fact in that the limit $u^{\ast}$ of $\tilde{u }$ is in $L^2(M)$ and we have conservation of the mass at the limit i.e. $\| \tilde{u }(t)  \|_{L^2(M)} \to \| u^{\ast} \|_{L^2(M)}$ when $t$ goes to $T$. Equivalently, this means that the weak convergence in $L^2(M)$ is in fact strong. 

\medskip 

As in Lemma \ref{smallnessL2H1}, we may prove an $L^2H^1$ bound on $u$: for $R>0$ small enough
\[
\int_0^T \| \nabla u\|_{{L^2}( |r-r(T)| >R)} ^2 \leq C(R).
\]
Moreover, we have
\[
\| \tilde{u} \|_{H^{1/2}(|r-r(T)| >R)} \leq C(R).
\]
Indeed, we have already proved such a bound in the regions $|r-1| > 1/2$ (see bootstrap conclusion \textbf{B7}). Near the blow up curve, we may perform the same type of argument as in the proof of the $H^{1/2}$ smallness of $\tilde{u}$ outside poles in the section smallness of the critical norm. This allows to get the boundedness as close to the blow up curve as we want. 

\medskip

Now, let $R>0$ and $\chi$ be a smooth cut-off function such that $\chi (r)=0$ if $|r-r(T)| \leq R$ and $\chi(r)=1$ if $|r-r(T)| \geq 2R$. Let $t \in [0,T)$ and $s$ such that $t+s<T$. We set $v_s(t)=u(t+s)-u(t)$ and $\var >0$. We introduce $t_0$ such that for all $s>0$ small enough:
\[
 \int | v_s(t_0) |^2 \leq \var.
\]
Then
\begin{eqnarray*}
\frac{d}{dt} \| \chi ^6 v_s(t) \|_{L^2} ^2&=& - 2 \mathrm{Im} \int \chi^6 \overline{v_s(t)}  \left ( ( | u(t+s) |^4 u(t+s) - |u(t)|^4 u(t)) + (\Delta u(t+s)- \Delta u(t) )  \right ) \\
                                        & =& A+B.
\end{eqnarray*}
We bound the $A$ term as follow:
\begin{eqnarray*}
|A| & \leq & \int \chi^6 |u(t+s)- u(t)| \left | | u(t+s) |^4 u(t+s) - |u(t)|^4 u(t) \right | \\
         & \leq &  \int \chi ^6 (|u(t+s)|^6 + |u(t)|^6) .
\end{eqnarray*}
But for all $\tau \in [0,T)$, from Gagliardo-Niremberg inequality:
\begin{eqnarray*}
\int \chi^6 |u(\tau)|^6 &\leq & \| \chi u(\tau) \|_{H^{1/2}} ^4 \| \nabla (\chi u(\tau)  ) \|_{L^2}^2 .
\end{eqnarray*}
Therefore, with the $H^{1/2}$ boundedness and conservation of the $L^2$ norm, we get
\begin{eqnarray*}
|A| &\leq &C \left ( 1+ \| \nabla u(t+s)   \|_{L^2(|r-r(T)| > R )}^2 + \| \nabla u(t)   \|_{L^2(|r-r(T)| > R )}^2  \right ).
\end{eqnarray*}
By integration by parts, mass conservation and $ab \leq a^2 + b^2$, we have for the $B$ term,
\begin{eqnarray*}
|B| &\leq & \left |  \int ( \nabla u(t+s) - \nabla u(t) ) \left ( \nabla (\chi ^6 )\overline{v_s} + \chi ^6  \nabla \overline{v_s} \right )      \right | \\
   &\leq& C \left ( 1+ \| \nabla u(t+s) \|_{L^2(|r-r(T)| > R )}^2+   \| \nabla u(t) \|_{L^2(|r-r(T)| > R )}^2 \right ).
\end{eqnarray*}
Finally, we have
\[
\frac{d}{dt} \| \chi ^6 v_s(t) \|_{L^2} ^2 \leq C \left ( 1+ \| \nabla u(t+s) \|_{L^2(|r-r(T)| > R )}^2+  C \| \nabla u(t) \|_{L^2(|r-r(T)| > R )}^2 \right ).
\]
We integrate this inequality in time between $t_0$ and $t$ to deduce:
\begin{eqnarray*}
\| \chi ^6 v_s(t) \|_{L^2} ^2  & \leq & \int | v_s(t_0) |^2 + C (T-t_0) + C \int_{t_0}^T \| \nabla u(\tau) \|_{L^2(|r-r(T)| > R )}^2 \\
     & &  + C \int_{t_0 +s} ^T \| \nabla u(\tau) \|_{L^2(|r-r(T)| > R )}^2 \\
     & \leq &  2 \var ,
\end{eqnarray*}
if $t_0$ is close enough to $T$. Thus $u(t)$ satisfies the Cauchy criteria in the area $|r-r(T)| > 2 R$ so that it converges. Since 
\[
\frac{1}{ \sqrt{\lambda (t)}} \tilde{Q}_b \left ( \frac{ r-r(t)} { \lambda (t)} \right ) e ^{i \gamma (t)} \to 0 \quad \textrm{in} \ L^2( |r-r(t)| >2 R),
\]
$\tilde{u}(t)$ also converges to an element $u^{\ast} \in L^2 (|r-r(T)| > 2 R)$ for all $R>0$. In particular since $\tilde{u} (t)$ is uniformly bounded in $L^2$, we deduce $ u^{\ast} \in L^2(M)$ and $\tilde{u}(t)$ converges weakly to $u ^{\ast}$ in $L^2(M)$. Moreover let $R(t) = \lambda (t) A(t)$ where $A(t)$ is defined in (\ref{choicea} so that $R(t) \to 0$ as $t \to T$. Let also $\Phi$ be a smooth cut-off such that $\Phi (r) =1$ if $r \geq 1$ and $\Phi (r)=0$ if $r \leq 1/2$ and $0 \leq \Phi \leq 1$. Then for a fixed time $t$, after using the equation for $u$ and integration by parts,
\begin{eqnarray*}
\left | \frac{d}{d \tau} \int \Phi \left ( \frac{r-r(t)}{ R(t) } \right ) | u (\tau) |^2 \right | &  = & \left | \frac{2}{ R(t)}  \mathrm{Im} \int \nabla u( \tau) \overline{u}( \tau) \nabla \Phi  \left ( \frac{r-r(t)}{ R(t) } \right ) \right | \\
  & \leq & \frac{C}{ R(t)} \| \nabla u( \tau) \|_{L^2(M)}.
\end{eqnarray*}
We integrate this inequality between $t$ and $T$ to have
\begin{eqnarray} \label{convergenceustar}
\left | \int \Phi \left ( \frac{r-r(t)}{ R(t) } \right ) |u^{\ast} |^2 -\int \Phi \left ( \frac{r-r(t)}{ R(t) } \right ) | u (t) |^2 \right | & \leq & \frac{C}{ R(t)} \int_t^T \frac{d \tau }{\lambda( \tau)}.
\end{eqnarray}
But the right hand side of the above inequality goes to $0$. Indeed, we know the log log behavior of $\lambda$ and since 
\[
b(s) \geq \frac{C}{ \mathrm{log} (s)}, \qquad s = \int_0^t \frac{1}{ \lambda ^2}, 
\]
we deduce 
\[
b(t) \geq \frac{C}{ \mathrm{log} | \mathrm{log} (T-t) |},
\]
so that for some $\alpha >0$, $A(t) \geq | \mathrm{log} (T-t)|^{\alpha}$. Reinjecting this and the decay of $\lambda (t)$, we conclude to the convergence to $0$. But by Lebesgue theorem,
\[
\int \Phi \left ( \frac{r-r(t)}{ R(t) } \right ) |u^{\ast} |^2 \to \int | u^{\ast} |^2 \quad \textrm{as} \ t \to T,
\]
so that with (\ref{convergenceustar}),
\[
\int \Phi \left ( \frac{r-r(t)}{ R(t) } \right ) | u (t) |^2  \to \int | u^{\ast} |^2 \quad \textrm{as} \ t \to T,
\]
or equivalently since outside $u- \tilde{u}$ is negligeable outside the blow up curve,
\begin{equation} \label{convergencetou}
\int \Phi \left ( \frac{r-r(t)}{ R(t) } \right ) | \tilde{u} (t) |^2  \to \int | u^{\ast} |^2 \quad \textrm{as} \ t \to T.
\end{equation}
Now, we treat the part near the blow up curve. Writing $\tilde{u}$ in term of $\var$, we get
\begin{eqnarray*}
\int \left ( 1-  \Phi \left ( \frac{r-r(t)}{ R(t) } \right ) \right )  | \tilde{u}|^2 & \leq & \int _{ |r- r(t)| \leq R(t)} | \tilde{u}|^2 \\
                                                                                     & \leq & C \int _{|y| \leq A(t)} | \var (t,y) |^2 dy,
\end{eqnarray*}
where in the last line we used the boundedness of $\mu(y)$ for $|y| \leq A(t)$. But the right hand side may be controlled by $\mathcal E(t)$. Indeed, we use the inequality: for all $M \geq 1$,
\[
\int _{|y| \leq M} | \var (y) |^2 dy \leq C M^2 \left ( \int _{ |y| \leq 2M } | \partial_y \var |^2 dy + \int _{ |y| \leq 1} | \var (y) |^2 e^{- |y|} \right ).
\]
We refer to \cite{MerRap2004} for the proof. Therefore, 
\begin{eqnarray*}
\int \left ( 1- \Phi \left ( \frac{r-r(t)}{ R(t) } \right )  \right ) \ | \tilde{u}|^2 & \leq & C A^2(t) \left ( \int _{|y| \leq 2 A(t) } | \partial_y \var |^2dy  + \int  _{ |y| \leq 1} | \var (y) |^2 e^{-|y|} dy     \right ) \\
  & \leq &  e^{ \frac{a}{b(t)} } \mathcal E(t) \leq e^{ \frac{a}{b(t)} } \Gamma_b ^ {4/5} \to 0 .
\end{eqnarray*}
This together with (\ref{convergencetou}) gives the convergence  
\[
\int | \tilde{u} (t) |^2 \to \int | u^{\ast} |^2,
\]
and therefore we obtain the convergence of $\tilde{u}$ to $u^{\ast}$ in $L^2(M)$.

\medskip

\textbf{Proof of the convergence in the sense of measures.} Let $\phi$ be a continuous with compact support function on $M$. Then using polar coordonates and change of variable
\begin{eqnarray*}
\int  |u(t,x)|^2 \phi(x) dx  &=&  \int _{0} ^{2 \pi} \int |\tilde{Q}_b(y) |^2 \mu(y) \phi ( \lambda (t) y +r(t), \theta ) dy d \theta + \int  | \tilde{u} (t,x)|^2 \phi (x) dx  \\
                                          &  & + 2 \mathrm {Re} \left (  \int \int \tilde{ Q} _b(y) \overline{\var}(t, y) \phi(  \lambda (t) y + r(t) , \theta  ) \mu(y) dy d \theta  \right ) \\
                                          & =& A_1+A_2+A_3.
\end{eqnarray*}
Since $b(t)$ tends to $0$, $r(t)$ tends to $r(T)$, $\lambda (t) $ is much smaller than $b(t)$ and using the support of $\tilde{Q}_b$, by Lebesgue theorem, the first term tends to 
\[
 A_1 \to \| Q \| _{L^2}^2 h(r(t)) \int _{0}^{2 \pi} \phi(  r(T), \theta ) d \theta =  \| Q \| _{L^2}^2 \int \phi d \delta_{r(T)}.
\]
The second term converges to $\int |u ^{\ast} (x) | ^2 \phi(x) dx$. For the third term, since $\mathcal E(t)$ is exponentially small in $b$, by Cauchy-Schwarz
\[
 |A_3 | \leq C \mathcal E (t)  \frac{1}{ \sqrt b (t)} \to 0 \quad \textrm{as} \quad  t \to T.
\]
This proves the second point of Theoreme \ref{theoreme1} and the proof is finished.

\bigskip

\noindent \textbf{Acknowledgements.} This work is a part of the PhD thesis of the author under the direction of Nikolay Tzvetkov that I thank for his help and the rereading of the text.

\nocite{HolRou2010}

\bibliographystyle{plain}
\bibliography{bibliography}

\end{document}